\documentclass[11pt,reqno]{amsart}
\usepackage{bm,amsmath,amsthm,amssymb,mathtools,verbatim,amsfonts,tikz-cd,mathrsfs}
\usepackage{enumitem}
\usepackage{float}

\usepackage[hidelinks,colorlinks=true,linkcolor=blue, citecolor=black,linktocpage=true]{hyperref}
\usepackage{graphicx}
\usepackage[alphabetic]{amsrefs}
\usepackage{caption}
\usepackage{morefloats}
\usepackage[top=1.4in, bottom=1.2in, left=1.4in, right=1.4in,marginpar=1in]{geometry}
\usepackage{subfigure}
\usepackage{adjustbox}

\usetikzlibrary{matrix,arrows,decorations.pathmorphing}

\usepackage{chngcntr}
\counterwithin{figure}{section}

\newtheorem{thm}{Theorem}[section]
\newtheorem{prop}[thm]{Proposition}

\newtheorem{cor}[thm]{Corollary}
\newtheorem{lem}[thm]{Lemma}

\theoremstyle{definition}
\newtheorem{define}[thm]{Definition}

\theoremstyle{remark}
\newtheorem{rem}[thm]{Remark}

\newcommand{\ve}[1]{\boldsymbol{\mathbf{#1}}}
\newcommand{\veit}[1]{\boldsymbol{\mathit{#1}}}
\newcommand{\R}{\mathbb{R}}

\newcommand{\Z}{\mathbb{Z}}

\newcommand{\Q}{\mathbb{Q}}

\renewcommand{\d}{\partial}
\renewcommand{\subset}{\subseteq}

\renewcommand{\bar}{\overline}

\newcommand{\iso}{\cong}

\DeclareMathOperator{\Aff}{{Aff}}

\DeclareMathOperator{\gr}{{gr}}

\DeclareMathOperator{\Hom}{{Hom}}

\DeclareMathOperator{\id}{{id}}

\DeclareMathOperator{\Int}{{int}}

\DeclareMathOperator{\rank}{{rank}}

\DeclareMathOperator{\Span}{{Span}}

\DeclareMathOperator{\sgn}{{sgn}}

\newcommand{\lk}{\mathrm{lk}}

\newcommand{\bF}{\mathbb{F}}

\newcommand{\bI}{\mathbb{I}}

\newcommand{\bT}{\mathbb{T}}

\newcommand{\cA}{\mathcal{A}}
\newcommand{\cB}{\mathcal{B}}
\newcommand{\cC}{\mathcal{C}}

\newcommand{\cF}{\mathcal{F}}
\newcommand{\cG}{\mathcal{G}}

\newcommand{\cK}{\mathcal{K}}

\newcommand{\cO}{\mathcal{O}}

\newcommand{\cS}{\mathcal{S}}
\newcommand{\cT}{\mathcal{T}}

\newcommand{\cX}{\mathcal{X}}
\newcommand{\cY}{\mathcal{Y}}

\newcommand{\frb}{\mathfrak{b}}

\newcommand{\scE}{\mathscr{E}}

\newcommand{\scL}{\mathscr{L}}

\newcommand{\scT}{\mathscr{T}}

\newcommand{\HF}{\mathit{HF}}

\newcommand{\xs}{\ve{x}}
\newcommand{\ys}{\ve{y}}
\newcommand{\zs}{\ve{z}}
\newcommand{\ws}{\ve{w}}

\renewcommand{\a}{\alpha}
\renewcommand{\b}{\beta}

\newcommand{\veps}{\varepsilon}

\usepackage{leftidx}

\DeclareMathOperator{\Cone}{{Cone}}

\newcommand{\Sss}[1]{\scriptscriptstyle{#1}}

\numberwithin{equation}{section}

\newcommand{\ar}{\mathrm{a.r.}}

\newcommand{\alg}{\mathrm{alg}}

\allowdisplaybreaks

\newcommand{\cCo}{\cC\! o}
\newcommand{\cTr}{\cT\! r}

\DeclareMathOperator{\opp}{opp}
\DeclareMathOperator{\val}{val}
\DeclareMathOperator{\Tw}{Tw}
\DeclareMathOperator{\Heis}{Heis}
\newcommand{\rmT}{\mathrm{T}}

\DeclareMathOperator{\wt}{wt}

\newcommand{\MOD}{\mathsf{Mod}}

\title{Koszul duality and the link surgery formula}

\author{Ian Zemke}
\address{Department of Mathematics\\University of Oregon\\  Eugene, OR, USA}
\email{izemke@uoregen.edu}
                        
\begin{document}
\maketitle
\begin{abstract} In previous works, the author described an associative algebra whose $A_\infty$-module categories encode the Heegaard Floer Dehn surgery formulas. In this article, we describe the Koszul dual of this algebra. We construct dualizing bimodules, and prove several equivalences of categories. The constructions of this paper have applications to computational problems involving the link surgery formula.
\end{abstract}

\tableofcontents

\section{Introduction}

The link surgery formula of Manolescu and Ozsv\'{a}th \cite{MOIntegerSurgery} is a powerful tool for computing the Heegaard Floer homology of 3-manifolds  \cite{OSDisks}. This theory is an extension of Ozsv\'{a}th and Szab\'{o}'s mapping cone formula for null-homologous and rationally null-homologous knots \cites{OSIntegerSurgeries,OSRationalSurgeries}. To a link $L\subset S^3$ with integral framing $\Lambda$, Manolescu and Ozsv\'{a}th \cite{MOIntegerSurgery} describe the \emph{link surgery complex}, denoted  $\cC_{\Lambda}(L)$, which computes the Heegaard Floer homology of $S^3_{\Lambda}(L)$. The author extended this theory to all longitudinally framed links in closed, oriented 3-manifolds in \cite{ZemExact}.

In \cite{MOT_Grid_Diagrams}, Manolescu, Ozsv\'{a}th and Thurston use the link surgery formula to give a combinatorial description of the Heegaard Floer homology of any 3-manifold. Their strategy is to present a closed 3-manifold $Y$ as integral surgery on a link $L\subset S^3$ and algorithmically compute the link surgery complex using grid diagrams.  The complexity of this algorithm makes it impractical for most non-trivial computations. (For example, the grid complex of a link has $n!$ generators, where $n$ is the size of the grid). Nonetheless, it illustrates the power of the link surgery formula, and motivates the study of algebraic techniques to understand and manipulate the link surgery formula.

 In \cite{ZemBordered,ZemExact}, the author gave another perspective on the Manolescu--Ozsv\'{a}th--Szab\'{o} surgery formulas. Therein, we reinterpreted the surgery formulas using the language of Lipshitz, Ozsv\'{a}th and Thurston's  bordered Floer theory \cite{LOTBordered}. We described how to view the surgery formulas as a theory for compact 3-manifolds with parametrized torus boundary components. In \cite{ZemBordered,ZemExact}, we described an associative algebra $\cK$, and reformulated the surgery formula in terms of certain categories of $A_\infty$-modules over this algebra.
 
 In the present paper, we develop an algebraic tool which can help facilitate certain computations using the surgery formula: Koszul duality. Herein, we describe a curved $dg$-algebra $\cK^!$ which is naturally dual to $\cK$. We prove that certain categories of $A_\infty$-modules and twisted complexes of modules over these algebras are equivalent. The algebra $\cK^!$ will be useful to simplify certain computations related to bimodules over the algebras $\cK$. As a particular example, we explain in Section~\ref{sec:CZZ} how the description of the bimodules for L-space satellite operators in \cite{CZZSatellites} has a natural interpretation in terms of the Koszul dual algebra $\cK^!$.

 \subsection{Surgery modules}

The author described in \cite{ZemBordered} an associative algebra $\cK$, which we refer to as the \emph{surgery algebra}. It is an algebra over the idempotent ring $\ve{I}=\ve{I}_0\oplus \ve{I}_1$, where each $\ve{I}_{\veps}$ is a copy of $\bF:=\Z/2$. We recall the definition of the algebra. We set 
\[
\ve{I}_0\cdot \cK\cdot \ve{I}_0=\bF[W,Z],\quad  \ve{I}_1\cdot \cK\cdot \ve{I}_0=\bF[U,T,T^{-1}]\otimes \langle \sigma,\tau\rangle,
\]
\[
\ve{I}_0\cdot \cK \cdot \ve{I}_1=0,\quad \text{and} \quad \ve{I}_1\cdot \cK \cdot \ve{I}_1=\bF[U,T,T^{-1}].
\]
The only non-obvious relations are that
\[
\sigma W=UT^{-1} \sigma, \quad \sigma Z= T\sigma, \quad \tau W=T^{-1} \tau, \quad \tau Z=UT \tau.
\]

 We use Lipshitz, Ozsv\'{a}th and Thurston's formalism of type-$D$ and type-$A$ modules \cite{LOTBordered} to encode the invariants of longitudinally framed knots in closed 3-manifolds. To a knot $K$ in a closed, oriented 3-manifold $Y$, equipped with Morse (i.e. longitudinal) framing $\lambda$, the author described a \emph{type-$D$ module}
\[
\cX_\lambda(Y,K)^{\cK}
\]
as well as a \emph{type-$A$ module}
\[
{}_{\cK} \cX_{\lambda}(Y,K).
\]
A type-$A$ module is the same thing as an $A_\infty$-module. A type-$D$ module is a certain type of $dg$-module, which the reader can think of as a projective $dg$-module.

$A_\infty$-\emph{bi}modules play an important role in applications of the theory. To a 2-component link $L\subset Y$ with longitudinal framing $\Lambda$, the theory in \cite{ZemBordered} assigns a $DA$-bimodule
\[
{}_{\cK} \cX_{\Lambda}(Y,L)^{\cK}.
\]
In \cite{CZZSatellites}, Chen, Zhou and the author study these bimodules for 2-component L-space links in $S^3$. We used these bimodules to compute the effect on knot Floer complexes of certain satellite operators. We plan to use the algebra $\cK^!$ to help study $DA$-bimodules of the form ${}_{\cK} \cX_{\Lambda}(Y,L)^{\cK}$ for more general families of links.

\subsection{Koszul duality}

Koszul duality is an important topic in many areas of algebra. Classically, Koszul duality is simplest to describe for \emph{quadratic algebras}.  We describe this briefly for algebras over a field $\bF$. A quadratic algebra $\cA$ is one which is described by taking an $\bF$-vector space $V$ of generators, and quotienting the tensor algebra $T^* V:=\bigoplus_{i\ge 0} V^{\otimes i}$ by an ideal $(R)$ where $R\subset V\otimes V$.
When $A$ satisfies several additional conditions (i.e. is \emph{Koszul}), the Koszul dual algebra is described by taking the dual set of generators $V^*=\Hom(V,\bF)$, and quotienting the tensor algebra $T^* V^*$ by the ideal $(R^*)$ where $R^*\subset V^*\otimes V^*$ is the annihilator of $R$. See \cite{PP_Quadratic} for background on this perspective. The classic example is the Koszul duality between the polynomial ring $\bF[X_1,\dots, X_n]$ and the exterior algebra $\Lambda^*(\theta_1,\dots, \theta_n)$.

Koszul duality can be formulated for algebras which are not Koszul, but the above perspective requires generalization. We take the following as the definition of Koszul duality in this paper:

\begin{define} We say that $dg$-algebras $\cA$ and $\cB$ over a ring $\ve{k}$ are \emph{Koszul dual} if there is a $DD$-bimodule ${}^{\cA} [\cCo]^{\cB}$ and an $AA$-bimodule ${}_{\cB} [\cTr]_{\cA}$, both of which have underlying $(\ve{k},\ve{k})$-bimodule equal to $\ve{k}$, such that
\[
{}^{\cA}[\cCo]^{\cB}\boxtimes {}_{\cB} [\cTr]_{\cA}\simeq {}^{\cA} \bI_{\cA}\quad \text{and} \quad  {}_{\cB} [\cTr]_{\cA}\boxtimes {}^{\cA}[\cCo]^{\cB}\simeq {}_{\cB} \bI^{\cB}.
\] 
Here, $\simeq$ denotes homotopy equivalence and ${}^{\cA} \bI_{\cA}$ and ${}_{\cB} \bI^{\cB}$ denote the identity bimodules (which have underlying space $\ve{k}$, and structure maps $\delta_2^1(1,a)=a\otimes 1$ and $\delta_2^1(b,1)=1\otimes b$, respectively). Also, $\boxtimes$ denotes a model of the derived tensor product, described in \cite{LOTBordered}.
\end{define}

The above definition is slightly non-standard in that it is phrased in the language of Lipshitz, Ozsv\'{a}th and Thurston's type-$D$ and type-$A$ modules, as well as their box tensor product, though this is the most convenient setting for our purposes.

We refer to ${}_{\cK}[\cTr]_{\cK^!}$ and ${}^{\cK^!}[\cCo]^{\cK}$ as the \emph{trace and cotrace} bimodules. This terminology is not standard, but we feel that it is descriptive due to its analogy with the trace and cotrace maps of vector spaces.

\begin{rem}Given a type-$A$ module ${}_{\cB} M$ one can tensor ${}_{\cB} M$ with the cotrace module ${}^{\cA} [\cCo]^\cB$ to obtain a type-$D$ module ${}^{\cA} M$. Under suitable hypotheses on boundedness, these two modules contain the same information, since we can tensor with ${}_{\cB} [\cTr]_{\cA}$ to recover the original module ${}_{\cA} M$, up to homotopy equivalence. In practice, the type-$D$ module ${}^{\cA} M$ typically requires less information to encode, since we need to encode just a single map $\delta^1\colon M\to \cA \otimes M$, as opposed to an infinite family of maps $m_{n+1} \colon \cB^{\otimes n}\otimes M\to M$. This makes Koszul duality a helpful tool for computations.
\end{rem}

\subsection{The Koszul dual algebra}

In this paper, we describe a curved $dg$-algebra $\cK^!$ which is Koszul dual to $\cK$. We set
\[
\ve{I}_0\cdot \cK^!\cdot \ve{I}_0=\frac{\bF\langle w,z,\theta\rangle}{w^2=z^2=[\theta,w]=[\theta,z]=0}.
\]
Here, $\bF$ denotes $\Z/2$ and $\bF\langle w,z,\theta \rangle$ denotes the free non-commutative algebra on the elements $w$, $z$ and $\theta$. Also, $[a,b]$ denotes the commutator $ab-ba$. Examples of algebra elements are
\[
1,\quad w,\quad wzwzw, \quad zwz \theta.
\]

We set 
\[
\ve{I}_1\cdot \cK^! \cdot \ve{I}_1=\frac{\bF\langle  \varphi_+,\varphi_-,\theta\rangle}{\varphi_{\pm}^2=[\theta,\varphi_{\pm}]= \theta^2=0}.
\]
Examples of algebra elements are
\[
1, \quad \varphi_+\varphi_-,\quad \varphi_-\varphi_+\varphi_-\theta.
\]

We define $\ve{I}_1\cdot \cK^!\cdot \ve{I}_0=0$. We define $\ve{I}_0\cdot \cK^!\cdot \ve{I}_1$ to be spanned as a bimodule over $\ve{I}_0 \cdot \cK^! \cdot \ve{I}_0$ and $\ve{I}_1 \cdot \cK^!\cdot \ve{I}_1$ by two algebra elements, labeled $t$ and $s$. Multiplication with these elements is given by concatenation, modulo the relations
\[
w s=zt=s \varphi_-=t\varphi_+=[s,\theta]=[t,\theta]=0
\]
and
\[
z s=s \varphi_+,\quad wt=t\varphi_-.
\]
With these relations, $\ve{I}_0\cdot \cK^! \cdot \ve{I}_1$ has rank 8 over $\bF$, and is spanned by the elements
\[
s, \quad s \varphi_+,\quad s \theta,\quad s \varphi_+ \theta, \quad t,\quad t \varphi_-,\quad t \theta,\,\,\, \text{and}\,\,\,  t  \varphi_-\theta.
\]

The algebra $\cK^!$ has a differential, denoted $\mu_1\colon \cK^!\to \cK^!$. We declare
\[
\mu_1(\theta)=wz+zw,
\]
and extend this to all algebra elements by the Leibniz rule.
There is also a curvature term
\[
\mu_0=\varphi_+\varphi_-+\varphi_-\varphi_+,
\] 
which is central in $\cK^!$. 
We declare $\cK^!$ to have no higher composition maps (so in particular it is a curved $dg$-algebra, as opposed to a curved $A_\infty$-algebra).

We will describe two bimodules
\[
{}_{\cK} [\cTr]_{\cK^!}\quad \text{and} \quad {}^{\cK^!} [\cCo]^{\cK},
\]
which have underlying $(\ve{I},\ve{I})$-module equal to $\ve{I}$. We prove the following, which we view as the most basic statement of Koszul duality between the two algebras:

\begin{thm}
\label{thm:equality-of-box-tensor-products} The bimodules ${}_{\cK} [\cTr]_{\cK^!}$ and  ${}^{\cK^!} [\cCo]^{\cK}$ satisfy
\[
{}_{\cK} [\cTr]_{\cK^!}\boxtimes {}^{\cK^!}[\cCo]^{\cK}={}_{\cK} \bI^{\cK},\quad \text{and} \quad {}^{\cK^!} [\cCo]^{\cK}\boxtimes {}_{\cK} [\cTr]_{\cK^!}={}^{\cK^!} \bI_{\cK^!}. 
\]
\end{thm}

\subsection{Equivalence of categories}

In Section~\ref{sec:equivalence-of-categories}, we define two module categories
${}^{\cK^!} \MOD_{(U),\frb}$ and ${}_{\cK} \MOD_{(U),\frb}$, which are type-$D$ and $A$ modules which satisfy certain requirements with respect to completions and boundedness. We prove the following:

\begin{thm} The categories ${}^{\cK^!} \MOD_{(U),\frb}$ and ${}_{\cK} \MOD_{(U),\frb}$ are equivalent. 
\end{thm}

We also prove extensions of the above theorem for $DA$ and $DD$ bimodules over $(\cK,\cK)$ and $(\cK^!, \cK)$, respectively.

\begin{rem} In principle, one would expect there to be equivalences between the categories of right type-$D$ modules over $\cK$, denoted $\MOD^{\cK}$, and the categories of right $A_\infty$-modules over $\cK^!$, denoted $\MOD_{\cK^!}$. However, the category of $A_\infty$-modules over a curved $A_\infty$-algebra (or curved $dg$-algebra) is somewhat problematic if not defined carefully. See \cite{Positelski_Memoirs}*{Remark~7.3}. Typically, one must impose extra restrictions on $A_\infty$-modules and morphisms, such as requiring certain boundedness assumptions to be satisfied, in order to obtain a meaningful homotopy category of modules. Since we are not aware of any practical advantages of the category of $A_\infty$-modules over $\cK^!$ over the category of type-$D$ modules over $\cK$, we do not explore this.
\end{rem}

\subsection{Computational applications}

Our main motivation for understanding the Koszul dual algebra $\cK^!$ lies in performing computations involving bimodules. For many topological applications, such as satellite operations, computations involving bimodules play an important role. However, calculations involving a $DA$-bimodules ${}_{\cK} \cX^{\cK}$ are frequently challenging because $\cK$ is infinite dimensional over $\bF$.  Even recording formulas for the structure maps of a bimodule can be challenging. Koszul duality gives a very practical technique for working with bimodules.

 If ${}_{\cK} \cX^{\cK}$ is a $DA$-bimodule, write ${}^{\cK^!} \cX^{\cK}$ for the $DD$-bimodule 
\[
{}^{\cK^!} \cX^{\cK}:={}^{\cK^!} [\cCo]^{\cK}\boxtimes {}_{\cK} \cX^{\cK}.
\]
In practice, the $DD$-bimodule ${}^{\cK^!} \cX^{\cK}$ is typically easier to work with than the $DA$-bimodule ${}_{\cK} \cX^{\cK}$. In favorable circumstances (e.g. when ${}_{\cK} \cX^{\cK}$ is a bimodule of a 2-component link in $S^3$), the data of the bimodule ${}^{\cK^!} \cX^{\cK}$ can be easily encoded in a finite (and frequently small) set of data. The original $DA$-bimodule ${}_{\cK} \cX^{\cK}$ can be recovered, up to homotopy equivalence, by tensoring on the left with ${}_{\cK} [\cTr]_{\cK^!}$. 

The bimodule ${}_{\cK} [\cTr]_{\cK^!}$ has many higher actions, though we give a combinatorial description which is concrete and can be easily computed by hand for relatively small numbers of inputs (or on a computer for larger numbers of inputs).

The techniques from \cite{CZZSatellites} by Chen, Zhou and the author on L-space satellite operators have a natural interpretation in terms of the Koszul dual algebra $\cK^!$ described in this paper. When $L$ was a 2-component L-space link, we described a model of the surgery $DA$-bimodule ${}_{\cK} \cX(S^3,L)^{\bF[W,Z]}$. We used this construction to give an algorithmic description of the effect on knot Floer homology of a large number of satellite operators.  With Chen and Zhou, we implemented our algorithm from \cite{CZZSatellites} in Python code \cite{CZZCode}. We describe the connection between our present paper and \cite{CZZSatellites} in Section~\ref{sec:CZZ}.

 We expect that the techniques of the present paper will be useful in extending these algorithms to other families of satellite operators, and to developing other computational techniques involving the surgery formula.

\subsection{Acknowledgements}

We thank Daren Chen, Kristen Hendricks, Robert Lipshitz, Matthew Stoffregen, and Hugo Zhou for helpful conversations.

\section{Algebraic background}

\subsection{$A_\infty$-algebras and modules}

\begin{define} A (curved) \emph{$A_\infty$-algebra} $\cA$ over a ring $\ve{k}$ (of characterstic 2) consists of a $(\ve{k},\ve{k})$-bimodule $\cA$, equipped with $\ve{k}$-linear maps
\[
\mu_n\colon \underbrace{\cA\otimes_{\ve{k}}\cdots \otimes_{\ve{k}} \cA}_n\to \cA
\]
ranging over $n\ge 0$. These maps are assumed to satisfy the associativity relations
\[
0=\sum_{i=0}^{n-k+1} \sum_{k=1}^{n+1} \mu_{n-i+1}(a_1,\dots, a_{k-1}, \mu_i(a_k,\dots, a_{k+i-1}),a_{k+i},\dots, a_n)
\]
\end{define}

We view $\mu_0\colon \ve{k}\to \cA$ as an element of $\cA$, which is called the \emph{curvature} of $\cA$.

\begin{define}
 A \emph{curved $dg$-algebra} $\cA$ is a curved $A_\infty$-algebra $(\cA, \mu_i)$ so that $\mu_i=0$ if $i\not \in \{0,1,2\}$, and furthermore
 \[
 \mu_1^2=\mu_2(\mu_0,-)+\mu_2(-,\mu_0)=0.
 \]
 The last equation asks for $(\cA,\mu_1)$ to be a chain complex and for $\mu_0$ to be central in $\cA$.
\end{define}

\begin{define} Let $\cA$ be a (potentially curved) $A_\infty$-algebra. A left \emph{$A_\infty$-module} ${}_{\cA} M$ over $\cA$ is a left $\ve{k}$-module $M$ which is equipped with $\ve{k}$-linear maps
\[
m_{n+1}\colon \underbrace{\cA\otimes_{\ve{k}} \cdots \otimes_{\ve{k}} \cA}_n \otimes_{\ve{k}} M\to M
\]
for $n\ge 0$. These satisfy
\[
\begin{split}
0=&\sum_{i=0}^n m_{n-i+1}(a_n,\dots, a_{i+1}, m_{i+1}(a_i,\dots, a_1,\xs))\\
+&\sum_{i=0}^n \sum_{j=0}^{n-i} m_{n-i+2}(a_n,\dots, \mu_{i}(a_{i+j},\dots, a_{j+1}),a_{j},\dots, a_1, \xs).
\end{split}
\]
\end{define}

If ${}_{\cA} M$ and ${}_{\cA} N$ are $A_\infty$-modules, a \emph{morphism} $f_*\colon {}_{\cA}M\to {}_{\cA} N$ consists of a collection of $\ve{k}$-linear maps
\[
f_{j+1}\colon \cA\otimes_{\ve{k}} \cdots \otimes_{\ve{k}} \cA \otimes_{\ve{k}} M\to M,
\]
ranging over $j\ge 1$.
The space of morphisms from ${}_{\cA} M$ to ${}_{\cA} N$ forms a chain complex, where the  differential of a morphism is given by the formula:
\[
\begin{split}
\d(f_*)_{n+1}=&\sum_{i=0}^n f_{n-i+1}(a_n,\dots, a_{i+1}, m_{i+1}(a_i,\dots, a_1,\xs))\\
+&\sum_{i=0}^n m_{n-i+1}(a_n,\dots, a_{i+1}, f_{i+1}(a_i,\dots, a_1,\xs))\\
+&\sum_{i=0}^n \sum_{j=0}^{n-i} f_{n-i+2}(a_n,\dots, \mu_{i}(a_{i+j},\dots, a_{j+1}),a_{j},\dots, a_1, \xs).
\end{split}
\]
It is straightforward to verify that $\d^2=0$. Sometimes authors reserve the word \emph{morphism} for a map $f_*$ as above which satisfies $\d(f_*)=0$, however we adopt the more general definition above.

If $f_*\colon {}_{\cA} M\to {}_{\cA} N$ and $g_*\colon {}_{\cA} N\to {}_{\cA} P$ are morphisms, the composite $g_*\circ f_*$ is defined via the formula
\[
(g_*\circ f_*)_{n+1}(a_n,\dots, a_1,\xs)=\sum_{i=0}^{n} g_{n-i}(a_n,\dots, a_{i+1}, f_{i+1}(a_i,\dots, a_1,\xs)).
\]
It is straightforward to see that composition is strictly associative and satisfies
\[
\d(g\circ f)=\d(g)\circ f+g\circ \d(f).
\]

We now recall the following basic notion:

\begin{define}
\label{def:strict-unality}
\item
\begin{enumerate}
\item Suppose that $\cA$ is a (curved) $A_\infty$-algebra. We say that $\cA$ is \emph{strictly unital} if there is an element $1\in \cA$ such that $\mu_2(1,\xs)=\mu_2(\xs,1)=\xs$ for all $\xs\in \cA$, and if $\mu_n(a_1,\dots, a_{i-1},1,a_{i+1}, \dots, a_n)=0$ whenever $a_1,\dots, a_{i-1},a_{i+1},\dots ,a_n\in \cA$ and $n\neq 2$. 
\item If $\cA$ is a strictly unital $A_\infty$-algebra, then a bimodule ${}_{\cA} M$ is called \emph{strictly unital} if $m_{2}(1,\xs)=\xs$ for all $\xs\in M$, and if $m_{n+1}(a_n,\dots, a_{i+1},1,a_{i-1},\dots, a_1,\xs)=0$ whenever $n>1$ . A morphism $f_{*} \colon {}_{\cA} M\to {}_{\cA} N$ is called \emph{strictly unital} if $f_{n+1}(a_n,\dots,a_{i+1},1,a_{i-1},\dots,  a_1, \xs)=0$ whenever $n>0$.
\end{enumerate}
\end{define}

\subsection{Type-$D$, $DD$, $DA$ and $AA$ modules}

We now recall the category of type-$D$ modules, defined by Lipshitz, Ozsv\'{a}th and Thurston \cite{LOTBordered}. If $\cA$ is an $A_\infty$-algebra over $\ve{k}$, then a \emph{right type-$D$ module} $X^{\cA}$ consists of a right $\ve{k}$-module $X$ equipped with a map $\delta^1\colon X\to X\otimes \cA$, which satisfies the following type-$D$ structure relation. We write $\delta^n\colon X\to X\otimes \cA^{\otimes n}$ for the $n$-th iterate of $\delta^1$. We assume that
\[
\sum_{n=0}^\infty (\bI_{X}\otimes \mu_n)\circ \delta^n=0.
\]
Some assumption is required for the above sum to make sense. For our purposes, it is simplest to assume that $\mu_n=0$ for $n\gg 0$.

Type-$D$ modules over an $A_\infty$-algebra $\cA$ form an $A_\infty$-category. If $\cA$ is a $dg$-algebra (potentially with curvature), then type-$D$ modules over $\cA$ form a $dg$-category.  We refer the reader to \cite{LOTBimodules}*{Section~2.2.3} for more details. 

If $\cA$ and $\cB$ are $dg$-algebras over $\ve{i}$ and $\ve{j}$ respectively, then a type-$DD$ bimodule ${}^{\cA} M^{\cB}$ consists of a $(\ve{i},\ve{j})$-bimodule with a structure map
\[
\delta^{1,1}\colon M\to \cA\otimes_{\ve{i}} M \otimes_{\ve{j}} \cB,
\]
which satisfies a certain structure equation. The structure equation is same as the type-$D$ structure equation if we instead view $\delta^{1,1}$ as a map
\[
\delta^{1,1}\colon M\to \cA\otimes_{\bF} \cB^{\opp} \otimes  M
\]
 where $\cB^{\opp}$ is the opposite algebra (i.e. $\mu_2^{\opp}(b_1, b_2)=\mu_2(b_2,b_1)$). Here, the tensor product $\cA\otimes_{\bF} \cB^{\opp}$ is equipped with differential $\mu_1\otimes \id+\id\otimes \mu_1$, and curvature $1\otimes \mu_0+\mu_0\otimes 1$. 

We now consider $AA$-bimodules.  If $\cA$ and $\cB$ are $A_\infty$-algebras over $\ve{j}$ and $\ve{k}$, then a type-$AA$ bimodule ${}_{\cA} M_{\cB}$ consists of a $(\ve{j},\ve{k})$-bimodule $M$, equipped with structure maps
\[
m_{n|1|k}\colon \cA^{\otimes n}\otimes M\otimes \cB^{\otimes k}\to M
\]
which satisfy the following associativity relation:
\[
\begin{split}
0&=\sum_{\substack{0\le i\le n\\
0\le j\le k}} m_{n-i|1|k-j}(a_n,\dots, a_{i+1}, m_{i|1|j}(a_i,\dots, a_1,\xs,b_1,\dots, b_j),b_{j+1},\dots, b_k)\\
&+\sum_{0\le i\le j\le n} m_{n-j+i+1|1|k}(a_n,\dots,a_{j+1}, \mu_{j-i}(a_j,\dots, a_{i+1}),a_i,\dots, a_1, \xs, b_1,\dots, b_k)\\
&+\sum_{0\le i\le j\le n} m_{n|1|k-j+i+1}(a_n,\dots, a_1, \xs, b_1,\dots,b_i, \mu_{j-i}(b_{i+1},\dots, b_j),b_{j+1},\dots, b_k)
\end{split}
\]

A \emph{type-$DA$ bimodule} ${}_{\cA} M^{\cB}$ consists of a $(\ve{j},\ve{k})$-bimodule $M$, equipped with $(\ve{j},\ve{k})$ linear maps
\[
\delta_{n+1}^1\colon \cA^{\otimes n}\otimes M\to M\otimes \cB.
\]
These satisfy a suitable compatibility relation, which is indicated schematically by the following diagram:
\begin{equation}
\begin{tikzcd}[row sep=.3cm] \cA^{\otimes n} \ar[d,Rightarrow] & M \ar[dd]&\,\\
\Delta\ar[dr,Rightarrow] \ar[ddr,Rightarrow]\ar[dddr,Rightarrow]&\\
& \delta_*^1 \ar[d] \ar[dddr]\\
& \vdots \ar[d] \ar[ddr,Rightarrow]\\
& \delta_*^1\ar[dr] \ar[dd]\\
&\,& \mu_*\ar[d]\\
&\,&\,
\end{tikzcd}+\begin{tikzcd}[row sep=.3cm, column sep={1cm,between origins}] \cA^{\otimes n} \ar[dr,Rightarrow]&& M \ar[dd]&\,\\
&D_{\cA}\ar[dr,Rightarrow] &\\
&& \delta_*^1 \ar[d] \ar[dr]\\
&\,&\,&\,
\end{tikzcd}=0.
\label{eq:differential-squares-to-zero}
\end{equation}
Here, $D_{\cA}$ denotes the sum of maps of the form $\id_{\cA^{\otimes j}}\otimes \mu_{k}\otimes \id_{\cA^{\otimes (n-j-k)}}$. In the above diagram, $\Delta$ denotes the sum of all ways of splitting an elementary tensor $a_1\otimes \cdots \otimes a_n$ into a tensor of tensors by adding (non-nested) parentheses. More formally, one defines $T^* \cA=\bigoplus_{n=0}^\infty \cA^{\otimes n}$. Then we define maps $\Delta_i$ as follows. We set $\Delta_1=T^* \cA\to T^* \cA$ to be the identity. We set
\[
\Delta_2 \colon T^* \cA\to (T^* \cA)\otimes (T^* \cA)
\]
to be the map $\Delta_2(a_1\otimes \cdots \otimes a_n)=\sum_{k=0}^n (a_1\otimes \cdots \otimes  a_k)\otimes (a_{k+1}\otimes \cdots\otimes a_n)$. For $i>2$, we set $\Delta_{i}=\bI_{(T^* \cA)^{\otimes i-1}} \otimes \Delta_2$. Then $\Delta$, in the above diagram, denotes the sum over all $\Delta_i$. See \cite{LOTBimodules}*{Section~2.2.4} for further details.

The notion of strict unality, Definition~\ref{def:strict-unality}, extends easily to type-$DA$ and $AA$ bimodules. 

\begin{define}
Suppose $\cA$ and $\cB$ are strictly unital $A_\infty$-algebras.
\begin{enumerate}
\item A $DA$-bimodule ${}_{\cA} M^{\cB}$ is \emph{strictly unital} if $\delta_{2}^1(1,\xs)=\xs\otimes 1$ for all $\xs\in M$, and if
\[
\delta_{n+1}^1(a_n,\dots, a_{i+1},1,a_{i-1},\dots, a_1, \xs)=0
\]
if $n>1$.
\item A type-$AA$ bimodule ${}_{\cA} M_{\cB}$ is \emph{strictly unital} if $m_{1|1|0}(1,\xs)=\xs=m_{0|1|1}(\xs,1)$ for all $\xs\in M$, and if
\[
m_{n|1|m}(a_n,\dots, a_1, \xs, b_1,\dots, b_m)=0
\]
if $n+m>1$ and if some $a_i$ or $b_i$ is 1.
\end{enumerate}
\end{define}

\subsection{The homological perturbation lemma}

We now recall the homological perturbation lemma. The homological perturbation lemma is a very flexible technique which applies broadly to $dg$-categories, though we focus on the case of $AA$-bimodules since that is the main setting that we use the lemma. 

\begin{lem} Let $\cA$ and $\cB$ be $A_\infty$-algebras over $\ve{j}$ and $\ve{k}$, respectively. Suppose that ${}_{\cA} M_{\cB}$ is a type-$AA$ bimodule and $(Z,m_{0|1|0})$ is a chain complex of $(\ve{j},\ve{k})$-modules such that there is a diagram of $(\ve{j},\ve{k})$-linear maps
\begin{equation}
\begin{tikzcd}
\ar[loop left, "h"]  M \ar[r, shift left, "\pi"]& Z  \ar[l, shift left, "i"]
\end{tikzcd}
\label{eq:sdr-diagram}
\end{equation}
such that $\d (\pi)=0$, $\d(i)=0$ and
\begin{equation}
\pi\circ i=\bI_Z,\,\, i\circ \pi= \bI+\d(h),\,\, h\circ h=0, \,\, \pi\circ h=0,\,\,h\circ i=0.
\label{eq:strong-deformation-retraction}
\end{equation}
Here, $\d(\pi)$ denotes the morphism differential of $\pi$, when viewed as a map of chain complexes (i.e. $\d(\pi)=[m_{0|1|0},\pi]$). Then there is a type-$AA$ bimodule structure on $Z$ over $(\cA,\cB)$ which extends $m_{0|1|0}$, as well as extensions of $\pi$, $h$ and $i$ to $AA$-bimodule morphisms $\Pi$, $I$ and $H$, so that $d(\Pi)=0$, $\d(I)=0$ and the analog of Equation~\eqref{eq:strong-deformation-retraction} holds.
\end{lem}

The proof is standard, so we omit it. For future reference, one frequently defines a \emph{strong deformation retraction} to be a homotopy equivalence given by three maps $\pi$, $i$ and $h$, as in Equation~\eqref{eq:sdr-diagram},  which satisfy $\d(\pi)=0$, $\d(i)=0$ and which satisfy Equation~\eqref{eq:strong-deformation-retraction}.

There is a generalization of the above result which is also quite helpful, and which we will make use of. Let ${}_{\cA} M_{\cB}=(M,m_{i|1|j})$ be a type-$AA$ bimodule. A \emph{twist} of ${}_{\cA} M_{\cB}$ consists of a type-$AA$ morphism
\[
\a=\a_{*|1|*}\colon {}_{\cA} M_{\cB}\to {}_{\cA} M_{\cB}
\]
which satisfies the Mauer-Cartan equation
\[
\a_{*|1|*}\circ \a_{*|1|*}+\d(\a_{*|1|*})=0.
\]
Here, $\d(\a_{*|1|*})$ denotes the boundary of $\a_{*|1|*}$ when viewed as an endomorphism of ${}_{\cA} M_{\cB}$. 
Given a twist $\a$ of ${}_{\cA} M_{\cB}$, we can define $\Tw_{\a}({}_{\cA} M_{\cB})$ to be the type-$AA$ bimodule whose underlying $(\ve{i},\ve{j})$-module is $M$ and whose structure maps are $m_{i|1|j}+\a_{i|1|j}$.

\begin{lem}
\label{lem:HPL-twist} Suppose that ${}_{\cA} M_{\cB}$ and ${}_{\cA} Z_{\cB}$ are type-$AA$ bimodules so that there is a diagram of morphisms of $AA$-bimodules
\[
\begin{tikzcd}
\ar[loop left, "h"] {}_{\cA} M_{\cB} \ar[r, shift left, "\pi"]& {}_{\cA} Z_{\cB}  \ar[l, shift left, "i"]
\end{tikzcd}
\]
so that $\d (\pi)=0$, $\d(i)=0$ and
\begin{equation}
\pi\circ i=\bI_Z,\,\, i\circ \pi= \bI_M+\d(h),\,\, h\circ h=0, \,\, \pi\circ h=0,\,\,h\circ i=0.\label{eq:strong-deformation-retraction2}
\end{equation}
Let $\a$ be a twist of ${}_{\cA} M_{\cB}$ and suppose that $\bI_M+ h \circ \a$ is invertible, with inverse $\sum_{i=0}^\infty (h\circ \a)^i$. Then there is a twist $\b$ of ${}_{\cA} Z_{\cB}$, as well as a diagram of maps
\[
\begin{tikzcd}
\ar[loop left, "H"] \Tw_{\a}({}_{\cA} M_{\cB}) \ar[r, shift left, "\Pi"]& \Tw_\b({}_{\cA} Z_{\cB})  \ar[l, shift left, "I"],
\end{tikzcd}
\]
so that $\d(\Pi)=0$, $\d(I)=0$ and which satisfy the natural adaptation of Equation~\eqref{eq:strong-deformation-retraction2}.
\end{lem}

One very helpful aspect of the above lemma is that $\b$, $\Pi$, $I$ and $H$ have explicit formulas. Namely they are given by
\[
\begin{split}
\b&=\sum \pi \circ \a \circ h\circ \cdots \circ h\circ \a \circ i\\
\Pi&=\sum \pi \circ \a \circ h\circ \cdots \circ \a\circ h\\
I&= \sum h\circ \a \circ \dots \circ h\circ \a \circ i\\
H&= \sum h\circ \a \circ \cdots \circ h\circ \a\circ h.
\end{split}
\]
The proof is standard.

\subsection{Modules over $\cK$}
\label{sec:regularly-U-adic-modules}

In \cite{ZemBordered,ZemExact}, we defined several module categories over $\cK$. We review these presently. They are versions of the type-$D$ and type-$A$ module categories of Lipshitz, Ozsv\'{a}th and Thurston, but adapted to the setting of linear topological spaces. Here, we present modified versions of these categories  which are better adapted to the setting of Koszul duality.

\begin{define}
\label{def:regularly-filtered}
 Let $\ve{k}$ be a ring. A \emph{regular $\Z$-filtered $\ve{k}$-module} consists of a left $\ve{k}$-module $M$, equipped with a decreasing filtration of $\ve{k}$-submodules $M_{\ge i}\subset M$, ranging over $i\in \Z$. Furthermore, we assume that $M_{\ge i}=M$ for $i\ll 0$ and $\bigcap_{i=0}^\infty M_{\ge i}=\{0\}$. 
\end{define}

We are primarily interested in the case that $\ve{k}$ is $\bF=\Z/2$ or $\ve{I}$.

The \emph{completion} of a regular $\Z$-filtered module consists of the inverse limit
\[
\veit{M}=\varprojlim_{i\in \Z} M/M_{\ge i}.
\]
The completion is also a regularly $\Z$-filtered $\ve{k}$-module.

\begin{define}
A map $f\colon M\to N$ is said to be \emph{weakly filtered} if there is a constant $C$ so that
\[
f(M_{\ge i})\subset N_{\ge i-C}
\]
for all $i\in \Z$. 
\end{define}
A weakly filtered map between regularly filtered modules induces a linear map between their completions, which is also weakly filtered.

In the category of regularly $\Z$-filtered spaces, we define a morphism from  ${}_{\ve{k}} M$ to ${}_{\ve{k}} N$ to be a $\ve{k}$-linear, weakly filtered map between their completions. In particular, in the so-formed category, a module is isomorphic to its completion. Therefore we typically conflate a regularly $\Z$-filtered $\ve{k}$-module with its completion. 

\begin{define}
\label{def:commensurate} Suppose that $M_{\a}, N_{\a}$ are two families of regularly $\Z$-filtered $\ve{k}$-modules, ranging over $\a\in A$, and that
\[
f_\a\colon M_\a\to N_{\a}
\] is a family of weakly filtered maps. We say that the maps $f_\a$ are \emph{commensurate} if there is a single number $C$ so that
\[
f_\a((M_\a)_{\ge i})\subset (N_{\a})_{\ge i-C}
\]
for all $i\in \Z$ and $\a\in A$.
\end{define}

\begin{define} If $M_{\ve{k}}$ and ${}_{\ve{k}}N$ are $\Z$-filtered $\ve{k}$-modules, then the tensor product $M \otimes_{\ve{k}} N$ is also  $\Z$-filtered, with filtration levels
\[
(M\otimes_{\ve{k}} N)_{\ge n}=\sum_{i+j=n} M_{\ge i}\otimes N_{\ge j}.
\]
\end{define}

We consider the following filtration on $\cK$:

\begin{define} We define the \emph{$U$-adic filtration} on $\cK$ to be the filtration induced by the sequence of two sided ideals $\cK_{\ge i}:=(U^i)\subset \cK$. 
\end{define}

\begin{rem} In \cite{ZemBordered,ZemExact}, we also considered an additional perspective on $\cK$, which was as a linear topological chiral algebra. For simplicity, we will not consider that perspective in this paper.
\end{rem}

We define two categories, ${}_{\cK}\MOD_{(U)}$ and $\MOD_{(U)}^{\cK}$. We call these the categories of \emph{regularly $U$-adic} type-$D$ and $A$ modules. Modules and morphisms in this category are slightly more restrictive than the categories of $U$-adic type-$D$ and $A$ modules from \cite{ZemExact}. 

Before defining them, we need some auxiliary definitions:
\begin{define}
\label{def:input-trees} 
\item
\begin{enumerate}
\item  We define an  \emph{(ordinary) $A_\infty$-module structure tree} $\cT$ to be a connected tree which is embedded in the unit disk $D$ and which satisfies the following. The vertices which are embedded on $\d D$ are called \emph{boundary vertices}, and the rest are called \emph{interior vertices}. One boundary vertex is labeled as the \emph{output}, and the rest are \emph{inputs}.  We assume the output is embedded as the lowest point in the disk (i.e. $-i$, if we view $D$ as the unit complex disk). We assume that there is at least one input vertex. We assume that the boundary vertices have valence 1, and all other vertices have valence 3 or greater.  We label the rightmost input vertex as the \emph{module input}, and we call the output the \emph{root}. Note that there is a unique path from this input to the root. 
\item We define an \emph{extended $A_\infty$-module input tree} $\cT_+$, similarly to the above, except that we additionally allow interior vertices which have valence 1 or 2.
\end{enumerate}
\end{define}

Given an $A_\infty$-module ${}_{\cA} M$ and an extended $A_\infty$-module structure tree $\cT_+$, there is an evaluation map $m_{\cT_+}$ obtained by composing the structure maps $\mu_i$ from $\cA$ and the actions $m_{i+1}$ from $M$, according to the tree $\cT$. Valence 2 vertices are associated with $\mu_1$ or $m_1$, and valence 1 vertices with $\mu_0$.

 The degree of a tree $\cT$ is defined
 \[
 \deg(\cT)=\sum_{v\in V_{\Int}(\cT)} (\val(v)-3),
 \]
 where $V_{\Int}(\cT)$ are the interior vertices, and $\val(v)$ is the valence.

\begin{define}
\label{def:regularly-filtered-type-A} The objects of ${}_{\cK} \MOD_{(U)}$ consist of regularly $\Z$-filtered, left $\ve{I}$-modules $M$, equipped with weakly filtered maps
\[
m_{n+1}\colon \underbrace{\cK\otimes_{\mathrm{I}} \cdots \otimes_{\mathrm{I}} \cK}_{n}\otimes M\to M
\]
such that the maps $m_{\cT_+}$, ranging over all extended $A_\infty$-operation trees $\cT_+$, are commensurate. Furthermore, the maps $m_{n+1}$ are assumed to satisfy the $A_\infty$-module relations.
A \emph{morphism} in ${}_{\cK} \MOD_{(U)}$ from ${}_{\cK} M$ to ${}_{\cK} N$ consists of a collection of weakly filtered maps
\[
f_{n+1}\colon \cK\otimes\cdots \otimes \cK\otimes M\to N,
\]
ranging over $n\ge 0$, which are commensurate.
\end{define}

\begin{lem}
\label{rem:only-m2m1} Suppose that ${}_{\cK} M$ is a type-$A$ module with $m_j=0$ for $j\not\in \{1,2\}$. If $m_1$ and $m_2$ are weakly filtered, then ${}_{\cK} M$ satisfies Definition~\ref{def:regularly-filtered-type-A}. 
\end{lem}
\begin{proof} If $\cT$ is an extended $A_\infty$-module operation tree, then $m_{\cT}$ can be written as the composition of several copies of $\mu_2$ of $\cK$, together with a single application of $m_2$, and either 0 or 1 copies of $m_1$. This is because $m_{j+1}=0$ if $j>1$, so multiplication is strictly associative and $m_1$ commutes with $m_2(a,-)$ for all $a$. Since multiplication on $\cK$ is filtered, i.e. $\mu_2$ maps $\cK_{\ge i}\otimes \cK_{\ge j}$ to $\cK_{\ge i+j}$, it follows that $m_{\cT}$ is weakly filtered with the constant $C_1+C_2$, where $C_1$ is the constant for $m_1$ and $C_2$ is the constant for $m_2$.
\end{proof}

The category ${}_{\cK} \MOD_{(U)}$ forms a $dg$-category in the same way as ordinary $A_\infty$-modules.

We now discuss type-$D$ modules. We define the type-$D$ module category $\MOD^{\cK}_{(U)}$ as follows. Objects are pairs $(M,\delta^1)$ where $M$ is a regularly $\Z$-filtered, right $\ve{I}$-module and $\delta^1$ is a weakly filtered map
\[
\delta^1\colon M\to M\otimes_{\mathrm{I}} \cK
\]
which satisfies the type-$D$ structure relation. Furthermore, we assume that that the iterates
\[
\delta^n\colon M\to M\otimes \cK\otimes \cdots \otimes \cK
\]
are commensurate.

The construction in \cite{ZemExact}*{Section~3} produces knot (and link) surgery type-$D$ modules $\cX(Y,K)^{\cK}$ which are finitely  generated over $\ve{I}$. We now show that such modules are always regularly $U$-adic.

\begin{lem}\label{lem:fg-weakly-filtered} Suppose that $M$ is a finitely generated, regularly $\Z$-filtered, right $\ve{I}$-module and $\delta^1\colon M\to M\otimes \cK$ is a weakly-filtered map which satisfies the type-$D$ structure relation. Then $(M,\delta^1)$ is a regularly $U$-adic type-$D$ module.
\end{lem}
\begin{proof} Pick $N>0$ so that $M_{\ge N}=0$ and $M_{\ge -N}=M$, which exists since $M$ is regularly filtered and finitely generated. Let $C=2N$. If $i<N$, then
\[
\delta^n(M_{\ge i})\subset M\otimes\cK\otimes  \cdots \otimes  \cK= M_{\ge i-2N}\otimes \cK\otimes \cdots \otimes \cK
\]
which is contained in $(M\otimes \cK\otimes \cdots \otimes \cK)_{\ge i-C}. $ Similarly if $i\ge N$, then $M_{\ge i}=0$, so
\[
\delta^1(M_{\ge i})\subset (M\otimes \cK\otimes \cdots \otimes \cK)_{\ge i-C}
\]
vacuously.
\end{proof}

\subsection{Gradings on $\cK$}
\label{sec:gradings-K}

We now discuss gradings on the algebra $\cK$. We describe two perspectives on these gradings. The first is as a collection of several Maslov and Alexander gradings which satisfy a compatibility condition with respect to multiplication. The second perspective is in terms of a group valued grading, similar to the grading that appears in the bordered theory of Lipshitz, Ozsv\'{a}th and Thurston \cite{LOTBordered}*{Section~3.3}. 

On $\ve{I}_0 \cdot \cK\cdot  \ve{I}_0$, we will define a $\Z\times \Z$-valued Maslov bigrading, denoted $(\gr_{\ws},\gr_{\zs})$. On the other idempotents, we will define a $\Z\times \Z$-valued Maslov-Alexander grading, denoted $(\gr,A)$. These are defined on generators as
 \[
 \begin{split}
 (\gr_{\ws},\gr_{\zs})(W^i Z^j)=&(-2i,-2j)\\
 (\gr,A)(U^i T^j\sigma)=&(-2i,j),\\
 (\gr,A)(U^iT^j \tau)=&(-2i,j),\\
  (\gr,A)(U^iT^j)=&(-2i, j).
  \end{split}
 \]
 
 These are additive under multiplication in the following sense. If $a,b\in \ve{I}_0\cdot \cK\cdot \ve{I}_0$ then
 \[
 (\gr_{\ws},\gr_{\zs})(ab)= (\gr_{\ws},\gr_{\zs})(a)+ (\gr_{\ws},\gr_{\zs})(b).
 \]
If $a\in \ve{I}_1\cdot \cK\cdot \ve{I}_1$ and $b\in \ve{I}_1\cdot \cK\cdot \ve{I}_\veps$, for $\veps\in \{0,1\}$, then 
\[
(\gr,A)(ab)=(\gr,A)(a)+(\gr,A)(b).
\]  If $a$ is a multiple of $\sigma$, and $b\in \ve{I}_0 \cdot \cK \cdot \ve{I}_0$, then
 \begin{equation}
 (\gr,A)(a\cdot b)=(\gr,A)(a)+\left(\gr_{\ws}(b), \frac{\gr_{\ws}(b)-\gr_{\zs}(b)}{2}\right).\label{eq:add-gradings-1}
 \end{equation}
 If instead $a$ is a multiple of $\tau$, then
 \begin{equation}
 (\gr,A)(a\cdot b)=(\gr,A)(a)+\left(\gr_{\zs}(b), \frac{\gr_{\ws}(b)-\gr_{\zs}(b)}{2}\right).\label{eq:add-gradings-2}
 \end{equation}
 We will write $g(a)$ for $(\gr_{\ws},\gr_{\zs})(a)$ or $(\gr,A)(a)$, depending on the idempotent. We can rephrase the above as follows:
\begin{lem}\label{lem:grading-K} The action on $\cK$ satisfies
\[
g(\mu_2(a,b))=g(a)+g(b),
\] where addition of $g(a)$ and $g(b)$ is interpreted as in the right hand sides of Equations~\eqref{eq:add-gradings-1} and~\eqref{eq:add-gradings-2}.
\end{lem}

In \cite{LOTBordered}*{Section~3.3}, Lipshitz, Ozsv\'{a}th and Thurston construct gradings on their torus algebra by a non-abelian group. We give several methods for repackaging the above Maslov and Alexander gradings in a similar manner.  For most practical computations, it is easier to work with Maslov and Alexander gradings, though the perspective as a group valued grading is more elegant.

We first describe a grading by a groupoid $\cG$ with two objects, $i_0$ and $i_1$ (corresponding to the two idempotents of $\cK$). Firstly, we will write $\Aff(\R^2)$ for the group of invertible affine transformations of $\R^2$. If $\ve{v}\in \R^2$, write $\mathrm{T}_{\ve{v}}\in \Aff(\R^2)$ for the transformation
\[
\mathrm{T}_{\ve{v}}(\xs)=\xs+\ve{v}.
\]
 Let $\lambda_0,\lambda_1\in \Aff(\R^2)$ denote the transformations
 \[
 \lambda_0=\rmT_{\left(\begin{smallmatrix}\Sss{1}\\\Sss{1} \end{smallmatrix}\right)}, \quad  \lambda_1=\rmT_{\left(\begin{smallmatrix}\Sss{1}\\\Sss{0} \end{smallmatrix}\right)}.
 \]
 If $\veps,\veps'\in \{0,1\}$, we define $\cG(i_{\veps},i_{\veps'})$ (the $\Hom$ space of morphisms from $i_{\veps}$ to $i_{\veps'}$) to be the subset of $\Aff(\R^2)$ consisting of all invertible transformations $B\colon \R^2\to \R^2$ with the property that
 \[
 B\circ \lambda_{\veps}=\lambda_{\veps'}\circ B.
 \]
 
 The algebra $\cK$ is naturally graded by $\cG$. We define
 \[
 g(W)=\rmT_{\left(\begin{smallmatrix}\Sss{-2}\\\Sss{0} \end{smallmatrix}\right)},\quad  g(Z)=\rmT_{\left(\begin{smallmatrix}\Sss{0}\\\Sss{-2} \end{smallmatrix}\right)},\quad  g(U)=\rmT_{\left(\begin{smallmatrix}\Sss{-2}\\\Sss{0} \end{smallmatrix}\right)},\quad  g(T^{\pm 1})=\rmT_{\left(\begin{smallmatrix}\Sss{0}\\\Sss{\pm1} \end{smallmatrix}\right)}.
 \]
 In the above equation, $U$ denotes the so-named element in $\ve{I}_1\cdot \cK\cdot \ve{I}_1$. We define $g(\sigma)$ and $g(\tau)$ to be the linear transformations represented by the following matrices
  \[
  g(\sigma)=\begin{pmatrix} 1&0 \\1/2 &-1/2 \end{pmatrix}\quad g(\tau)=\begin{pmatrix} 0&1 \\1/2 &-1/2 \end{pmatrix}.
  \]
  We define multiplication in $\cG$ to be composition of functions, i.e. $g* g'=g\circ g'$.

  \begin{lem} The algebra $\cK$ is graded by the groupoid $\cK$, i.e., 
  \[
  g(\mu_2(a, b))=g(a)* g(b).
  \]
  \end{lem}

   \begin{rem} The correspondence between the $\cG$-valued grading and the $(\gr_{\ws},\gr_{\zs})$ and $(\gr,A)$ gradings on generators is as follows. If  $a\in \ve{I}_0\cdot \cK \cdot \ve{I}_0$ has $(\gr_{\ws}, \gr_{\zs})$ grading $\xs_0\in \Z\times \Z$, then
    \[
   g(a)= \rmT_{\xs_0^T}.
   \]
   Here $\xs_0^T$ denotes the transpose of the row vector $\xs_0$. A similar formula holds for algebra elements in $\ve{I}_1\cdot \cK\cdot \ve{I}_1$. 
   If $a\in \ve{I}_1 \cdot \cK \cdot \ve{I}_0$ is a multiple of $\sigma$, then \[
   g(a)=\mathrm{T}_{(\gr,A)^T(a)}\circ g(\sigma).\] A similar formula holds if $a$ is a multiple of $\tau$.
   Note that $(\gr,A)(a)$ and $g(a)$ contain essentially equivalent information. In practice, it is typically easier to work with the $(\gr,A)$ grading instead of the $\cG$-valued grading.  
   \end{rem}

\begin{rem} The groupoid valued grading $g$ can also be repackaged into a grading $g'$ by a subgroup of $\cG(i_1,i_1)$ by collapsing the gradings to $\cG(i_1,i_1)$ using $g(\sigma)$. That is, if $a\in \ve{I}_0\cdot \cK\cdot \ve{I}_0$, we define $g'(a)= g(\sigma) g(a) g(\sigma)^{-1}$. If $a\in \ve{I}_1\cdot \cK\cdot \ve{I}_0$, we set $g'(a)=g(a) g(\sigma)^{-1}$. If $a\in \ve{I}_1\cdot \cK\cdot \ve{I}_1$, we set $g'(a)=g(a)$. 

Let  $\Heis(\R^2)\subset \Aff(\R^2)$ denote the \emph{Heisenberg group}, which we recall is the set of affine transformations of the form
  \[
  \xs\mapsto \begin{pmatrix} 1 & a\\ 0& 1\end{pmatrix}\xs +\begin{pmatrix} b\\ c \end{pmatrix},
  \]
  for $a,b,c\in \R$. We note that the grading $g'$ takes values in $\Heis(\R^2)$. More explicitly, we have  $\lambda=\rmT_{\left(\begin{smallmatrix}\Sss{1}\\\Sss{0} \end{smallmatrix}\right)}$ and
  \[
  g'(W)=\rmT_{\left(\begin{smallmatrix}\Sss{-2}\\\Sss{-1} \end{smallmatrix}\right)}, \quad g'(Z)=\rmT_{\left(\begin{smallmatrix}\Sss{0}\\\Sss{1} \end{smallmatrix}\right)},\quad g'(\sigma)=\bI,\quad  g'(U)=\rmT_{\left(\begin{smallmatrix}\Sss{-2}\\\Sss{0} \end{smallmatrix}\right)},\quad  g'(T^{\pm 1})=\rmT_{\left(\begin{smallmatrix}\Sss{0}\\\Sss{\pm 1} \end{smallmatrix}\right)}.
  \]
Also $g'(\tau)$ is the linear transformation induced by the matrix
  \[
  g'(\tau):=g(\tau)g(\sigma)^{-1}=\begin{pmatrix} 1& -2\\0& 1\end{pmatrix}.
  \]
  \end{rem}
  
  \begin{rem} The appearance of the Heisenberg group is perhaps unsurprising, given that the original bordered algebra of Lipshitz, Ozsv\'{a}th and Thurston is also graded by the Heisenberg group of the intersection form of a surface. See \cite{LOTBordered}*{Chapter~10}. We present some basic exposition about this topic, for the benefit of the reader. If $V$ is a symplectic vector space with form $\omega$, we can define group $\Heis(V,\omega)$ to be $V\times \Q$, with group operation $(\ve{v},r)*(\ve{w}, q)=(\ve{v}+\ve{w}, q+r+\omega(\ve{v},\ve{w}))$. In the Lipshitz, Ozsv\'{a}th, Thurston theory, they consider an integral version where $V$ is a certain relative homology class of a punctured torus, and $\omega$ is an intersection form on this group.

     In our present setting, we take $V$ to be $\Q\oplus \Q\iso H_1(\bT^2;\Q)$. We view the first component of $V$ as spanned by the longitude $[\lambda]$, and the second component as spanned by the meridian $[\mu]$, equipped with the standard intersection form $\omega$ on $H_1(\bT^2;\Q)$. 
  Recall that the group $\Heis(\R^2)$, defined above, can be identified with the group of $3\times 3$ matrices of the form
  \[
  \begin{pmatrix}1 & v_1& r\\
  0&1&v_2\\
  0&0&1
  \end{pmatrix}.
  \]
The above matrix is identified with the transformation
  \[
  \xs\mapsto \begin{pmatrix}1&v_1\\
  0&1
  \end{pmatrix}\xs+\begin{pmatrix} r\\v_2 \end{pmatrix}.
  \]
  In terms of the previously defined gradings, we can write
  \[
  g(a)' =\begin{pmatrix}1 & e& \gr\\
    0&1&A\\
    0&0&1
    \end{pmatrix}
  \]
  where $\gr$ is the Maslov grading, $A$ is the Alexander grading, and $e\in \{0,-2\}$ is $-2$ if $a$ is a multiple of $\tau$, and 0 otherwise.
  \end{rem}

\section{The Koszul dual algebra $\cK^!$}

In this section, we define our algebra $\cK^!$ and prove that it is Koszul dual to $\cK$.

\subsection{The algebra $\cK^!$}

We now recall the algebra $\cK^!$ from the introduction. We set
\[
\ve{I}_0\cdot \cK^!\cdot \ve{I}_0=
\frac{\bF\langle w,z,\theta\rangle }{\theta^2= w^2=z^2=[w,\theta]=[z,\theta]=0}. 
\]
Here, $\bF\langle X_1,\dots, X_n\rangle$ denotes the free, non-commutative ring on the generators $X_1,\dots, X_n$. 

We set 
\[
\ve{I}_1\cdot\cK^!\cdot \ve{I}_1=\frac{\bF\langle \varphi_+,\varphi_-,\theta\rangle}{[\theta,\varphi_{\pm}]=\theta^2=\varphi_{+}^2=\varphi_-^2=0}.
\]
We set $\ve{I}_1\cdot \cK^!\cdot \ve{I}_0=0$. Finally we define $\ve{I}_0\cdot \cK^!\cdot \ve{I}_1$ to be spanned as a bimodule over $(\ve{I}_0\cdot \cK^!\cdot \ve{I}_0, \ve{I}_1\cdot \cK^! \cdot \ve{I}_1)$ by two generators $s$ and $t$. These satisfy the relations
\[
z  s=s \varphi_+, \quad w t=t \varphi_-,\quad  \theta  s=s \theta,\quad \theta  t=t\theta,
\]
\[
 s \varphi_-= t\varphi_+=w s=z  t=0.
\]

Over $\bF$, the subspace $\ve{I}_0\cdot \cK^!\cdot \ve{I}_1$ has rank $8$, spanned by the elements
 over $\bF$ by the generators
\[
s, \,\, zs, \,\, \theta s,\,\, \theta zs, \,\, t, \,\, wt, \,\, \theta t,\,\, \theta wt.
\]

We define the differential $\mu_1$ to satisfy
\[
\mu_1(\theta )=wz +z w,
\]
and to vanish on $w$, $z$, $s$, $t$, $\varphi_+$ and $\varphi_-$. We extend $\mu_1$ to all elements of $\cK^!$ via the Leibniz rule.

Finally, we define the curvature term
\[
\mu_0=\varphi_+\varphi_-+\varphi_-\varphi_+.
\]

\begin{lem} The algebra $\cK^!$ is a curved $dg$-algebra.
\end{lem}
\begin{proof}Clearly multiplication is associative. It suffices therefore to show that $\mu_1^2=0$, $\mu_1(\mu_0)=0$, and that $\mu_0$ is central. The relations $\mu_1^2=0$ and $\mu_1(\mu_0)=0$ are straightforward to show.

We now show that $\mu_0$ is central. Since $\mu_0$ is contained in $\ve{I}_1\cdot \cK^!\cdot \ve{I}_1$, it trivially commutes with algebra elements in $\ve{I}_0\cdot \cK^!\cdot \ve{I}_0$. It is also not hard to show that $\mu_0$ is central in $\ve{I}_1\cdot \cK^1 \cdot \ve{I}_1$. Therefore it remains to show that $s \cdot\mu_0=t\cdot\mu_0=0$. Using the relations that $s \varphi_+=z  s$ and $s \varphi_-=0$, we see that
\[
s \cdot (\varphi_-\varphi_++\varphi_+\varphi_-)=s \varphi_+\varphi_-=z  s \varphi_-=0.
\]
Therefore $\mu_0$ is central, and the claim follows.
\end{proof}

\begin{rem} We can describe $\cK$ is a \emph{quadratic-linear-scalar algebra}, i.e. an algebra $\cA$ obtained by quotienting the tensor algebra $T^* V$ on a set of generators by an ideal generated by a subspace $R\subset V^{\otimes 2}\oplus V\oplus \ve{k}$.  Our Koszul dual algebra $\cK^!$ is built by setting $V$ to be the span of the algebra elements $\{W,Z, U ,\sigma,\tau, T, T^{-1}\}$ (viewing $U$ as being a sum of a copy of $U$ in each idempotent). The relation $\sigma W=U T^{-1} \sigma$ does not look like a quadratic-linear-scalar relation, but is actually a consequence of the quadratic-linear-scalar relations
\[
T^{-1}T=1,\quad \sigma Z=T \sigma, \quad \sigma U=U\sigma, \quad\text{and} \quad ZW=U
\]
 via the following manipulation: 
 \[
 \sigma W= T^{-1} T \sigma W=T^{-1} \sigma Z W= T^{-1} \sigma U=T^{-1} U \sigma.
 \] 
  With respect to this description, the Koszul dual algebra $\cK^!$ follows from an essentially standard recipe. See \cite{PP_Quadratic}*{Section~5.4}. In general, one expects the Koszul dual of a quadratic-linear-scalar algebra to be a curved $dg$-algebra.
\end{rem}

\subsection{An equivalent $A_\infty$-algebra}
 We now define an equivalent $A_\infty$-algebra $\cK_{\infty}^{!}$. This will be a curved $A_\infty$-algebra whose only non-trivial actions are $\mu_0,\mu_2$ and $\mu_3$.
 
We  define $\cK^!_\infty$ as follows. We set
\[
\ve{I}_0\cdot \cK^!_\infty\cdot \ve{I}_0=\Lambda^*(w ,z ),
\]
i.e., the mod 2 exterior algebra on two generators, $w , z $.
We set 
\[
\ve{I}_1\cdot\cK^!_\infty\cdot \ve{I}_1=\frac{\bF\langle \varphi_+,\varphi_-,\theta\rangle }{[\theta,\varphi_{\pm}]=\theta^2=\varphi_{+}^2=\varphi_-^2=0}.
\]
We set $\ve{I}_1\cdot \cK_\infty^!\cdot \ve{I}_0=0$. Finally we define $\ve{I}_0\cdot \cK_\infty^!\cdot \ve{I}_1$ to be spanned over $\bF$ by the generators
\[
s,\quad  t, \quad z  s,\quad w  t,
\]
as well as the product of all of the above with $\theta$.

We have the relations
\[
z  s=s \varphi_+,\quad  w  t=t \varphi_-,\quad s \varphi_-= t\varphi_+=w  s=z  t=0.
\]
We extend these relations to all products with $\theta$.

The curvature is given by
\[
\mu_0=\varphi_-\varphi_++\varphi_+\varphi_-.
\]
The action of $\mu_3$ is given 
\[
\begin{split}
\mu_3(aw , z , s b)&=as\theta b,\\
\mu_3(aw,z,tb)&=a t\theta b,\\
\mu_3(az , w , s b)&=0,\\
\mu_3(az , w , t b)&=0,
\end{split}
\]
for any $a,b\in \cK^!$. We set $\mu_3$ to vanish on other sequences of monomials.

\begin{lem} The curved $A_\infty$-algebra relations are satisfied by $\cK_{\infty}^!$, and furthermore, there is an equivalence of curved $A_\infty$-algebras $\cK^!\simeq \cK^!_\infty$.
\end{lem}
\begin{proof}
We apply the homological perturbation lemma for $A_\infty$-algebras to the curved $dg$-algebra $\cK^!$. We first describe a homotopy equivalence of chain complexes of $(\ve{I},\ve{I})$-bimodules
\[
\begin{tikzcd} \ar[loop left, "H"] \cK^! \ar[r, "\Pi", shift left] & \cK^!_{\infty} \ar[l, "I", shift left]. 
\end{tikzcd}
\]
The homotopy equivalence we construct will be a \emph{strong deformation retraction}, i.e. $\Pi$, and $I$ will be chain maps, and furthermore
\[
\Pi\circ I=\bI_{\cK^!_\infty}, \quad I\circ \Pi=\bI_{\cK^!}+\d(H), \quad \Pi\circ H=0,\quad H\circ I=0, \quad H\circ H=0.
\]
The maps $I$ and $\Pi$ act by the identity on elements of $\ve{I}_0\cdot \cK^!_\infty\cdot \ve{I}_1$ and $\ve{I}_1\cdot \cK^!\cdot \ve{I}_1$. On $\ve{I}_0\cdot \cK^!_\infty\cdot \ve{I}_0$, we have
\[
I(1)=1,\quad I(w)=w,\quad I(z)=z, \quad I(wz)=zw.
\]
The map $\Pi$ sends $1$ to $1$, $w$ to $w$, $z$ to $z$, and both $wz$ and $zw$ to $wz$. The map $\Pi$ vanishes on other monomials in $\ve{I}_0\cdot \cK^!\cdot \ve{I}_0.$ 

The homotopy $H$ is zero except on $\ve{I}_0\cdot \cK^!\cdot \ve{I}_0$. We set $H(1)=H(w)=H(z)=0$. We set $H(zw)=0$ and $H(wz)=\theta$. For $i>0$, we set
\[
H((wz)^i)=(wz)^{i-1}\theta,\quad H((zw)^i)=(zw)^{i-1}\theta,
\]
\[
 H(w (zw)^i)=w(zw)^{i-1}\theta,\quad H(z (wz)^i)=z(wz)^{i-1}\theta. 
\]
We define $H$ to vanish on any multiple of $\theta$. We leave it to the reader to verify that this determines a strong deformation retraction of chain complexes.

We recall that the homological perturbation lemma for $A_\infty$-algebras endows $\cK^!_\infty$ with the structure of an $A_\infty$-algebra. The actions are given as follows. Given $a_1,\dots, a_n\in \cK_\infty^!$, to compute $\mu_n(a_1,\dots, a_n)$, we first include each $a_i$ into $\cK^!$ via $I$. Then we sum over all ways of successively applying $H\circ \mu_i$ to consecutive algebra tensor factors, until a final application of $\Pi\circ \mu_i$. Typically, such actions are described by trees of the following shape:
\[
\begin{tikzcd}[row sep=.5cm]
a_1 \ar[d] &a_2 \ar[d] &a_n \ar[d]\\
I\ar[dr] &I\ar[d]&I\ar[ddd]
\\
& \mu_2\ar[d]\\
& H \ar[dr]\\
&& \mu_2 \ar[d]\\
&& \Pi
\end{tikzcd}
\]
 See \cite{SeidelFukaya}*{Section~1i} for further background.  It is not hard to see that the $A_\infty$-algebra structure that this determines on $\cK^!_\infty$ is exactly the one which is described above. 
\end{proof}

\begin{rem} Although $\cK^!_{\infty}$ is smaller in rank than $\cK^!$, we mostly focus on $\cK^!$. The main reason is that since $\cK^!_{\infty}$ has a $\mu_3$ operation, the category of type-$D$ modules over $\cK^!_{\infty}$ is an $A_\infty$-category (instead of a $dg$-category), and is therefore a bit cumbersome to work with.
\end{rem}

\subsection{Gradings on $\cK^!$}

 Similar to the gradings we defined on $\cK$ in Section~\ref{sec:gradings-K}, we can also define gradings on $\cK^!$. The Maslov and Alexander gradings are defined as follows:
 \[
 (\gr_{\ws},\gr_{\zs})(1)=(0,0), \quad (\gr_{\ws},\gr_{\zs})(w )=(1,-1), \quad (\gr_{\ws},\gr_{\zs})(z )=(-1,1).
 \]
 We declare
 \[
(\gr,A)(s)=(\gr,A)(t)=(-1,0),\quad  (\gr,A)(\theta )=(1,0)\quad (\gr,A)(\varphi_{\pm})=(-1,\mp 1).
 \]
 
 We write $g(a)$ for $(\gr_{\ws},\gr_{\zs})(a)$ if $a\in \ve{I}_0\cdot \cK^!\cdot \ve{I}_0$ and $(\gr(a),A(a))$ if $a$ is in $\ve{I}_1\cdot \cK^!\cdot \ve{I}_1$ or $\ve{I}_0\cdot \cK^!\cdot \ve{I}_1$.  The following is straightforward to check:
 \begin{lem} If $a_1,\dots, a_n\in \cK^!$ are homogeneously graded, then
 \[
 g(\mu_n(a_1,\dots, a_n))=(n-2)\lambda+\sum_{i=1}^n g(a_i).
 \]
 Here, $g(a_i)$ denotes $(\gr_{\ws},\gr_{\zs})(a_i)$ or $(\gr,A)(a_i)$, depending on idempotent. The meaning of addition of grading is the same as in the same way as in Lemma~\ref{lem:grading-K}. Furthermore, $\lambda$ denotes $(1,0)$ if the right-most idempotent is $1$, and $(1,1)$ if the right-most idempotent is $0$.
 \end{lem}

As with the gradings on $\cK$ in Section~\ref{sec:gradings-K}, we can repackage the above gradings on $\cK^!$ into group or groupoid valued gradings. We focus on the groupoid valued gradings by $\cG$. The gradings are
  \[
  g(w )=  \rmT_{\left(\begin{smallmatrix}\Sss{1}\\\Sss{-1} \end{smallmatrix}\right)} ,\quad  g(z )=\rmT_{\left(\begin{smallmatrix}\Sss{-1}\\\Sss{1} \end{smallmatrix}\right)},
  \]
  \[
  g(s)=g(\sigma)^{-1}*\lambda^{-1}, \quad g(t)=g(\tau)^{-1}*\lambda^{-1},\]
  \[
   g(\varphi_{\pm})=\rmT_{\left(\begin{smallmatrix}\Sss{-1}\\\Sss{\mp 1} \end{smallmatrix}\right)},\quad \text{and} \quad  g(\theta)=\lambda. 
  \]
  Note 
  \[
  g(\sigma)^{-1}=\begin{pmatrix} 1 &0\\1&-2 \end{pmatrix}\quad \text{and} \quad g(\tau)^{-1}=\begin{pmatrix} 1&2\\1&0 \end{pmatrix}.
  \]
 
 \begin{lem} 
 If $a_1,\dots, a_n\in \cK^!$ are homogeneously graded  then
 \[
 g(\mu_n(a_1,\dots, a_n))=\lambda^{n-2}* g(a_1)*\cdots *g(a_n). 
 \]
 \end{lem}
 \begin{proof} The proof is a straightforward computation and we leave the details to the reader.
 \end{proof}

It is also helpful to define another grading on $\cK^!$, which we call the \emph{algebraic grading} $\gr_{\alg}$. This takes values in $\Z^{\le 0}$.  On $\cK^!$, we define $\gr_{\alg}$ to  be minus the number of factors in a monomial. For example, we define
\[
\gr_{\alg}(1)=0, \quad \gr_{\alg}(w )=\gr_{\alg}(z )=\gr_{\alg}(s)=\gr_{\alg}(\varphi_+)=-1
\]
\[
\gr_{\alg}(s \theta )=-2\quad \text{and} \quad \gr_{\alg}(\varphi_+\varphi_-\varphi_+ \theta )=-4.
\]

At various points, it will be helpful to view $\cK$ as having an algebraic grading which is concentrated in grading 0.

\begin{lem} If $a_1,\dots, a_n\in \cK^!$ are homogeneously graded, then
\[
\gr_{\alg}(\mu_n(a_1,\dots, a_n))=n-2+\sum_{i=1}^n \gr_{\alg}(a_i). 
\]
\end{lem}
The proof of the above lemma is straightforward.

\subsection{Dualizing bimodules}

There is a canonical $DD$-bimodule ${}^{\cK^!}[\cCo]^{\cK}$, given by the following diagram:
\[
\begin{tikzcd}[row sep=2cm] i_0 \ar[loop right, "w |W"] \ar[loop left, "z |Z"] \ar[d, "s|\sigma", bend left] \ar[d, bend right, "t|\tau",swap] \ar[loop above, "\theta |U"]\\
i_1 \ar[loop left, "\varphi_+|T"] \ar[loop right, "\varphi_-|T^{-1}"] \ar[loop below, "\theta |U"]
\end{tikzcd}
\]

\begin{rem} Note that there is also a $DD$-bimodule ${}^{\cK^!_{\infty}}[\cCo]^{\cK}$, given by the following diagram:
\[
\begin{tikzcd}[row sep=2cm] i_0 \ar[loop right, "w |W"] \ar[loop left, "z |Z"] \ar[d, "s|\sigma", bend left] \ar[d, bend right, "t|\tau",swap]\\
i_1 \ar[loop left, "\varphi_+|T"] \ar[loop right, "\varphi_-|T^{-1}"] \ar[loop below, "\theta |U"]
\end{tikzcd}
\]
\end{rem}

We now define the bimodule
\[
{}_{\cK} \scL_{\cK^!}:={}_{\cK} \cK_{\cK} \boxtimes {}^{\cK} [\bar{\cCo}]^{\cK^!} \boxtimes {}_{\cK^!} \bar{\cK}^!_{\cK^!},
\]
which is our candidate for a quasi-inverse to ${}^{\cK^!}[\cCo]^{\cK}$. Compare \cite{LOTMorphisms}*{Lemma~8.7}. Here, ${}_{\cK^!} \bar{\cK}^!_{\cK^!}$ denotes the graded dual group $\Hom(\cK^!, \bF)$ (i.e. the $\bF$-linear span of linear functionals which supported in a single algebraic grading). We can naturally endow $\bar{\cK}^!$ with the structure of a type-$AA$ bimodule over $(\cK^!, \cK^!)$, dual to ${}_{\cK^!} \cK^!_{\cK^!}$ .  Here, ${}^{\cK}[\bar{\cCo}]^{\cK^!}$ denotes the opposite module to ${}^{\cK^!}[\cCo]^{\cK}$. For our purposes, this amounts to reversing all of the arrows in ${}^{\cK^!}[\cCo]^{\cK}$ and switching which side the algebra elements are on).

We now show the that ${}_{\cK} \scL_{\cK^!}$ has a smaller model:

\begin{lem}
\label{lem:Tr-construction} There is a strong deformation retraction of type-$AA$ bimodules between ${}_{\cK} \scL _{\cK^!}$ and some type-$AA$ bimodule 
whose underlying $(\ve{I},\ve{I})$-bimodule is $\ve{I}$. We define this type-$AA$ bimodule to be ${}_{\cK} [\cTr]_{\cK^!}$.
\end{lem}
\begin{proof} Our proof will use the homological perturbation lemma.  We will write 
\[
m_{0|1|0}=d+\a_{0|1|0},
\]
for some map $d$, which is described in Equations~\eqref{eq:map-d-idempotent-00}, ~\eqref{eq:map-d-idempotent-10} and ~\eqref{eq:map-d-idempotent-11}. Therein,  $R_0$ denotes $\bF[W,Z]$ and $R_1$ denotes $\bF[U,T,T^{-1}]$. Also, we use $|$ to denote tensor products.

In idempotent $(0,0)$, the map $d$ is shown below:
\begin{equation}
d=
\begin{tikzcd}[labels=description, column sep={1.5cm,between origins}, row sep=1.5cm]
\vdots&&\vdots&[1cm] \vdots&& \vdots
\\[-1.5cm]
R_0|(wzw\theta )^*
	\ar[d, "W\Pi_{W^{\ge0}}"]
	&&
R_0|(zwz \theta )^*
	\ar[d, "Z\Pi_{Z^{\ge0}}"]
	&
R_0|(wzw)^*
	\ar[d, "W\Pi_{W^{\ge0}}"]
	&&
R_0|(zwz)^*
	\ar[d, "Z\Pi_{Z^{\ge0}}"]
	\\
R_0|(wz \theta )^*
	\ar[d, "Z\Pi_{Z^{\ge0}}"]
	&&
R_0|(zw \theta )^*
	\ar[d, "W\Pi_{W^{\ge0}}"]
	& 
R_0|(wz)^*
	\ar[d, "Z\Pi_{Z^{\ge0}}"]
	&&
R_0|(zw)^*
	\ar[d, "W\Pi_{W^{\ge0}}"]
	\\
R_0|(w\theta )^*
	\ar[dr, "W\Pi_{W^{\ge0}}"]
	&&
R_0|(z \theta )^*
	\ar[dl, "Z\Pi_{Z^{\ge0}}"]
	&
R_0|(w)^*
	\ar[dr, "W\Pi_{W^{\ge0}}"]
	&&
R_0|(z)^*
	\ar[dl, "Z\Pi_{Z^{\ge0}}"]
\\
	&
R_0|(\theta )^*
	\ar[rrr, "U\Pi_{W^0Z^0U^{\ge 0}}"]
	&&&
R_0|(1)^*
\end{tikzcd}
\label{eq:map-d-idempotent-00}
\end{equation}
Here, $\Pi_{W^{\ge 0}}$ denotes the $\bF$-linear projection onto monomials of the form $W^iZ^j$ where $i\ge 0$. The map $\Pi_{W^0Z^0U^{\ge 0}}$ denotes the $\bF$-linear projection onto the span of monomials of the form $U^i$ for $i\ge 0$.

We define a homotopy $H$ on this idempotent via the formula shown below:
\[H=
\begin{tikzcd}[labels=description, column sep={1.5cm,between origins}, row sep=1.5cm]
\vdots&&\vdots&[1cm] \vdots&& \vdots
\\[-1.5cm]
R_0|(wzw\theta )^*
	\ar[from=d, "W^{-1}\Pi_{W^{>0}}"]
	&&
R_0|(zwz \theta )^*
	\ar[from=d, "Z^{-1}\Pi_{Z^{>0}}"]
	&
R_0|(wzw)^*
	\ar[from=d, "W^{-1}\Pi_{W^{>0}}"]
	&&
R_0|(zwz)^*
	\ar[from=d, "Z^{-1}\Pi_{Z^{>0}}"]
	\\
R_0|(wz\theta )^*
	\ar[from=d, "Z^{-1}\Pi_{Z^{>0}}"]
	&&
R_0|(zw \theta )^*
	\ar[from=d, "W^{-1}\Pi_{W^{>0}}"]
	& 
R_0|(wz)^*
	\ar[from=d, "Z^{-1}\Pi_{Z^{>0}}"]
	&&
R_0|(zw)^*
	\ar[from=d, "W^{-1}\Pi_{W^{>0}}"]
	\\
R_0|(w\theta )^*
	\ar[from=dr, "W^{-1}\Pi_{W^{>0}}"]
	&&
R_0|(z \theta )^*
	\ar[from=dl, "Z^{-1}\Pi_{Z^{>0}}"]
	&
R_0|(w)^*
	\ar[from=dr, "W^{-1}\Pi_{W^{>0}}"]
	&&
R_0|(z)^*
	\ar[from=dl, "Z^{-1}\Pi_{Z^{>0}}"]
\\
	&
R_0|(\theta )^*
	\ar[from=rrr, "U^{-1}\Pi_{W^0Z^0U^{>0}}"]
	&&&
R_0|(1)^*
\end{tikzcd}
\]
In the above, $\Pi_{W^0Z^0U^{>0}}$ denotes projection onto monomials of the form $U^i$ for $i>0$.

We now consider idempotent $(1,0)$. Here, our map $d$ takes the following form:
\begin{equation}
d=
\begin{tikzcd}[ column sep={1.3cm,between origins}, row sep=1cm]
\vdots&&\vdots&[1cm]&[.8cm]&[1cm] \vdots&& \vdots
\\[-1.1cm]
R_1 \sigma|(wzwz \theta)^*
	\ar[d, "T"]
	&&
R_1 \sigma|(zwzw \theta)^*
	&&&
R_1 \sigma|(zwzw)^*
	&&
R_1 \sigma|(wzwz)^*
	\ar[d, "T"]
\\
R_1 \sigma|(wzw \theta)^*
	&&
R_1 \sigma|(zwz \theta)^*
	\ar[d, "T"]
	&&&
R_1 \sigma|(zwz)^*
	\ar[d, "T"]
	&&
R_1 \sigma|(wzw)^*
	\\
R_1 \sigma|(wz\theta)^*
	\ar[d, "T"]
	&&
R_1 \sigma|(zw\theta )^*
	& &&
R_1 \sigma|(zw)^*
	&&
R_1 \sigma|(wz)^*
	\ar[d, "T"]
	\\
R_1 \sigma|(w\theta)^*
	&&
R_1 \sigma|(z\theta)^*
&
R_1|(zs\theta)^*
	\ar[l, "1"]
& 
R_1|(zs)^*
	\ar[r, "1"]
&
R_1 \sigma|(z)^*
	&&
R_1 \sigma|(w)^*
\\
	&
R_1 \sigma|(\theta)^*
	&&
R_1|(s\theta)^*
	\ar[ll,"1"]
&
R_1|(s)^*
	\ar[rr, "1"]
&&
R_1 \sigma|(1)^*
\end{tikzcd}
\label{eq:map-d-idempotent-10}
\end{equation}
We define our homotopy $H$ as in the following diagram:
\[
H=
\begin{tikzcd}[ column sep={1.3cm,between origins}, row sep=1cm]
\vdots&&\vdots&[1cm]&[.8cm]&[1cm] \vdots&& \vdots
\\[-1.1cm]
R_1 \sigma|(wzwz \theta)^*
	\ar[from=d, "T^{-1}"]
	&&
R_1 \sigma|(zwzw \theta)^*
	&&&
R_1 \sigma|(zwzw)^*
	&&
R_1 \sigma|(wzwz)^*
	\ar[from=d, "T^{-1}"]
\\
R_1 \sigma|(wzw \theta)^*
	&&
R_1 \sigma|(zwz \theta)^*
	\ar[from=d, "T^{-1}"]
	&&&
R_1 \sigma|(zwz)^*
	\ar[from=d, "T^{-1}"]
	&&
R_1 \sigma|(wzw)^*
	\\
R_1 \sigma|(wz\theta)^*
	\ar[from=d, "T^{-1}"]
	&&
R_1 \sigma|(zw\theta )^*
	& &&
R_1 \sigma|(zw)^*
	&&
R_1 \sigma|(wz)^*
	\ar[from=d, "T^{-1}"]
	\\
R_1 \sigma|(w\theta)^*
	&&
R_1 \sigma|(z\theta)^*
&
R_1|(zs\theta)^*
	\ar[from=l, "1"]
& 
R_1|(zs)^*
	\ar[from=r, "1"]
&
R_1 \sigma|(z)^*
	&&
R_1 \sigma|(w)^*
\\
	&
R_1 \sigma|(\theta)^*
	&&
R_1|(s\theta)^*
	\ar[from=ll,"1"]
&
R_1|(s)^*
	\ar[from=rr, "1"]
&&
R_1 \sigma|(1)^*
\end{tikzcd}
\]

We define $d$ and $H$ symmetrically on the subspace of idempotent $(1,0)$ which are multiples of $\tau$ or $t$.

Finally, we consider idempotent $(1,1)$. We define $d$ as in the following diagram
\begin{equation}
d=
\begin{tikzcd}[labels=description, column sep={1.8cm,between origins}, row sep=1.7cm]
\vdots&&\vdots&[1cm] \vdots&& \vdots
\\[-1.5cm]
R_1|(\varphi_- \varphi_+\varphi_-\theta )^*
	\ar[d, "T^{-1}\Pi_{T^{\le 0}}"]
	\ar[rrr, "U", bend left=15,crossing over]
	&&
R_1|(\varphi_+\varphi_-\varphi_+\theta )^*
	\ar[d, "T\Pi_{T^{\ge0}}"]
	&
R_1|(\varphi_- \varphi_+\varphi_-)^*
	\ar[d, "T^{-1}\Pi_{T^{\le 0}}"]
	&&
R_1|(\varphi_+\varphi_-\varphi_+)^*
	\ar[d, "T\Pi_{T^{\ge0}}"]
	\ar[from=lll, "U", bend right=15,crossing over]
	\\
R_1|(\varphi_-\varphi_+ \theta )^*
	\ar[d, "T\Pi_{T^{\ge 0}}"]
	&&
R_1|(\varphi_+\varphi_-\theta )^*
	\ar[d, "T^{-1}\Pi_{T^{\le 0}}"]
	& 
R_1|(\varphi_-\varphi_+)^*
	\ar[d, "T\Pi_{T^{\ge0}}"]
		\ar[from=lll, "U", bend left=15,crossing over]
	&&
R_1|(\varphi_+\varphi_-)^*
	\ar[d, "T^{-1}\Pi_{T^{\le 0}}"]
	\ar[from=lll, "U", bend right=15,crossing over]
	\\
R_1|(\varphi_-\theta )^*
	\ar[dr, "T^{-1}\Pi_{T^{\le 0}}"]
	\ar[rrr, "U", bend left=15,crossing over]
	&&
R_1|(\varphi_+ \theta )^*
	\ar[dl, "T\Pi_{T^{\ge0}}"]
	&
R_1|(\varphi_-)^*
	\ar[dr, "T^{-1}\Pi_{T^{\le 0}}"]
	&&
R_1|(\varphi_+)^*
	\ar[dl, "T\Pi_{T^{\ge0}}"]
	\ar[from=lll, "U", bend right=15,crossing over, pos=.65]
\\
	&
R_1|(\theta )^*
	\ar[rrr, "U"]
	&&&
R_1|(1)^*
\end{tikzcd}
\label{eq:map-d-idempotent-11}
\end{equation}
We define $H$ via the following diagram:
\[
H=
\begin{tikzcd}[labels=description, column sep={1.7cm,between origins}, row sep=1.5cm]
\vdots&&\vdots&[1cm] \vdots&& \vdots
\\[-1.5cm]
R_1|(\varphi_- \varphi_+\varphi_-\theta )^*
	\ar[from=d, "T\Pi_{T^{< 0}}"]
	&&
R_1|(\varphi_+\varphi_-\varphi_+\theta )^*
	\ar[from=d, "T^{-1}\Pi_{T^{> 0}}"]
	&
R_1|(\varphi_- \varphi_+\varphi_-)^*
	\ar[from=d, "T\Pi_{T^{< 0}}"]
	&&
R_1|(\varphi_+\varphi_-\varphi_+)^*
	\ar[from=d, "T^{-1}\Pi_{T^{>0}}"]
	\\
R_1|(\varphi_-\varphi_+ \theta )^*
	\ar[from=d, "T^{-1}\Pi_{T^{> 0}}"]
	&&
R_1|(\varphi_+\varphi_-\theta )^*
	\ar[from=d, "T\Pi_{T^{< 0}}"]
	& 
R_1|(\varphi_-\varphi_+)^*
	\ar[from=d, "T^{-1}\Pi_{T^{>0}}"]
	&&
R_1|(\varphi_+\varphi_-)^*
	\ar[from=d, "T\Pi_{T^{< 0}}"]
	\\
R_1|(\varphi_-\theta )^*
	\ar[from=dr, "T\Pi_{T^{< 0}}"]
	&&
R_1|(\varphi_+ \theta )^*
	\ar[from=dl, "T^{-1}\Pi_{T^{>0}}"]
	&
R_1|(\varphi_-)^*
	\ar[from=dr, "T\Pi_{T^{< 0}}"]
	&&
R_1|(\varphi_+)^*
	\ar[from=dl, "T^{-1}\Pi_{T^{>0}}"]
\\
	&
R_1|(\theta )^*
	\ar[from=rrr, "U^{-1} \Pi_{U^{>0}T^0}"]
	&&&
R_1|(1)^*
\end{tikzcd}
\]

Note that we can naturally view $(\scL, d)$ as a curved $AA$-bimodule by defining structure maps $d_{i|1|j}$ by setting $d_{0|1|0}=d$ and setting $d_{i|1|j}=0$ when $i\neq 0$ or $j\neq 0$. (Note that we cannot do this with the original differential $m_{0|1|0}$ from $\scL$ replacing $d_{0|1|0}$ because $m_{0|1|0}$ does not square to zero, due to the curvature). 

In this manner, we can view the bimodule $(\scL, m_{i|1|j})$ as a deformation of $(\scL, d_{0|1|0})$, i.e. we can write
\[
m_{*|1|*}=d_{0|1|0}+\a_{*|1|*}
\]
where $\a_{*|1|*}$ is a type-$AA$ endomorphism of $(\scL, d_{0|1|0})$ which satisfies the Mauer-Cartan equation
\begin{equation}
\d(a_{*|1|*})+a_{*|1|*}\circ a_{*|1|*}=0. \label{eq:Mauer-Cartan}
\end{equation}
Here $\d(\a_{*|1|*})$ denotes the morphism differential applied to $\a_{*|1|*}$, when we view $\a_{*|1|*}$ as an $AA$-bimodule endomorphism of $(\scL,d_{0|1|0})$. Note that Equation~\eqref{eq:Mauer-Cartan} is automatic from the fact that $(\scL,d_{0|1|0})$ and $(\scL,m_{*|1|*})$ satisfy the $AA$-bimodule structure equation.

Note that we can also view (trivially) the underlying $(\ve{I},\ve{I})$-bimodule of $\ve{I}$ as an $AA$-bimodule over $(\cK, \cK^!)$ by setting $m_{i,1,j}=0$ for all $i$ and $j$.

We claim that we have a strong deformation of type $AA$-bimodules over $(\cK,\cK^!)$
\[
\begin{tikzcd}
(\scL, d_{0|1|0})\ar[r,shift left, "\Pi"] \ar[loop left, "H"] & \ve{I} \ar[l, shift left, "I"].
\end{tikzcd}
\] 
The projection map $\Pi$ sends $1\otimes 1^*\in R_0\otimes \bar{\cK}^!$ to $i_0\in \ve{I}_0$. It sends $1\otimes 1^*\in R_1\otimes \bar{\cK}^!$ to $i_1\in \ve{I}_1$. The map $\Pi$ vanishes on all other monomials. The inclusion map $I$ sends $i_0$ and $i_1$ to the appropriate copies of $1\otimes 1^*$ (the monomials that $\Pi$ is non-vanishing on).

The relations 
\[
\Pi\circ I=\bI_{[\cTr]},\quad I\circ \Pi=\bI_{\scL}+\d(H),\quad H\circ H=0,\quad H\circ I=0,\quad \Pi\circ H=0
\]
are immediate. By Lemma~\ref{lem:HPL-twist}, it suffices to show that
\[
\bI_{\scL}+\a_{*|1|*}\circ H
\]
is an invertible endomorphism of $\scL$, with inverse equal to the infinite sum
\[
\sum_{i=0}^\infty (\a_{*|1|*}\circ H)^i.
\]
 We will show that the above sum involves only finitely many terms, when applied to a fixed collection of algebra elements from $\cK$ and $\cK^!$. To this end, we claim that
\begin{equation}
(\a_{0|1|0} \circ H)^N=0
\label{eq:aH^N=0}
\end{equation}
for all $N\gg 0$.  For conceptual purposes, it is helpful to observe that $\a_{0|1|0}$ can naturally be decomposes as a sum
\begin{equation}
\a_{0|1|0}=\a^{\mu_1}_{0|1|0}+\a^{\mu_2}_{0|1|0}
\label{eq:a010-decomposition}
\end{equation}
where $\a^{\mu_2}_{0|1|0}$ consists of the terms obtained from $\delta^{1,1}$ of $\bar{\cCo}$ and $\a^{\mu_1}_{0|1|0}$ consists of the terms obtained from $\mu_1$ on $\cK^!$. It is straightforward to check that
\[
H\circ \a^{\mu_2}_{0|1|0}\circ H=0.
\]
Additionally, one can check that
\[
H\circ \a^{\mu_1}_{0|1|0}\circ H\circ \a^{\mu_1}_{0|1|0} \circ H=0.
\]
Therefore, we may take $N=3$ in Equation~\eqref{eq:aH^N=0}.

Homological perturbation therefore equips the vector space $\ve{I}$ with the structure of a type-$AA$ bimodule, and this bimodule is homotopy equivalent to ${}_{\cK} \scL_{\cK^!}$. We define ${}_{\cK} [\cTr]_{\cK^!}$ to be the resulting type-$AA$ bimodule.
\end{proof}

When studying the actions on ${}_{\cK} [\cTr]_{\cK^!}$, the following lemma will be helpful:

\begin{lem}
\label{lem:a010-no-contribution}
Decompose $\a_{0|1|0}$ as $\a_{0|1|0}^{\mu_1}+\a_{0|1|0}^{\mu_2}$, as above.
\begin{enumerate}
\item The map $\a_{0|1|0}^{\mu_2}$ makes no contribution to any action $m_{i|1|j}$ on ${}_{\cK} [\cTr]_{\cK^!}$.
\item The only components of $\a_{0|1|0}^{\mu_1}$ which can contribute to an action $m_{i|1|j}$ are the components of the following form:
\begin{enumerate}
\item $R_0\otimes (wz)^*\to R_0\otimes \theta^*$.
\item $R_0\otimes (zw)^*\to R_0\otimes \theta^*$.
\item $R_1\sigma \otimes (wz)^*\to R_1\sigma\otimes \theta^*$.
\item $R_1 \sigma\otimes (zwz)^* \to R_1\sigma \otimes (z\theta)^*$
\item $R_1\tau \otimes (zw)^*\to R_1\tau \otimes \theta^*$
\item $R_1\tau \otimes (wzw)^* \to R_1\tau \otimes (w\theta)^*$.
\end{enumerate}
\end{enumerate}
\end{lem}
\begin{proof} For the first claim, the structure maps $m_{i|1|j}$ were given by the homological perturbation lemma. 
The claim is a consequence of the following three formulas, which are straightforward to verify:
\[
\begin{split}H\circ \a_{0|1|0}^{\mu_2}\circ H&=0\\
\Pi\circ \a_{0|1|0}^{\mu_2}\circ H&=0\\
\a_{0|1|0}^{\mu_2}\circ I&=0.
\end{split}
\]
For the second claim, we observe that
\[
\Pi\circ \a_{0|1|0}^{\mu_1}=0,\qquad \quad \a_{0|1|0}^{\mu_1}\circ I=0,
\]
and that the six listed terms are the only components of $\a_{0|1|0}^{\mu_1}$ which contribute non-trivially to $H\circ \a_{0|1|0}^{\mu_1}\circ H$ is non-vanishing.
\end{proof}

We now discuss some grading properties of the bimodules ${}_{\cK} [\cTr]_{\cK^!}$ and ${}^{\cK^!} [\cCo]^{\cK}$.  We begin by defining an algebraic grading on ${}_{\cK} \scL_{\cK^!}$ by setting $\gr_{\alg}(x^*)=-\gr_{\alg}(x)$, for $x\in \cK^!$. For example, we declare $W^i Z^j \otimes (w)^*$ to have algebraic grading $1$. 

\begin{lem}
\label{lem:grading-shifts}
\item
\begin{enumerate}
\item The operations $\mu_n$ on $\cK^!$ satisfy
\[
\gr_{\alg}(\mu_n(a_1,\dots, a_n))=n-2+\sum_{i=1}^n \gr_{\alg}(a_i).
\]
\item If we view the underlying vector space of $[\cCo]$ as being concentrated in algebraic grading 0, then the structure map on ${}^{\cK^!} [\cCo]^\cK$ lowers algebraic grading by 1.
\item If we view $\cK$ as being supported in algebraic grading 0, then the actions $m_{i|1|j}$ on ${}_{\cK} \scL_{\cK^!}$ and ${}_{\cK} [\cTr]_{\cK^!}$ satisfy
\[
\begin{split}
\gr_{\alg}(m_{k|1|n}(a_k,\dots, a_1, x, b_1,\dots, b_n))&=n+k-1+\sum_{i=1}^k \gr_{\alg}(a_i)+\sum_{i=1}^n \gr_{\alg}(b_i)
\end{split}
\]
\item The structure maps on ${}^{\cK^!} [\cCo]^{\cK} \boxtimes {}_{\cK} [\cTr]_{\cK^!}$ satisfy 
\[
\gr_{\alg}(\delta_{n+1}^1(1,b_1,\dots, b_n))=n-1+ \sum_{i=1}^n \gr_{\alg}(b_i).
\]
\item If we view ${}^{\cK^!}[\cCo]^{\cK}$ as being graded by the $\cG$-set $\cG$ (and concentrated in grading $\id$) then $\delta^{1,1}$ has grading $\lambda^{-1}$. 
\item If we view ${}_{\cK} [\cTr]_{\cK^!}$ as being graded by the $\cG$-set $\cG$ (and concentrated in grading $\id$), then if $m_{i|1|j}(a_i,\dots, a_1, 1, b_1,\dots, b_j)\neq 0$, then
\[
g(a_i)*\cdots * g(a_1)*g(b_1)*\cdots *g(b_j)*\lambda^{i+j-1}=\id. 
\]
\end{enumerate}
\end{lem}
\begin{rem} We remind the reader that we view $\cK$ as being supported in algebraic grading 0. Therefore the third claim can be rephrased as
\[
\gr_{\alg}(m_{k|1|n}(a_k,\dots, a_1, x, b_1,\dots, b_n))=n+k-1+\sum_{i=1}^n \gr_{\alg}(b_i).
\]
\end{rem}
\begin{proof} Most of the claims are immediate from the definitions. As an example, we illustrate the claim about the algebraic grading and ${}_{\cK} [\cTr]_{\cK^!}$, since the structure maps on this bimodule are somewhat complicated. We note that the homotopy $H$ constructed in Lemma~\ref{lem:Tr-construction} increases the algebraic grading by $1$. The type-$AA$ structure maps constructed via the homological perturbation lemma are sums of the form
\[
\Pi \circ \a_{i_N|1|j_N}\circ H\circ \a_{i_{N-1}|1|j_{N-1}}\circ \cdots\circ H\circ \a_{i_1|1|j_1}\circ I.
\]
Note that the only $\a_{i|1|j}$ which are non-zero are $\a_{1|1|0}$, $\a_{0|1|1}$ and $\a_{0|1|0}$. Note that $\a_{0|1|0}$ decreases the algebraic grading by 1, while $H$ increases the algebraic grading by $1$. Furthermore $\a_{0|1|1}(-,b)$ shifts the algebraic grading by $\gr_{\alg}(b)$. Similarly $\a_{1|1|0}(a,-)$ preserves the algebraic grading (or equivalently, shifts it by $0=\gr_{\alg}(a)$). The total shift will therefore be
\[
n+k-1+\sum_{i=1}^n \gr_{\alg}(b_i),
\]
as claimed. Similar arguments hold for all of the other claims. 
\end{proof}

We now consider boundedness of the bimodule ${}_{\cK} [\cTr]_{\cK^!}$.
Recall that we say a type-$A$ module ${}_{\cA} M$ is \emph{bounded} if there is an $n_0>0$ so that $m_{n+1}=0$ if $n>n_0$. Similarly, we say a type-$AA$ bimodule is \emph{bounded} if there is an $n_0>0$ so that $m_{i|1|j}=0$ if $i+j>n_0$. For our module categories, it will be helpful to consider a stronger condition, defined in \cite{LOTDiagonals}*{Definition~3.35}.

\begin{define}\label{def:bonsai}
 We say that an $A_\infty$-module  ${}_{\cA} M$ is \emph{bonsai} if there is some $n_0$ so that the structure map $m_{\cT}$ vanishes for all ordinary $A_\infty$-module input trees $\cT$ with $\deg(\cT)\ge n_0$. (See Definition~\ref{def:input-trees} for the definition of a ordinary $A_\infty$-module input tree).
\end{define}

 Note that ${}_{\cK} [\cTr]_{\cK^!}$ is not bounded (and hence not bonsai). For example
\[
m_{1|1|n}(T^n, 1,\underbrace{\varphi_+,\dots, \varphi_+}_n)\neq 0
\]
for all $n\ge 1$.  

We will also find it useful to consider a weaker notion  for bimodules:

\begin{define} We say a bimodule ${}_{\cA} M_{\cB}$ is \emph{right bonsai} if, for each sequence of algebra elements $a_1,\dots, a_n\in \cA$, the structure actions $m_{n|1|k}$ satisfy the bonsai condition when the algebra input sequence for $\cA$ is $(a_n,\dots, a_1)$. We say ${}_{\cA} M_{\cB}$ is \emph{left bonsai} if the analog holds with the roles of $\cA$ and $\cB$ reversed.
\end{define}

\begin{lem}
\label{lem:Tr-LR-bonsai}
 The bimodule ${}_{\cK} [\cTr]_{\cK^!}$ is left and right bonsai.
\end{lem}
\begin{proof} The left bonsai condition follows from the fact that the structure maps respect the algebraic grading by Lemma~\ref{lem:grading-shifts}, and the fact that $\cK$ is supported in algebraic grading $0$. 

The right bonsai condition is more complicated. Firstly, note that by direct computation, the structure maps of ${}_{\cK} [\cTr]_{\cK^!}$ satisfy $m_{0|1|j}=0$ for all $j\ge 0$. Using also the fact that $\cK^!$ has no operations $\mu_i$ for $i\ge 2$, it suffices to show that the module ${}_{\cK} [\cTr]_{\cK^!}$ is \emph{right-bounded}, in the sense that if $a_1,\dots, a_j$ are fixed, then
\[
m_{j|1|k}(a_j,\dots, a_1,1,b_1,\dots, b_k)=0
\]
if $k$ is sufficiently large.

 By Lemma~\ref{lem:grading-shifts}, we have that if $m_{j|1|k}(a_j,\dots, a_1,1,b_1,\dots, b_k)\neq 0$, then 
\[
j-1+\sum_{i=1}^k (\gr_{\alg}(b_i)+1)=0.
\]
Note that if some $b_i\in \ve{I}$, then $m_{j|1|k}(a_j,\dots, a_1,1,b_1,\dots, b_k)=0$ unless $j=0$ and $k=1$. Therefore, if $j\neq 0$, then at most $j-1$ of the $b_i$ can have algebraic grading less than $-1$ if the action is non-trivial.

The actions on ${}_{\cK}[\cTr]_{\cK^!}$ are defined via the homological perturbation lemma. We observe that on ${}_{\cK} \scL_{\cK^!}$, only $\a_{0|1|0}$, $\a_{1|1|0}$ and $\a_{0|1|1}$ are non-trivial. Recalling the decomposition of $\a_{0|1|0}=\a_{0|1|0}^{\mu_1}+\a_{0|1|0}^{\mu_2}$ from Equation~\eqref{eq:a010-decomposition}, we observe that the component $\a_{0|1|0}^{\mu_2}$ never contributes to the actions on ${}_{\cK} [\cTr]_{\cK^!}$ by Lemma~\ref{lem:a010-no-contribution}. 
 Therefore, it suffices to show that if $\xs\in \scL$, then there is some $N_{\xs}>0$ such that if $\a_1,\dots, \a_{N_{\xs}}$ is a sequence of maps which are each of the form $\a_{0|1|0}^{\mu_1}$ or $m_{0|1|1}(-, b_i)$ where $b_i\in \cK^!$ has algebraic grading $-1$, then
\begin{equation}
\left(H\circ \a_{N_{\xs}}\circ H\circ \cdots \circ H\circ \a_1 \circ H\right)(\xs)=0.
\label{eq:HPL-sequence}
\end{equation}
Note that since at most one of $b_1,\dots, b_k$ can be a multiple of $s$ or $t$ if 
\[
m_{\cT}(a_1,\dots, a_j,1,b_1,\dots, b_k)\neq 0,
\]
 it suffices only to consider the case that all $b_i$ are in $\ve{I}_0 \cdot \cK^!\cdot \ve{I}_0$ or all are in $\ve{I}_1\cdot \cK^! \cdot \ve{I}_1$.  We may furthermore break the argument into cases, depending on the idempotent of the generator $\xs\in \scL$.

We consider first the case that $\xs\in \ve{I}_0\cdot \scL\cdot \ve{I}_0$. In this case, we write $\xs= \zs\otimes \ys^*$ and consider $\deg(\zs)$, defined by $\deg(W^iZ^j)=i+j$. Note that all components of $H$ decrease the quantity $\deg(\zs)$ in this idempotent. Furthermore, $\a_{0|1|0}^{\mu_1}$ and $m_{0|1|1}(-,a)$ always preserve $\deg(\zs)$. Therefore any sequence as in Equation~\eqref{eq:HPL-sequence} with generators and algebra elements in this idempotent and with more than $\deg(\zs)$ factors of $H$ will vanish, which proves the claim in this case.

Next, we consider the case that $\xs\in \ve{I}_1\cdot \scL\cdot \ve{I}_1$. In this case, the argument is similar. We write $\xs=\zs\otimes \ys^*$ and consider $\deg(\zs)$ defined by $\deg(U^i T^j)=i+|j|$. The map $\a_{0|1|0}^{\mu_1}$ vanishes on this idempotent, the homotopy $H$ decreases this quantity, and the actions $m_{0|1|1}(-,b_i)$ preserve this quantity, so a sequence in Equation~\eqref{eq:HPL-sequence} will be trivial when applied to $\xs$ if it has more than $\deg(\zs)$ factors of $H$.

Finally, we consider the case that $\xs\in \ve{I}_1\cdot \scL\cdot \ve{I}_0$. This argument is different than the others. We first observe that if $\xs\in \ve{I}_1\cdot \scL\cdot \ve{I}_0$, then
\[
(H \circ m_2(-,b)\circ H)(\xs)=0
\]
for any $b\in \{w,z,\theta\}$. Therefore, if
\[
(H\circ \a_{N} \circ H\cdots \circ \a_1\circ H)(\xs)\neq 0
\]
and all $\a_i$ are either $\a_{0|1|0}^{\mu_1}$ or $m_{0|1|1}(-,b_i)$ for $b_i$ with $\gr_{\alg}(b_i)=-1$, then $\a_i=\a_{0|1|0}^{\mu_1}$ for all $i$. However one computes easily that
\[
H\circ \a_{0|1|0}^{\mu_1}\circ H\circ \a_{0|1|0}^{\mu_1}\circ H=0,
\]
 which proves the claim.
 \end{proof}

Finally, we prove one additional result:

\begin{lem}
\label{lem:strictly-unital} The bimodule ${}_{\cK} [\cTr]_{\cK^!}$ is strictly unital.
\end{lem}
\begin{proof} The computation that $m_{1|1|0}(\xs,1)=m_{0|1|1}(1,\xs)=\xs$ follows from the fact that this relation holds on $\scL$, and the fact that $H\circ I$, $\Pi\circ H$ and $H\circ H$ all vanish. The second fact follows from similar logic.
\end{proof}

\subsection{Koszul duality}

We now prove that $\cK^!$ and $\cK$ are Koszul dual. This amounts to the following result:

\begin{thm} We have
\[
{}_{\cK}[\cTr]_{\cK^!}\boxtimes {}^{\cK^!} [\cCo]^{\cK}={}_{\cK} \bI^{\cK}, \quad \text{and} \quad {}^{\cK^!} [\cCo]^{\cK}\boxtimes {}_{\cK}[\cTr]_{\cK^!}={}^{\cK^!} \bI_{\cK^!}.
\]
\end{thm}

We will break the proof into several steps.

\begin{prop}
\label{prop:quasi-inverse-1}
The bimodule ${}_{\cK} [\cTr]_{\cK^!} \boxtimes {}^{\cK^!} [\cCo]^{\cK}$ is equal to ${}_{\cK} [\bI]^{\cK}$.
\end{prop}
\begin{proof}
We make the following claims:
 \begin{enumerate}
\item $\delta_2^1(W,i_0)=i_0\otimes W$, $\delta_{2}^1(Z,i_0)=i_0\otimes Z$, $\delta_2^1(\sigma, i_0)=i_1\otimes \sigma$, $\delta_2^1(\tau, i_0)=i_1\otimes \tau$, $\delta_2^1(T^{\pm 1},i_1)=i_1\otimes T^{\pm 1}$ and $\delta_2^1(U,i_1)=i_1\otimes U$.
\item $\delta_{j+1}^1=0$ for $j\neq 1$.
\end{enumerate}
Together, the main result follows from the above two claims and the $DA$-bimodule structure relations.

 The first claim is a direct computation, which we leave to the reader, as it follows immediately from the recipe from the homological perturbation lemma.
 
 For the second claim, we use the algebraic grading. The grading shifts from Lemma~\ref{lem:grading-shifts} imply that the algebraic grading shift of $\delta_{j+1}^1$ is $j-1$. Since $\cK$ is concentrated in algebraic grading 0, this implies that $\delta_{j+1}^1$  can only be non-zero if $j=1$, which completes the proof.
 \end{proof}

We will prove the following:

\begin{prop}\label{prop:quasi-inverse-2} The bimodules ${}^{\cK^!} [\cCo]^{\cK}\boxtimes {}_{\cK} [\cTr]_{\cK^!}$ and $ {}^{\cK^!}\bI_{\cK^!}$ are equal. 
\end{prop}

Our proof of Proposition~\ref{prop:quasi-inverse-2} is slightly complicated since grading considerations alone do not force the two bimodules to be equal, and we have to do more computations. We will break the argument into cases, depending on the idempotents of the algebra inputs.

\begin{lem} 
\label{lem:quasi-inverse-idempotent-00}
The bimodule ${}^{\cK^!} [\cCo]^{\cK}\boxtimes {}_{\cK} [\cTr]_{\cK^!}$ coincides with ${}^{\cK^!} \bI_{\cK^!}$ when restricted to algebra elements in $\ve{I}_0\cdot \cK^!\cdot \ve{I}_0$.
\end{lem}

\begin{proof}
We make the following claims:
\begin{enumerate}
\item $\delta_2^1(i_0,w)=w\otimes i_0$, $\delta_2^1(i_0,z)=z\otimes i_0$, and $\delta_2^1(i_0, \theta)=\theta\otimes i_0$.
\item $\delta_{j+1}^1(i_0,a_1,\dots, a_j)=0$ if all $a_i\in \ve{I}_0\cdot \cK^!\cdot \ve{I}_0$ and $j\neq 1$.
\end{enumerate}
Note that the main claim easily follows from the above two claims using the $DA$-bimodule relations.

The first claim is a straightforward computation, which we leave to the reader. The second claim is more involved. 

 Throughout the proof, it is helpful to write
\[
R_0=\ve{I}_0\cdot \cK\cdot \ve{I}_0 \quad \text{and} \quad R_0^!=\ve{I}_0\cdot \cK^!\cdot \ve{I}_0.
\]

We first observe that $\delta_1^1=0$ by Lemma~\ref{lem:a010-no-contribution}.

We will now show that
\[
\delta_{j+1}^1(i_0,a_1,\dots, a_j)=0
\]
for $j>1$, when $a_k\in R_0^!$ for all $k$. We define a filtration $\cF_{i}$ on $ \ve{I}_0\cdot {}_{\cK}\scL_{\cK^!}\cdot \ve{I}_0$ as follows. We first define a filtration $\cF_i$ on $R_0$ by declaring $W^iZ^j\in R_0\subset \cK$ to have filtration  
\[
\cF(W^iZ^j)=\min(i,j)+|j-i|.
\]
We can alternatively define $\cF(a)$ of a monomial $a\in R_0$ to be the minimal $n$ so that $a=a_1\cdots a_n$ for some $a_i\in \{W,Z, U\}$. 
Note that 
\[
\cF(ab)\le \cF(a)+\cF(b)
\]
for $a,b\in R_0$.
  
 Write $\cF_i=\Span_{\bF}\{a\in R_0: \cF(a)\le i\}$. We extend this filtration to tensor products tensorially, i.e.
\[
\cF(a_1\otimes \cdots \otimes a_n)=\sum_{i=1}^n \cF(a_i).
\]
By viewing $\scL$ as a tensor product of $\cK$ and $\bar{ \cK}^!$ (with $\bar{\cK}^!$ concentrated in filtration level 0), we therefore can view the bimodule $\scL$ as being filtered by $\cF$ as well. We observe that our homotopy $H$ maps $\cF_{i}$ to $\cF_{i-1}$. Furthermore, we observe that the map $\a_{0|1|0}$ maps $\cF_i$ to $\cF_{i+1}$, and $\a_{0|1|1}$ and $\a_{1|1|0}$ send $\cF_{i}\otimes \cF_j$ to $\cF_{i+j}$. It follows from description of the actions in the homological perturbation lemma that the structure maps $m_{i|1|j}$ on ${}_{\cK} [\cTr]_{\cK^!}$ satisfy 
\begin{equation}
m_{i|1|j}\left(\cF_n(R_0^{\otimes i} \otimes \ve{I} \otimes (R_0^!)^{\otimes j})\right)\subset \cF_{n-i-j+1}(\ve{I})
\label{eq:filtration-Fn-mi1j}
\end{equation}
However this space is empty if $n-i-j+1<0$. Note that the differential of the tensor product ${}^{\cK^!} [\cCo]^{\cK}\boxtimes {}_{\cK} [\cTr]_{\cK^!}$, when restricted to generators and algebra elements in idempotent $(0,0)$, is obtained by summing over $i$ diagrams as shown below:
\[
\begin{tikzcd}& {\cCo}\ar[d]&\cTr\ar[dd]& (R_0^!)^{\otimes j}\ar[ddl,Rightarrow,bend left=10]\\
&\delta^{i,i}\ar[dl,Rightarrow,bend right=10] \ar[dr, Rightarrow, bend left=10]\ar[dd]&\, \\
\Pi \ar[d]&\,&m_{i|1|j} \ar[d]\\
\,&\,&\,
\end{tikzcd}
\]
Here, $\Pi$ means to iteratively apply $\mu_2$. Each algebra element $a_k\in \{a_1,\dots, a_i\}\subset R_0$ which is input into $m_{i|1|j}$ in the above diagram has $\cF(a_k)=1$. Therefore $\cF(a_1\otimes \cdots \otimes a_i)=i$. However, combining this with Equation~\eqref{eq:filtration-Fn-mi1j}, we see that 
\[
\delta_{j+1}^1\big(\ve{I}_0\otimes (R_0^!)^{\otimes j}\big)=\delta_{j+1}^1\big(\cF_{\le 0}(\ve{I}_0\otimes (R_0^!)^{\otimes j})\big)\subset \cF_{1-j}.
\]
If $j>1$, $\cF_{1-j}$ is zero by definition. Therefore $\delta_{j+1}^1=0$ if $j>1$ and all inputs are from $\ve{I}_0\cdot \cK^! \cdot \ve{I}_0$. This completes the proof. 
\end{proof}

Next, we move onto idempotent 1:

\begin{lem} 
\label{lem:quasi-inverse-idempotent-11}
The structure maps of the bimodule ${}^{\cK^!} [\cCo]^{\cK}\boxtimes {}_{\cK} [\cTr]_{\cK^!}$ coincide with those from ${}^{\cK^!} \bI_{\cK^!}$ when restricted to algebra elements in $\ve{I}_1\cdot \cK^!\cdot \ve{I}_1$.
\end{lem}
\begin{proof}Write $R_1=\ve{I}_1\cdot \cK\cdot \ve{I}_1.$  The proof is essentially the same as the proof of Lemma~\ref{lem:quasi-inverse-idempotent-00}. 
One computes directly that
\begin{equation}
\delta_2^1(i_1, \varphi_{\pm})=\varphi_{\pm}\otimes i_1 \quad \text{and} \quad \delta_2^1(i_1,\theta)=\theta\otimes i_1.
\label{eq:delta-2-1-idemponte-1-1}
\end{equation}
Next, one shows that $\delta_{j+1}^1=0$ for $j\neq 1$. Using the $DA$-bimodule relations, the main claim follows easily.

To show that $\delta_{j+1}^1=0$ for $j\neq 1$, one argues essentially identically to the argument in Lemma~\ref{lem:quasi-inverse-idempotent-00}. The only difference is that the filtration $\cF$ on $R_1$ is defined by the formula
\[
\cF(U^iT^j)=i+|j|.
\]
From here, the argument is formally identical to the argument in Lemma~\ref{lem:quasi-inverse-idempotent-00}. 
\end{proof}

As a last step, we compute the structure maps of ${}^{\cK^!} [\cCo]^{\cK}\boxtimes {}_{\cK} [\cT_r]_{\cK^!}$ when one of the inputs is a multiple of $s$ or $t$:

\begin{lem} If $b_1\otimes \cdots \otimes b_n\in \ve{I}_0\cdot (\otimes^{n} \cK^!)\cdot \ve{I}_1$ is a non-zero tensor of monomials, then
\[
\delta_{n+1}^1(1, b_1,\dots, b_n)=\begin{cases}b_1 & \text{ if } n=1\\
0& \text{ if } n\neq 1. \end{cases}
\] 
\end{lem}
\begin{proof}
It seems challenging to give an analog of the filtration $\cF$ appearing in Lemmas~\ref{lem:quasi-inverse-idempotent-00} and ~\ref{lem:quasi-inverse-idempotent-11}. Therefore,  our proof is by direct computation.  We will organize the proof as follows. For each monomial $b\in \{s, z s, \theta s, z\theta s,t, wt, \theta t, w\theta t\}$, we will compute all of the possible sequences $b_1,\dots, b_n\in \cK^!$ such that
\[
\delta_{n+1}^1(1,b_1,\dots, b_n)=b.
\]  
Since the modules and homotopy $H$ are symmetric between $s$ and $t$, it suffices to consider only $b\in \{s,zs, \theta s, z\theta s\}$.

 We now describe how to do these computations practically. Firstly, we enumerate all ways of factoring $b$ into monomials which are output by the $DD$-bimodule ${}^{\cK^!} [\cCo]^{\cK}$. For each factorization of $b$, we obtain a corresponding sequence $a_1,\dots, a_k\in \cK$ of dual monomials. It suffices, for each of these sequences $a_1,\dots, a_k$, to find all sequences $b_1,\dots, b_n\in \cK^!$ such that $m_{k|1|n}(a_k,\dots, a_1,1,b_1,\dots, b_n)$ is non-zero. To find all such $b_1,\dots, b_n$, we draw a planar tree with one input and several outputs. The input is always $I(i_0)=1|1^*\in \ve{I}_0\cdot \scL\cdot \ve{I}_0$.  Each edge of the tree illustrates a possible application of $m_{1|1|0}(a_i,-)$, $m_{0|1|1}(-,b_i)$, $\a_{0|1|0}^{\mu_1}$ or $H$. We alternate between actions $m_{0|1|1}$, $m_{0|1|1}$ and $\a_{0|1|0}^{\mu_1}$ of $\scL$ and the homotopy $H$  as we move downward in the tree, as per the homological perturbation lemma. At each vertex, the downward pointing arrows enumerate all possible non-vanishing actions on the element of $\scL$ at a given vertex. The outputs of the tree (the lowest vertices) are labeled with $0$ or $1|1^*$. We label a vertex of minimal height by $1|1^*$ if and only if the corresponding path from the input contributes non-trivially to $m_{k|1|n}$.
 
 We begin by considering $b=s$. This has a unique factorization $b=s$, so $a_1=\sigma$, and we are interested in computing sequences $b_1,\dots, b_n$ such that $m_{1|1|n}(\sigma,1,b_1,\dots, b_n)$ is non-trivial. The tree in this case is quite simple, namely it is
 \[
  \begin{tikzcd}[row sep=.4cm]
 1|1^* 
 	\ar[d, "\sigma\cdot"]
 \\
 \sigma|1^*
	 \ar[d, "H"]
	 \\
 1|s^* 
 	\ar[d, "\cdot s"]
 	\\
 1|1^*.
 \end{tikzcd}
 \]
 In the above diagram, we write $a\cdot$ for $m_{1|1|0}(a,-)$ and we write $\cdot b$ for $m_{0|1|1}(-,b)$.
 
 We now consider $b=\theta s$. There are two factorizations of $b$, namely $b=\theta\cdot s$ and $b=s \cdot \theta$. This gives corresponding sequences $(U,\sigma)$ and $(\sigma,U)$ for $(a_2,a_1)$.
 The corresponding trees are as follows:
 \[
(\sigma,U):
 \begin{tikzcd}
 & 1|1^* \ar[d, "U\cdot"]\\
 & U|1^*\ar[d, "H"]&\\
 &1|\theta^*\ar[dl, "\cdot \theta"] \ar[d, "\sigma\cdot "] &\\
 1|1^*\ar[d, "H"]& \sigma|\theta^* \ar[d, "H"] \\
 0& 1|(s\theta)^* \ar[dl, "\cdot s"] \ar[d, "\cdot\theta"] \ar[dr, "\cdot s\theta"]\\
 1|\theta^* \ar[d, "H"] & 1|s^* \ar[d, "H"]&1|1^*\\
 0&0
 \end{tikzcd}
 (U,\sigma):
 \begin{tikzcd}& 1|1^*\ar[d, "\sigma\cdot"]\\
 & \sigma|1^* \ar[d, "H"]\\
 & 1|s^*
 \ar[dl, "\cdot s"] \ar[dr, "U\cdot "]
 \\
 1|1^* \ar[d, "H"]&& U|s^* \ar[d, "H"]\\
 0&&0
 \end{tikzcd}
 \]
 In the above diagram, the left tree encodes the only sequence $b_1,\dots, b_n$ for which  $m_{2|1|n}(\sigma,U,1,b_1,\dots, b_n)$ is non-zero. The only such sequence has $n=1$ and $a_1=s\theta$. The right tree shows that $m_{2|1|n}(U,\sigma,1,b_1,\dots, b_n)$ is always zero. This establishes, in particular, that $\delta_2^1(1,s\theta)=s\theta\otimes 1$.
 
 We now consider the case that $b=s\varphi_+$. There are two factorizations of $b$, namely $b=s\cdot \varphi_+$ and $b=z \cdot s$. These give corresponding sequences $(a_2,a_1)$ which are $(\sigma, Z)$ and $(T, \sigma)$. The corresponding trees are shown below:
 \[
 (\sigma,Z): 
 \begin{tikzcd} &1|1^* \ar[d, "Z\cdot"]\\
 & Z|1^* \ar[d, "H"] \\
 & 1|z^* \ar[dl, "\cdot z"] \ar[d, "\sigma\cdot"]\\
 1|1^* \ar[d, "H"] & \sigma|z^*  \ar[d, "H"]\\
 0&1|(zs)^* \ar[dl, "\cdot zs"] \ar[d, "\cdot z"] \ar[dr, "\cdot s"] \\
 1|1^*& 1|s^* \ar[d, "H"] & (1|\varphi^+)^* \ar[d,"H"]\\
 & 0&0 
 \end{tikzcd}
 (T,\sigma):
 \begin{tikzcd}
 1|1^* \ar[d, "\cdot \sigma"]\\
 \sigma|1^* \ar[d, "H"]\\
 1|s^*\ar[d, "T\cdot"] \ar[dr, "\cdot s"]\\
 T\otimes s^* \ar[d, "H"]& 1|1^* \ar[d, "H"]\\
 0&0
 \end{tikzcd}
 \]
 
 Finally, we consider the case that $b=s\theta \varphi_+$. In this case, there are six factorizations. The corresponding possibilities for $(a_3,a_2,a_1)$ are
 \[
 (U,T,\sigma),\,  (T,U,\sigma),\, (\sigma,Z,U),\, (T,\sigma,U),\, (U,\sigma,Z),\, (\sigma,U,Z).
 \]
The corresponding trees are described in Figure~\ref{fig:s-z-theta-term}. We see that only the sequence $(\sigma,Z,U)$ admits and sequence $b_1,\dots, b_n\in \cK^!$ such that $m_{3|1|n}(\sigma,Z,U,1,b_1,\dots, b_n)\neq 0$. Furthermore, there is only one such sequence $b_1,\dots, b_n$ with this property. This sequence has $n=1$ and $b_1=z s\theta$. In particular
 \[
 m_{3|1|1}(\sigma,Z,U,1,z s\theta)=i_1.
 \]
 This shows that $\delta_2^1(i_0,zs\theta)=zs\theta\otimes i_1$, and that there are no other sequences $b_1,\dots, b_n$ with $\delta_{n+1}^1(i_0,b_1,\dots, b_n)=zs\theta$. This completes the proof.
\end{proof}

\begin{figure}[h]
\begin{adjustbox}{scale=.75}
$(U,T,\sigma):$
\begin{tikzcd}[column sep=.2cm,row sep=.5cm]
1|1^* \ar[d, "\sigma\cdot "]
\\
\sigma|1^* \ar[d, "H"]
\\
1|s^* \ar[d, "\cdot s"] \ar[dr, "T\cdot"]
\\
1|1^* \ar[d, "H"] & T|s^* \ar[d, "H"]\\
0& 0
\end{tikzcd}
\hspace{.5cm}
$(T,U,\sigma):$
\begin{tikzcd}[column sep=.2cm,row sep=.5cm]
1|1^* 
\ar[d, "\sigma\cdot"]
\\
\sigma|1^*
	\ar[d, "H"]
\\
1|s^*
\ar[d, "\cdot s"] \ar[dr, "U\cdot"]
\\
1|1^* \ar[d, "H"] & U|s^* \ar[d, "H"]\\
0 & 0
\end{tikzcd}
\hspace{.5cm}
 $(\sigma,U,Z)$:
 \begin{tikzcd}[column sep=.3cm,row sep=.5cm]
 1|1^* \ar[d, "Z\cdot"]\\
 Z|1^* \ar[d, "H"]\\
 1|z^* \ar[d, "\cdot z"] \ar[dr, "U\cdot"]\\
 1|1^* \ar[d, "H"] & U|z^* \ar[d, "H"]\\
 0& 0
 \end{tikzcd}
\end{adjustbox}

\begin{adjustbox}{scale=.75}
$(T,\sigma,U):$ \begin{tikzcd}[column sep=.3cm,row sep=.5cm]
1|1^*
\ar[d, "U\cdot"]
\\
U|1^*
\ar[d, "H"]
\\
1|\theta^*
\ar[d, "\cdot \theta"] \ar[dr, "\sigma \cdot "]
\\
1|1^* \ar[d, "H"] & \sigma|\theta^* \ar[d, "H"]\\
0& 1|(s\theta)^* \ar[dl ," \cdot s",swap] \ar[d, "\cdot \theta",swap] \ar[dr, "\cdot s \theta",swap] \ar[drr, "T\cdot"] \\
1|\theta^* \ar[d,"H"] & 1|s^*\ar[d,"H"] & 1|1^* \ar[d,"H"]& T|(s\theta)^*\ar[d,"H"]\\
0&0&0&0
\end{tikzcd}
\hspace{.5cm}
$ (U,\sigma,Z):$
\begin{tikzcd}[column sep=.3cm,row sep=.5cm]
1|1^*
	\ar[d, "Z\cdot"]
\\
Z|1^*
	\ar[d, "H"]
\\
1|z^*
	\ar[d, "\cdot z"] \ar[dr, "\sigma\cdot"]
\\
1|1^* \ar[d, "H"] & \sigma|Z^* \ar[d,"H"]
\\
0 &1|(zs)^* \ar[dl, "\cdot z",swap] \ar[d, "\cdot s"] \ar[dr, "U\cdot",swap] \ar[drr, "\cdot zs"]\\
1|s^*\ar[d,"H"]& 1|\varphi_+^*\ar[d,"H"] & U|(zs)^*\ar[d,"H"] & 1|1^* \ar[d,"H"]\\
0&0&0&0
\end{tikzcd}
\end{adjustbox}

\vspace{.25cm}

\begin{adjustbox}{scale=.75}
\hspace{.5cm}$(\sigma, Z,U):$ \begin{tikzcd}[column sep={2.2cm,between origins}, row sep=.5cm]
1|1^* \ar[d, "U\cdot"]&&&&&&
\\
U|1^* \ar[d, "H"]
\\
1|\theta^* \ar[d, "\cdot \theta"] \ar[dr, "Z\cdot"]\\
1|1^* \ar[d, "H"] & Z|\theta^* \ar[d, "H"]\\
0& 1|(z\theta)^* \ar[dl, "\cdot z",swap] \ar[d, "\cdot \theta"] \ar[dr, "\sigma\cdot"]\\
1|\theta^* \ar[d, "H"] &1|z^* \ar[d, "H"]& \sigma|(z\theta)^* \ar[d,"H"]\\
0&0& 1|(zs\theta)^* \ar[dll, bend right=5, "\cdot z",swap] \ar[dl, "\cdot s",swap] \ar[d, "\cdot \theta",swap] \ar[dr, "\cdot\theta z",swap] \ar[drr, "\cdot s \theta",swap, pos=.6] \ar[drrr, bend left=5, "\cdot zs", pos=.75,swap] \ar[drrrr, bend left =10, "\cdot zs\theta"]
\\[.5cm]
1|(s\theta)^* \ar[d, "H"] & 1|(\theta  \varphi_+)^*\ar[d, "H"] & 1|(zs)^*\ar[d, "H"] & 1|s^*\ar[d, "H"] & 1|\varphi_+^*\ar[d, "H"] & 1|\theta^*\ar[d, "H"] & 1|1^*\\
0&0&0&0&0&0&
\end{tikzcd}
\end{adjustbox}
\caption{The six configurations which contribute to the $zs \theta$ coefficient of the differential of ${}^{\cK^!} [\cCo]^{\cK}\boxtimes {}_{\cK} [\cTr]_{\cK^!}$.}
\label{fig:s-z-theta-term}
\end{figure}

\section{Module categories}

In this section, we consider Koszul duality at the level of modules.  We define categories of modules ${}_{\cK} \MOD_{(U),\frb}$ and ${}^{\cK^!} \MOD_{(U),\frb}$, and we prove that tensoring with the bimodules ${}_{\cK}[\cTr]_{\cK^!}$ and ${}^{\cK^!}[\cCo]^{\cK}$ give functors which form an equivalence of categories between these categories. Additionally, we consider various categories of $DD$ and $DA$-bimodules.

\subsection{Bonsai modules}

 Recall from Definition~\ref{def:bonsai} that a type-$A$ module ${}_{\cA}M$ is called \emph{bonsai} if there is an $n_0>0$ so that the structure maps $m_{\cT}$ vanish for any ordinary (i.e. no valence 2 vertices) $A_\infty$-modules structure tree $\cT$ with $\deg(\cT)>n_0$.

 We now describe a dual notion. Unlike the notion of bonsai, our dual notion makes use of the algebraic grading on the algebra $\cK^!$, and therefore does not make sense for an arbitrary algebra.

\begin{define}
\label{def:cobonsai} We say that a type-$D$ module ${}^{\cK^!} M$ is \emph{cobonsai} if there is an $n_0>0$ so that each summand $b_1\otimes \cdots \otimes b_n \otimes \ys$ of $\delta^n(\xs)$, where $b_i\in \cK^!$ is a monomial, has the property that either some $b_i\in \ve{I}$, or
\[
-\sum_{i=1}^n (\gr_{\alg}(b_i)+1)<n_0.
\]
\end{define}

We can give a more basis independent description of the above definition, as follows. The algebra $\cK^!$ has a natural augmentation $\veps\colon \cK^!\to \ve{I}$. This map sends an algebra element in negative algebraic grading to 0, and it sends $1\in \cK^!$ to $1\in \ve{I}$. We write $\cK_+=\ker(\veps)$, and we write $\delta_+^1$ for the composition  
\[
\delta_+^1:=\left((\bI_{\cK^!}+\veps)\otimes \bI_M\right)\circ \delta^1.
\]
That is, $\delta_+^1$ is the composition of $\delta^1$ with the projection map from $\cK^!$ onto $\cK_+^!$. If we give $(\cK_+^!)^{\otimes n} \otimes M$ an algebraic grading by declaring
\[
\gr_{\alg}(b_1\otimes \cdots \otimes b_n\otimes \xs)=\sum_{i=1}^n (\gr_{\alg}(b_i)+1),
\]
then Definition~\ref{def:cobonsai} is equivalent to requiring there to be an $n_0>0$ so that $\delta_+^n$ has image in algebraic gradings above $-n_0$ for all $n$.

\begin{lem} \label{lem:bonsai-preserved}
\item
\begin{enumerate}
\item If ${}_{\cK} M$ is a bounded type-$A$ module, then ${}^{\cK^!} [\cCo]^{\cK}\boxtimes {}_{\cK} M$ is well-defined.  If ${}^{\cK^!} N$ is a type-$D$ module, then ${}_{\cK} [\cTr]_{\cK^!} \boxtimes {}^{\cK^!} N$ is well-defined. 
\item If ${}_{\cK} M$ is bonsai, then ${}^{\cK^!} [\cCo]^{\cK}\boxtimes {}_{\cK} M$ is cobonsai.
\item  If ${}^{\cK^!} N$ is cobonsai, then ${}_{\cK} [\cTr]_{\cK^!} \boxtimes {}^{\cK^!} N$ is bonsai.
\end{enumerate}
\end{lem}
\begin{proof}  We begin with the first claim. If ${}_{\cK} M$ is bounded, then ${}^{\cK^!} [\cCo]^{\cK} \boxtimes {}_{\cK} M$ is well-defined because the differential on the tensor product is a finite sum. Tensor products of the form ${}_{\cK} [\cTr]_{\cK^!}\boxtimes {}^{\cK^!} N$ are always well-defined since ${}_{\cK} [\cTr]_{\cK^!}$ is right bonsai by Lemma~\ref{lem:Tr-LR-bonsai}.

We now suppose that ${}_{\cK} M$ is bonsai. We will show that ${}^{\cK^!} [\cCo]^{\cK}\boxtimes {}_{\cK} M$ is cobonsai.  We consider the iterated structure map $\delta^n_+$ on the tensor product.
We can view $\delta^n_+$ as summing over certain planar graphs with two inputs (one for $\cTr$ and one for $M$). Examples are shown in Figure~\ref{fig:trees-delta^n}. 
 \begin{figure}[h]
\[
\begin{tikzcd}[row sep=.2cm]
 &&& {[\cCo]} \ar[d]& M \ar[dddd]\\
&&& \delta^{1,1} \ar[dddddddlll, bend right=20]\ar[d] \ar[dddr,bend left=20] &\\
&&& \delta^{1,1}\ar[ddddddlll, bend right=18]\ar[d] \ar[ddr,bend left=18]&\\
&&& \delta^{1,1}\ar[dddddlll, bend right=16]\ar[dd]\ar[dr, bend left=16] &\\
&&& &m_4 \ar[d]\\
&&&\bI \ar[d] \ar[dddll, "1", bend right=16]& m_1 \ar[dd]\\
&&& \delta^{1,1} \ar[dd] \ar[ddl, bend right=16]\ar[dr,bend left=16]&\\
&&&& m_2 \ar[d]\\
\Pi&\,&\,&\,&\,
\end{tikzcd}
\begin{tikzcd}[row sep=.2cm]
 && {[\cCo]} \ar[d]& M \ar[dddd]\\
&& \delta^{1,1} \ar[dddddddll, bend right=20]\ar[d] \ar[dddr,bend left=20] &\\
&& \delta^{1,1}\ar[ddddddll, bend right=18]\ar[d] \ar[ddr,bend left=18]&\\
&& \delta^{1,1}\ar[dddddll, bend right=16]\ar[dd]\ar[dr, bend left=16] &\\
&& &m_4 \ar[ddd]\\
&&\delta^{1,1} \ar[d] \ar[ddr, bend left=16] \ar[dddl, bend right=16]&\\
&& \delta^{1,1} \ar[dd] \ar[ddl, bend right=16]\ar[dr,bend left=16]&\\
&&& m_3 \ar[d]\\
\Pi&\Pi&\,&\,
\end{tikzcd}
\]
\caption{Two diagrams which encode summands of $\delta^n$ of ${}^{\cK^!} [\cCo]^{\cK} \boxtimes {}_{\cK} M$. The left tree is irrelevant to the cobonsai-ness of the tensor product, since it vanishes when we compose with $(\bI+\veps)\colon \cK^!\to \cK_+^!$. The right tree is relevant to the cobonsai-ness of the tensor product.}
\label{fig:trees-delta^n}
\end{figure}

If $b_1\otimes \cdots \otimes b_n\otimes \ys$ is a summand of $\delta^n_+(\xs)$ such that each $b_i$ is a monomial, then there is a corresponding ordinary $A_\infty$-module operation tree $\cT$ for $M$. (In Figure~\ref{fig:trees-delta^n}, $\cT$ would be obtained by taking the module input path starting at $M$ and going downward, together with any edge connected to a vertex along this path). There is one module vertex of $\cT$ for each $b_i$, and this vertex has valence $-\gr_{\alg}(b_i)+2$. Therefore,
\[
\deg(\cT)=\sum_{i=1}^n (-\gr_{\alg}(b_i)-1).
\]
If there is an $n_0$ so that $m_{\cT}=0$ if $\deg(\cT)>n_0$, then the above equation implies that the component of $\delta^n$ weighted by $b_1\otimes \cdots \otimes b_n$ will also vanish if $\sum_{i=1}^n (-\gr_{\alg}(b_i)-1)>n_0$, so ${}^{\cK^!} [\cCo]^{\cK}\boxtimes {}_{\cK} M$ is cobonsai.

We now suppose that ${}^{\cK^!} N$ is cobonsai. Let $\cT$ be a left $A_\infty$-module operation tree. Since $\cK$ is an associative algebra, for the purposes of showing bonsai-ness we may assume that the only vertices of valence greater than 3 occur along the path from the module input to the root. We consider the action $m_{\cT}$ on the tensor product ${}_{\cK} [\cTr]_{\cK^!} \boxtimes {}^{\cK^!} N$. For each $n\ge 0$, we take the iterated structure map $\delta^n$ of ${}^{\cK^!} N$, and sum all ways of splitting the algebra inputs into the module vertices of $m_{\cT}$. We will write $\cT'$ for a right expansion of $\cT$, i.e. a tree obtained by adding $n$ edges to the right side of the module vertices of $\cT$, such that these edges are each connected to a module vertex. See Figure~\ref{fig:5}.

We note that since each new edge is attached to a vertex of valence at least $3$, we have
  \begin{equation} 
\deg(\cT')=\deg(\cT)+n. \label{eq:degree-tree-bonsai}
\end{equation}

\begin{figure}[h]
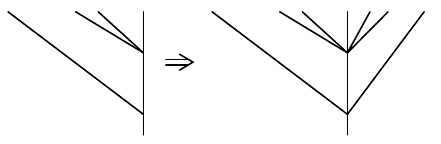
\caption{A left module input tree $\cT$ (left), and a right expansion $\cT'$ of $\cT$ (right)}
\label{fig:5}
\end{figure}

Since ${}_{\cK} [\cTr]_{\cK^!}$ is strictly unital, summands of $\delta^{1}$ of ${}^{\cK^!} N$ which are weighted by elements of $\ve{I}\otimes N$ make no contribution to $m_{\cT}$ (such differentials only contribute to $m_1$, which is never a factor of $m_{\cT}$). Since the structure maps of ${}_{\cK} [\cTr]_{\cK^!}$ are algebraically graded by Lemma~\ref{lem:grading-shifts}, we conclude that if the structure map is non-zero, then
\[
\deg(\cT')+\sum_{i=1}^n \gr_{\alg}(b_i)=0. 
\]
We can rewrite this as 
\[
\deg(\cT)+n+\sum_{i=1}^n \gr_{\alg}(b_i)=0
\]
or
\[
\deg(\cT)=\sum_{i=1}^n (-\gr_{\alg}(b_i)-1).
\]
Using the cobonsai-ness of ${}^{\cK^!} N$, the above equation implies that ${}_{\cK} [\cTr]_{\cK^!}\boxtimes {}^{\cK^!} N$ is also bonsai. 
\end{proof}

\subsection{Regularly $U$-adic modules}

Completions play an important role in the theory in \cite{ZemBordered}, so we discuss how they interact with Koszul duality in this section. We recall that in Section~\ref{sec:regularly-U-adic-modules}, we defined the notion of a regularly $U$-adic type-$A$ module ${}_{\cK} M$, which we recall is a type-$A$ module such that all of the operations $m_{\cT_+}$ are commensurate (Definition~\ref{def:commensurate}), where $\cT_+$ ranges over all extended $A_\infty$-module input trees. (Recall that ``extended'' means that valence 1 and 2 vertices are allowed). In this section, we describe a dual notion for type-$D$ modules over $\cK^!$.

We recall that if $a\in \cK$ is a monomial, we write $\wt_U(a)$ to be the maximum $i$ such that $a=U^i a'$ for some $a'\in \cK$. We set $\wt_U(0)=\infty$. Note that $\cK_{\ge i}$ is the $\bF$-span of monomials $a$ with $\wt_U(a)\ge i$.

The algebra $\cK^!$ has a dual weight function, as follows. If $b\in \cK^!$ is a monomial, we write $\wt_\theta^!(b)\in \{0,1\}$ for the number of $\theta$ factors in $b$. If $\ve{b}=b_1\otimes \cdots \otimes b_n$ is a sequence of monomials in $\cK^!$, we write
\[
\wt_{\theta}^!(\ve{b})=\sum_{i=1}^n \wt_\theta^!(b_i).
\]
We view $\wt_{\theta}^!(0)=-\infty$. 

We view $\cK^!$ as being a regularly $\Z$-filtered space where the filtration level of an algebra element $a$ is given by $-\wt_{\theta}^!(a)$. Then $\cK^!_{\ge -1}=\cK^!,$ $\cK^!_{\ge 0}$ is the span of monomials which are not multiples of $\theta$, and $\cK^!_{\ge 1}$ is 0.

\begin{define} A \emph{regularly $U$-adic type-$D$ module} ${}^{\cK^!} M$ consists of a regularly filtered $\ve{I}$-module $M$ equipped with a weakly filtered map $\delta^1\colon M\to \cK^!\otimes M$. Furthermore, the iterates
\[
\delta^n\colon M\to \cK^!\otimes \cdots\otimes \cK^!\otimes M
\]
  are commensurate.
\end{define}

\begin{rem}
  Informally, the above definition says that in the iterates $\delta^n$, if many $\theta$ terms appear amongst the $\cK^!$ factors, then the filtration level in $M$ must be high.
\end{rem}

\begin{lem}
\label{lem:inequality-U-weights} Suppose that $a_n,\dots, a_1\in \cK$ and $b_1,\dots, b_k\in \cK^!$ are monomials such that the structure map on ${}_{\cK} [\cTr]_{\cK^!}$ satisfies
\[
m_{n|1|k}(a_n,\dots, a_1, 1, b_1,\dots, b_k)\neq 0.
\]
Then
\[
\sum_{i=1}^n \wt_U(a_i)\le \sum_{i=1}^k \wt^!_\theta(b_i).
\]
\end{lem}
\begin{proof}
The structure maps on ${}_{\cK} [\cTr]_{\cK^!}$ are defined in Lemma~\ref{lem:Tr-construction} by applying the homological perturbation lemma to the bimodule ${}_{\cK} \scL_{\cK^!}$. As an $(\ve{I},\ve{I})$-module, $\scL$ decomposes as a direct sum $\scL_{\theta}\oplus \scL_1$. Here, $\scL_{\theta}$ denotes the span of generators $\xs\otimes \ys^*$ such that $\ys$ is a monomial and is a multiple of $\theta$. Here $\scL_1$ denotes the span of generators $\xs\otimes \ys^*$ such that $\ys$ is a monomial and is not a multiple of $\theta$.  The homotopy $H$ is constructed in Lemma~\ref{lem:Tr-construction} and decomposes as in the diagram below:
\[
H=\begin{tikzcd}
 \ar[loop left, "H_{\theta,\theta}"] \scL_\theta & \scL_{1} \ar[l, "H_{\theta,1}"] \ar[loop right, "H_{1,1}"] 
 \end{tikzcd}
\]
Furthermore, we observe that $H_{\theta,\theta}$ and $H_{1,1}$ preserve the $\wt_U$-weight of the $\cK$-factor of $\scL$, while $H_{\theta,1}$ decreases the $\wt_U$-weight by $1$.

We consider a sequence of generators
\[
1|1^*=\xs_1|\ys_1^*, \xs_2|\ys_2^*,\dots, \xs_m|\ys_m^*=1|1^*
\]
of $\scL$, which contribute to the action $m_{n|1|k}(a_n,\dots, a_1,1,b_1,\dots, b_k)$ of ${}_{\cK} [\cTr]_{\cK^!}$. More precisely, we assume that each 
$\xs_{2i}|\ys_{2i}^*$ is obtained from $\xs_{2i-1}|\ys_{2i-1}^*$ by applying either $m_{1|1|0}(a_j,-)$ or $m_{0|1|1}(-,b_j)$ or $a_{0|1|0}^{\mu_1}$. Each $\xs_{2i+1}|\ys_{2i+1}^*$ is obtained from $\xs_{2i}|\ys_{2i}^*$ by applying $H$. 
Note that $\a_{0|1|0}^{\mu_2}$ never contributes to the homological perturbation lemma by Lemma~\ref{lem:a010-no-contribution}.

Write $N_{\theta,1}$ for the number of times $H_{\theta,1}$ is applied. We consider the sequence $\wt_1,\dots, \wt_m$ where $\wt_i=\wt_U(\xs_i)$.  The sequence decreases by at most 1 at each step. A downward step by $1$ occurs if and only if $H_{\theta,1}$ is applied. Between each application of $H_{\theta,1}$, there must be an application of $m_{0|1|1}(-,b_i)$ for some $b_i$ with $\wt_{\theta}^!(b_i)=1$. Furthermore, between the final $H_{\theta,1}$ and the last element in the sequence, there must also be some $m_{0|1|1}(-,b_i)$ with $\wt_{\theta}^!(b_i)=1$.  Therefore,
\[
N_{\theta,1}\le \wt_{\theta}(b_1)+\cdots+\wt_{\theta}(b_k).
\]

 On the other hand, the sequence $\wt_i$ only increases after an application of $m_{1|1|0}(a_i,-)$, and the weight increases by at least $\wt_{U}(a_i)$. Since each drop in the weight sequence is by 1, the number of drops must be equal to sum of the increases in the weight sequence, which is at least $\wt_U(a_1)+\cdots +\wt_U(a_n)$. Therefore,
 \[
 \sum_{i=1}^n \wt_U(a_i)\le N_{\theta,1}\le  \sum_{i=1}^k \wt_\theta^!(b_i)
 \]
 completing the proof.
 \end{proof}

We obtain the following corollary of Lemma~\ref{lem:inequality-U-weights}:
\begin{cor}
\label{cor:filtered-Trace-map}
If we put $[\cTr]$ in filtration level 0, then the structure maps on $[\cTr]$ satisfy
\[
m_{\cT_+}\left(\cK\otimes \cdots \otimes \cK\otimes [\cTr]\otimes \cK^!\otimes \cdots \otimes \cK^!\right)_{\ge i}\subset [\cTr]_{\ge i}
\]
for any extended $A_\infty$-bimodule input tree $\cT_+$. 
\end{cor}
\begin{proof} Lemma~\ref{lem:inequality-U-weights} proves the case when $\cT_+$ has a single interior vertex. For the general case, we argue as follow. We observe that multiplication on $\cK$ satisfies
\[
\mu_2(\cK_{\ge i}\otimes \cK_{\ge j})\subset \cK_{\ge i+j}.
\]
Multiplication on $\cK^!$ satisfies the same equation. On $\cK^!$, we additionally have
\[
\mu_1(\cK^!_{\ge i})\subset \cK^!_{\ge i}\quad \text{and} \quad \mu_0\in \cK_{\ge 0}. 
\]
Finally, we note that the differential $m_{0|1|0}$ on ${}_{\cK}[\cTr]_{\cK^!}$ vanishes. 

 Therefore, it is sufficient to show the claim when $\cT_+$ is a tree whose interior vertices all occur along the path from the module input to the root, and such that there are no valence 2 vertices. If $\cT_+$ is such a bimodule operation tree with $N$ interior vertices, and $m_{\cT_+}(a_n,\dots, a_1, 1, b_1,\dots, b_k)\neq 0$, then the inequality $\sum_{i=1}^n \wt_U(a_i)\le \sum_{i=1}^k \wt_\theta^!(b_i)$ is obtained by applying Lemma~\ref{lem:inequality-U-weights} $N$ times. 
\end{proof}

\begin{lem}
\label{lem:strongly-U-adic-A->D} Suppose that ${}_{\cK} M$ is a bounded type-$A$ module (so that the tensor product ${}^{\cK^!} [\cCo]^{\cK} \boxtimes {}_{\cK} M$ is well-defined).  If ${}_{\cK} M$ is regularly $U$-adic, then ${}^{\cK^!} [\cCo]^{\cK} \boxtimes {}_{\cK} M$ is regularly $U$-adic. 
\end{lem}
\begin{proof} Let $C$ be the commensurability constant for ${}_{\cK} M$. The iterated structure map $\delta^n$ of the box tensor product can be described as a sum of maps of the form
\[
(\Pi_{\cT_+}\otimes  \bI_M)\circ (\bI_{\otimes^j \cK^!} \otimes \bI_{[\cCo]} \otimes  m_{\cT_+})\circ (\delta^{j,j}\otimes \bI_M),
\]
for $j\ge 0$ and various extended input trees $\cT_+$. Here, $\Pi_{\cT_+}$ means to multiply certain consecutive tensors of $\otimes^j \cK^!$ (corresponding to vertices of $\cT_+$ along the module path of valence 4 or more), and also add a tensor factor of $1$ for each valence 2 vertex along the module path of $\cT_+$. Compare Figure~\ref{fig:trees-delta^n}.

We observe that
\[
\delta^{n,n}( [\cCo]_{\ge 0})\subset (\cK^!\otimes \cdots \otimes \cK^!\otimes [\cCo]\otimes \cK\otimes \cdots \otimes \cK)_{\ge 0}. 
\]
This is because the only summands of the differential $\delta^{1,1}$ which are weighted by $\theta$ are also weighted by $U$. Similarly, $\Pi_{\cT_+}$ preserves the $\wt_\theta^!$-weight of a tensor of monomials (or maps them to zero). Therefore
\[
\Big((\Pi_{\cT_+}\otimes  \bI_M)\circ (\bI_{\otimes^j \cK^!} \otimes \bI_{[\cCo]} \otimes  m_{\cT_+})\circ (\delta^{j,j}\otimes \bI_M)\Big)\Big( ([\cCo]\otimes M)_{\ge i}\Big)\subset ([\cCo]\otimes M)_{\ge i-C}.
\]
\end{proof}

\begin{lem}
\label{lem:strongly-U-adic-D->A}
 If ${}^{\cK^!} M$ is a regularly $U$-adic type-$D$ structure, then the tensor product ${}_{\cK} [\cTr]_{\cK^!} \boxtimes {}^{\cK^!} M$ is regularly $U$-adic. 
\end{lem}
\begin{proof} By definition, if $\cT_+$ is an extended module operation tree, then the operation on  $m_{\cT_+}$ on ${}_{\cK} [\cTr]_{\cK^!} \boxtimes {}^{\cK^!} M$ can be described as a sum of certain compositions obtained by first applying iterates $\delta^i$ of the structure map on ${}^{\cK^!} M$, followed by the bimodule operation $m_{\cT}$ on $[\cTr]$, for some other (ordinary) $AA$-bimodule operation tree $\cT$.

Let $C$ be the commensurability constant of ${}^{\cK^!} M$. We observe that $\bI\otimes \delta^i$ maps
\[
(\cK\otimes \cdots \otimes \cK)_{\ge j} \otimes[\cTr]\otimes  M_{\ge k}
\]
into
\[
\sum_{s+t=k-C}(\cK\otimes \cdots \otimes \cK)_{\ge j} \otimes [\cTr]\otimes (\cK^!\otimes \cdots \otimes \cK^!)_{\ge s}\otimes M_{\ge t}.
\]
By Corollary~\ref{cor:filtered-Trace-map}, the structure map $m_{\cT}$ from $[\cTr]$ maps the above to
\[
\sum_{s+t=k-C} [\cTr]_{\ge j+s}\otimes M_{\ge t}\subset \left([\cTr]\otimes M \right)_{\ge j+k-C}.
\]
(Note that since $[\cTr]$ is concentrated in filtration level 0,  the above can be viewed instead as $[\cTr]\otimes M_{\ge j+k-C}$).
\end{proof}

\subsection{Categories} We  will write ${}^{\cK^!} \MOD_{(U),\frb}$ for the category of cobonsai, regularly $U$-adic type-$D$ modules over $\cK^!$. We write ${}_{\cK} \MOD_{(U),\frb}$ for the category of bonsai, regularly $U$-adic type-$A$ modules over $\cK$ which are also \emph{strictly unital}.

 We now define the notions of regularly $U$-adic and cobonsai morphisms in our categories:

\begin{define}
\item
\begin{enumerate}
\item If ${}^{\cK^!} M, {}^{\cK^!} N$ are type-$D$ modules, we say a map $f^1\colon M\to \cK^!\otimes N$ is \emph{regularly $U$-adic} if it is weakly filtered (equipping $\cK^!$ with the filtration supported in filtration levels $-1$ and $0$). We say a map is \emph{bonsai} if the algebraic gradings of the $\cK^!$-factors appearing as coefficients in $f^1$ are bounded from below.
\item  If ${}_{\cK} M, {}_{\cK} N$ are type-$A$ modules, then an $A_\infty$-morphism $f_*=(f_{j+1})_{j\ge 0}$ is called \emph{regularly $U$-adic} if the $f_{j+1}$ are all weakly filtered, and furthermore the family of maps is commensurate. We say the morphism $f_*$ is \emph{bonsai} if $f_{j+1}=0$ for all $j\gg 0$. 
\end{enumerate}
\end{define}

We define the morphisms in ${}_{\cK} \MOD_{(U),\frb}$ to be the set of strictly unital, regularly $U$-adic and bonsai morphisms, in the above sense. We define the morphisms in ${}^{\cK^!} \MOD_{(U),\frb}$ similarly.

\begin{lem} If $f_*$ is a morphism in ${}_{\cK} \MOD_{(U),\frb}$, then $\d(f_*)$ is also a morphism in this category. If $\d(f_*)=0$, then $\Cone(f_*)$ is an object in ${}_{\cK} \MOD_{(U),\frb}$ if and only if $f_*$ is a morphism in ${}_{\cK} \MOD_{(U),\frb}$. The same holds for morphisms in  ${}^{\cK^!} \MOD_{(U),\frb}$. 
\end{lem}

The proof of the above lemma is straightforward, so we omit it.

\subsection{Equivalences of categories}
\label{sec:equivalence-of-categories}

The categories ${}_{\cK} \mathsf{Mod}_{(U),\frb}$ and ${}^{\cK^!}\mathsf{Mod}_{(U),\frb}$ are $dg$-categories, but it will be helpful to view them as $A_\infty$-categories. Our main result is the following:

\begin{prop}\label{prop:equivalence-of-categories-D} There is an equivalence of $A_\infty$-categories:
\[
{}_{\cK} \MOD_{(U),\frb}\simeq {}^{\cK^!} \MOD_{(U),\frb}.
\]
\end{prop}

 We now give a brief sketch of the idea. Subsequently, we will give more details. Tensoring with ${}_{\cK}[\cTr]_{\cK^!}$ and ${}^{\cK^!} [\cCo]^{\cK}$ defines $A_\infty$-functors between the above two categories. Lemmas~\ref{lem:bonsai-preserved}, ~\ref{lem:strongly-U-adic-A->D} and ~\ref{lem:strongly-U-adic-D->A} show that these functors preserve the bonsai/cobonsai and regular $U$-adic properties. Lemma~\ref{lem:strictly-unital} implies that if ${}^{\cK^!} M$ is a type-$D$ module, then ${}_{\cK} [\cTr]_{\cK^!}\boxtimes {}^{\cK^!} M$ is strictly unital. From Propositions~~\ref{prop:quasi-inverse-1} and ~\ref{prop:quasi-inverse-2} we know that ${}^{\cK^!} [\cCo]^{\cK}$ and ${}_{\cK} [\cTr]_{\cK^!}$ are quasi-inverses which essentially immediately gives suitable versions of natural transformations between the composite functors and the identity functors.

We will now spell out a few more details in the above sketch, for the benefit of the reader. We first recall some basic categorical notions. We refer the reader to \cite{SeidelFukaya}*{Chapter~1} for a more detailed background.

\begin{define}
\item
\begin{enumerate}
\item An $A_\infty$-category $\cA$ is a set of objects $\cO(\cA)$,  together with vector spaces of morphisms  $\Hom(X,Y)$ (over $\bF=\Z/2$), and composition maps
\[
\mu_{n}\colon \Hom(X_{n-1},X_n)\otimes \cdots \otimes \Hom(X_0,X_1)\to \Hom(X_0,X_n),
\]
ranging over $n\ge 1$, which satisfy the $A_\infty$-associativity relations:
\[
0=\sum_{i=1}^n \sum_{j=1}^{n-i+1} \mu(a_d,\dots,a_{i+j},\mu_{j}(a_{i+j-1},\dots, a_{i}),a_{i-1},\dots a_1)
\]
\item If $\cA,\cB$ are $A_\infty$-categories, then an $A_\infty$-functor $\cF\colon \cA\to \cB$ consists of a map on objects $\cF\colon \cO(\cA)\to \cO(\cB)$, together with a collection of maps on morphisms
\[
\cF_n\colon \Hom(X_{n-1},X_{n})\otimes \cdots \otimes \Hom(X_0,X_1)\to \Hom(\cF(X_0),\cF(X_n))
\]
which satisfy the $A_\infty$-functor relation:
\[
\sum_{n} \sum_{i_1,\dots, i_n} \mu_n^{\cB}(\cF_{i_n}(a_d,\dots, a_{d-i_n+1}),\dots, \cF_{i_1}(a_{i_1},\dots, a_1)) 
\]
\[
=\sum_{k,j} \cF_{d-j+1}(a_d,\dots, a_{k+j-1},\mu_j^{\cA}(a_{k+j},\dots, a_{k+1}),a_{k}, \dots, a_1)
\]
\item If $\cF,\cG\colon \cA\to \cB$ are $A_\infty$-functors, then we can consider a vector space of \emph{prenatural transformations} from $\cF$ to $\cG$. These consist of collections of maps
\[
\phi_{n}\colon \Hom(X_{n-1},X_{n})\otimes \cdots \otimes \Hom(X_0,X_1)\to \Hom(\cF(X_0),\cG(X_n))
\]
ranging over $n\ge 0$. Here, we view $\phi_0$ as giving an element of $\Hom(\cF(X),\cG(X))$ for each $X\in \cO(\cA)$. The space of prenatural transformations is a chain complex. In fact, these are morphisms in the category of $A_\infty$-functors. The category of $A_\infty$-functors from $\cA$ to $\cB$ is itself an $A_\infty$-category. If $\phi\colon \cF\to \cG$ is a natural transformation, the differential is given (slightly informally) by
\[
\begin{split}
&\mu_1(\phi)(a_d,\dots, a_1)\\
=&\sum \mu_*^{\cB}(\cF_{*}(\cdots),\dots, \cF_{*}(\cdots),\phi_{*}(\cdots),\cF_{*}(\cdots),\dots,\cF_{*}(\cdots))\\
+&\sum \phi_{*}(\cdots, \mu_*^{\cA}(\cdots), \cdots).
\end{split}
\]
In the above, each $\cdots$ indicates a consecutive subcollection of $a_d,\dots, a_1$, and the sum is over possible numbers of $\cF_*(\cdots)$ terms.
Similarly, if
\[
\begin{tikzcd}
 \cF^0\ar[r, "\phi^1"] &\cF^1\ar[r]& \cdots \ar[r, "\phi^n"]& \cF^n
\end{tikzcd}
\]
is a diagram of prenatural transformations with $n\ge 2$, then $\mu_n(\phi^n,\dots, \phi^1)$ is given by the sum
\[
\begin{split}
&\mu_n(\phi^n,\dots, \phi^1)(a_d,\dots, a_1)
\\
=&\sum \mu_*^{\cB}(\cF^n(\cdots),\dots, \cF^n(\cdots), \phi^n(\cdots), \cF^{n-1}(\cdots),\dots,\\
& \quad\cF^1(\cdots), \phi^1(\cdots),\cF^0(\cdots),\dots, \cF^0(\cdots)).
\end{split}
\]
See \cite{SeidelFukaya}*{pg. 10} for more details.
\item A \emph{homotopy equivalence} between two functors $\cF,\cG\colon \cA\to \cB$ consists of a diagram of prenatural transformations
\[
\begin{tikzcd} 
\ar[loop left, "H"] \cF\ar[r, shift left, "\eta"] & \cG \ar[loop right, "J"] \ar[l, shift left, "\xi"]
\end{tikzcd}
\]
where $\mu_1(\eta)=0$, $\mu_1(\xi)=0$, $\mu_2(\xi, \eta)=\bI+\mu_1(H)$ and $\mu_2(\eta,\xi)=\bI+\mu_1(J)$.
\item We say two $A_\infty$-categories $\cA$ and $\cB$ are \emph{equivalent} if there are $A_\infty$-functors $\cF\colon \cA\to \cB$ and $\cG \colon \cB\to \cA$ as well as homotopy equivalences of functors
\[
\cG\circ \cF\simeq \bI_{\cA}\quad\text{and} \quad \cF\circ \cG\simeq \bI_{\cB}.
\]
\end{enumerate}
\end{define}

\begin{proof}[Proof of Proposition~\ref{prop:equivalence-of-categories-D}] The proof is a rather straightforward consequence of Theorem~\ref{thm:equality-of-box-tensor-products}, but we will spell out the argument. If $M$ is a bimodule, we will write $\cF_M$ for the functor induced by tensoring $M$. We have two functors:

\[
\cF_{[\cCo]}\colon {}_{\cK} \MOD_{(U),\frb}\to {}^{\cK^!} \MOD_{(U),\frb},\quad \text{and} \quad \cF_{[\cTr]}\colon {}^{\cK^!} \MOD_{(U),\frb}\to {}_{\cK} \MOD_{(U),\frb},
\]
 given by tensoring with ${}^{\cK^!} [\cCo]^{\cK}$ and ${}_{\cK} [\cTr]_{\cK^!}$. Noting that
\[
{}^{\cK^!} [\cCo]^{\cK} \boxtimes \left( {}_{\cK} [\cTr]_{\cK^!} \boxtimes {}^{\cK^!} X\right)=\left({}^{\cK^!} [\cCo]^{\cK} \boxtimes  {}_{\cK} [\cTr]_{\cK^!} \right)\boxtimes {}^{\cK^!} X={}^{\cK^!} X.
\]
for any type-$D$ module ${}^{\cK^!} X$, we conclude that $\cF_{[\cCo]}\circ \cF_{[\cTr]}=\bI$.

 On the other hand $\cF_{[\cTr]}\circ \cF_{[\cCo]}$, is more complicated since we only have a homotopy equivalence
\begin{equation}
 {}_{\cK} [\cTr]_{\cK^!}\boxtimes \left(  {}^{\cK^!} [\cCo]^{\cK} \boxtimes {}_{\cK} Y\right)\simeq \left({}_{\cK} [\cTr]_{\cK^!} \boxtimes {}^{\cK^!} [\cCo]^{\cK}   \right)\boxtimes {}_{\cK} Y={}_{\cK} Y
 \label{eq:homotopy-associativity}
\end{equation}
for a ${}_{\cK} Y\in {}_{\cK}\MOD_{(U),\frb}$.

To construct a homotopy equivalence of functors between $\cF_{[\cTr]}\circ \cF_{[\cCo]}$ and $\bI$, we give a slightly indirect argument. (A direct argument using diagonals, in the spirit of \cite{LOTDiagonals}, is possible, but a bit of a detour). Firstly, note that if ${}_{\cK}M^{\cK}$ and ${}_{\cK}N^{\cK}$ are homotopy equivalent $DA$-bimodules, then there is also a homotopy equivalence of functors $\cF_M\simeq \cF_N$. We use the homotopy equivalence 
\[
{}^{\cK^!}[\cCo]^{\cK}\simeq {}^{\cK^!} [\cCo]^{\cK} \boxtimes {}_{\cK} \scL_{\cK^!} \boxtimes {}^{\cK^!} [\cCo]^{\cK},
\]
to see
\[
\cF_{[\cTr]}\circ \cF_{[\cCo]}\simeq \cF_{[\cTr]}\circ \cF_{[\cCo]\boxtimes \scL\boxtimes [\cCo]}.
\]
On the other hand, ${}^{\cK^!}[\cCo]^{\cK}\boxtimes {}_{\cK}\scL_{\cK^!}\boxtimes {}^{\cK^!}[\cCo]^{\cK}$ is a \emph{separated type-$DD$ structure} \cite{LOTBimodules}*{Definition~2.2.57}. Recall that a separated type-$DD$ structure ${}^{\cA} M^{\cB}$ is one where $\delta^{1,1}=\delta^{1L}+\delta^{1R}$ where $\delta^{1L}\colon M\to \cA \otimes M\otimes \ve{j}$ and $\delta^{1R}\colon M\to \ve{k}\otimes M\otimes \cB$. (Here, $\cA$ and $\cB$ denote algebras over $\ve{k}$ and $\ve{j}$, respectively). It follows from \cite{LOTBimodules}*{Lemma~2.3.14 (3)}, that if $P_{\cA}$ and ${}_{\cB} Q$ are strictly unital type-$A$ modules, and ${}^{\cA} M^{\cB}$ is separated, then there is a canonical isomorphism
\[
(P_{\cA} \boxtimes {}^{\cA} M^{\cB})\boxtimes {}_{\cB} Q\iso P_{\cA}\boxtimes ({}^{\cA} M^{\cB}\boxtimes {}_{\cB} Q),
\]
and similarly for box tensor products of strictly unital morphisms. (Note that Lipshitz, Ozsv\'{a}th and Thurston focus on uncurved $dg$-algebras, though the same proof works for curved $dg$-algebras). In particular, box tensor products involving ${}^{\cK^!}([\cCo]\boxtimes \scL\boxtimes [\cCo])^{\cK}$ are strictly associative, so 
\[
{}_{\cK} [\cTr]_{\cK^!} \boxtimes \left({}^{\cK^!}([\cCo]\boxtimes \scL\boxtimes [\cCo])^{\cK}\boxtimes {}_{\cK} M\right) =\left({}_{\cK} [\cTr]_{\cK^!} \boxtimes{}^{\cK^!}([\cCo]\boxtimes \scL\boxtimes [\cCo])^{\cK} \right)\boxtimes {}_{\cK} M. 
\]
A similar manipulation holds for morphisms,  so we conclude that
\[
\cF_{[\cTr]}\circ \cF_{[\cCo] \boxtimes \scL\boxtimes [\cCo]}=\cF_{[\cTr]\boxtimes [\cCo]\boxtimes \scL\boxtimes [\cCo]}\simeq \cF_{\bI}=\bI,
\]
completing the proof.
\end{proof}

\subsection{$DA$ and $DD$-bimodules}

In this section, we extend our results about Koszul duality to categories of $DD$ and $DA$ bimodules over $(\cK^!,\cK)$ and $(\cK,\cK)$, respectively.

Firstly, we describe how to generalize the bonsai condition to $DA$ and $DD$-bimodules. If ${}_{\cK} M^\cK$ is a type-$DA$ module and $\cT$ is an ordinary (i.e. no valence 1 or 2 vertices) $A_\infty$-module structure tree with $k$ algebra inputs and $n$ interior vertices along the path from the module input to the root, then we will write 
\[
\delta_{\cT}^n\colon \underbrace{\cK\otimes \cdots \otimes \cK}_{k}\otimes M\to M\otimes \underbrace{\cK\otimes \cdots \otimes \cK}_n
\]
for the corresponding composition of the structure maps of $M$ and the algebra multiplication on the $\cK$ inputs (according to $\cT$). We do not multiply any of the outputs of any $\delta_{i+1}^1$. Instead they form the algebra factors in the codomain of $\delta_{\cT}^n$.

\begin{define}
\item 
\begin{enumerate}
\item We say a type-$DA$ bimodule ${}_{\cK} M^{\cK}$ is \emph{bonsai} if there is an $n_0$ so that for all ordinary (i.e. no valence 2 vertices) $A_\infty$-module input trees $\cT$ with $\deg(\cT)>n_0$, the structure map $\delta_{\cT}^n$ vanishes, where $n$ is the number of interior vertices along the module path of $\cT$.
\item We say that a type-$DD$ bimodule ${}^{\cK^!} M^{\cK}$ is cobonsai if there is an $n_0$ so that any summand $b_1\otimes \cdots \otimes b_n\otimes \xs\otimes a_n\otimes\cdots \otimes a_1$ appearing in the iterate $\delta^{n,n}$ has the property that
\[
\sum_{i=1}^n \max(-\gr_{\alg}(b_i)-1,0)<n_0.
\]
\end{enumerate}
\end{define}

We now describe how to generalize the regularly $U$-adic condition:

\begin{define}
\item 
\begin{enumerate}
\item  A \emph{type-$DD$ module} ${}^{\cK^!} M^{\cK}$ is called \emph{regularly $U$-adic} if $M$ is a regular $\Z$-filtered $(\ve{I},\ve{I})$-bimodule , equipped with a weakly filtered map $\delta^{1,1}\colon M\to \cK^!\otimes^! M\otimes^! \cK$ which satisfies the type-$DD$ structure equation, and such that all of the iterates
\[
\delta^{n,n}\colon X\to  \cK^!\otimes \cdots \otimes \cK^!\otimes X\otimes \cK\otimes \cdots \otimes \cK
\]
are commensurate.
\item
We say that a type-$DA$-module ${}_{\cK} M^{\cK}$ is \emph{regularly $U$-adic} if $M$ is a regularly $\Z$-filtered left $\ve{I}$-module such that the family of maps
\[
\delta_{\cT_+}^n\colon \underbrace{\cK\otimes \cdots \otimes \cK}_k\otimes X\to X\otimes \underbrace{\cK\otimes \cdots \otimes \cK}_n,
\]
ranging over all generalized $A_\infty$-module structure trees $\cT_+$, is commensurate. (Recall that ``generalized'' means that we allow valence 1 and 2 vertices).
\end{enumerate}
\end{define}

Essentially the same arguments as Lemmas~\ref{lem:strongly-U-adic-A->D} and~\ref{lem:strongly-U-adic-D->A} show that tensoring with ${}_{\cK} [\cTr]_{\cK^!}$ sends a regularly $U$-adic type-$DD$ module ${}^{\cK^!} M^{\cK}$ to a regularly $U$-adic $DA$ bimodule ${}_{\cK} M^{\cK}$, and similarly when we tensor with ${}^{\cK^!} [\cCo]^{\cK}$. Similarly, bonsai and cobonsai modules are switched by tensoring with these bimodules. 

We write ${}^{\cK^!} \MOD_{(U),\frb}^{\cK}$ for the category of regularly $U$-adic, cobonsai modules, and we write ${}_{\cK} \MOD_{(U), \frb}^{\cK}$ for the category of strictly unital, regularly $U$-adic, bonsai $DA$-bimodules. A similar argument to Proposition~\ref{prop:equivalence-of-categories-D} shows the following:

\begin{prop} There is an equivalence of $A_\infty$-categories
\[
{}_{\cK} \MOD_{(U),\frb}^{\cK}\simeq {}^{\cK^!} \MOD_{(U),\frb}^{\cK}.
\]
\end{prop}

It does not seem to be the case that if ${}_{\cK} M^{\cK}$ is a $DA$-bimodule with $\delta_2^1$ and $\delta_1^1$ continuous and $\delta_{j+1}^1=0$ for $j\not\in \{0,1\}$, then ${}_{\cK} M^{\cK}$ must be regularly $U$-adic (compare Remark~\ref{rem:only-m2m1}). The following lemma, although essentially trivial, is useful:

\begin{lem}
\label{lem:delta11,delta21}  Suppose that  ${}_{\cK} M^{\cK}$ is a type-$DA$
bimodule, all of whose structure maps are filtered (not just weakly filtered). Then ${}_{\cK} M^{\cK}$ is regularly $U$-adic.
\end{lem}
\begin{proof} Since each $\delta_{j+1}^1$ is filtered, and the multiplication map $\mu_2$ on $\cK$ is filtered, it follows that the composite maps $\delta_{\cT_+}^n$ are also filtered, and hence trivially commensurate.
\end{proof}

\begin{rem} The definitions of bonsai and regularly $U$-adic $DA$-bimodules feature different structure trees. For the bonsai condition, we consider ordinary $A_\infty$-module structure trees (no valence 2-vertices), and for the regular $U$-adic condition, we consider generalized $A_\infty$-module structure trees (valence 1 and 2 vertices allowed). Roughly speaking, the effect is that $\delta_1^1$ does not contribute to the bonsai condition, but does contribute to the regular $U$-adic condition. The reason for these choices is to ensure that box tensor products preserve both of these properties, and also to ensure that they are satisfied in the basic examples we have in mind. See Lemmas~\ref{lem:box-tensor-product-bonsai} and ~\ref{lem:CZZ-U-adic-cobonsai}. 
\end{rem}

\begin{rem}
\label{rem:bordered-3-manifolds-bimodules}
\begin{enumerate}
\item  The link surgery type-$DA$ bimodules of two component links are naturally presented as regularly $U$-adic and bonsai bimodules, as we now describe. If $L\subset Y$ is a 2-component link (equipped with a longitudinal framing), then ${}_{\cK} \cX(Y,L)^{\cK}$ is defined by taking the type-$D$ module $\cX(Y,L)^{\cK\otimes \cK}$ and box tensoring with the type-$A$ module ${}_{\cK\otimes \cK} [\bI^{\Supset}]$. (Here, we tensor only of the algebra components from each module together, yielding a $DA$-bimodule over $(\cK,\cK)$).  See \cite{ZemBordered}*{Section~8.4} for the definition of ${}_{\cK\otimes \cK} [\bI^{\Supset}]$. Write ${}_{\cK}\cY(Y,L)^{\cK}$ for the resulting tensor product.  This bimodule has only $\delta_1^1$ and $\delta_2^1$ non-trivial, and is therefore bonsai. The map $\delta_2^1$ is filtered. If we put $\cX(Y,L)^{\cK\otimes \cK}$ in filtration level 0 (see Lemma~\ref{lem:fg-weakly-filtered}), then the map $\delta_1^1$ on the tensor product will also be filtered. By Lemma~\ref{lem:delta11,delta21}, the $DA$-bimodule ${}_{\cK} \cY(Y,L)^{\cK}$ is regularly $U$-adic.
\item The small models for 2-component L-space links described in \cite{CZZSatellites}, recalled in Section~\ref{sec:CZZ} bonsai and regularly $U$-adic. See Lemma~\ref{lem:CZZ-U-adic-cobonsai}.
\item One could attempt to build a theory where morphisms and bimodules are all required to be filtered (not just weakly filtered). Such a theory seems insufficient for our purposes. For example, the transformer bimodule ${}_{\cK} \cT^{\cK}$ in Section~\ref{sec:transformer} does not have this property, but is regularly $U$-adic with commensurability constant $C=1$.
\item We cannot omit the strictly unital condition from the categories ${}_{\cK} \MOD_{(U),\frb}$ and ${}_{\cK} \MOD_{(U),\frb}^{\cK}$. For example, consider the type-$A$ module ${}_{\cK} M$ which has a single generator $\xs_0$, concentrated in idempotent 0, and which has $m_{j}=0$ for all $j$. This is not strictly unital. However ${}^{\cK^!} [\cCo]^{\cK}\boxtimes {}_{\cK} M$ has a single generator with vanishing $\delta^1$, and ${}_{\cK} [\cTr]_{\cK^!} \boxtimes ({}^{\cK^!} [\cCo]^{\cK}\boxtimes {}_{\cK} M)$ has a single generator $\xs_0$, but has $m_{2}(1,\xs_0)=\xs_0$ and has $m_2(W^iZ^j,\xs_0)=0$ whenever $i+j>0$. This tensor product is not homotopy equivalent to ${}_{\cK} M$. 
\end{enumerate}
\end{rem}

Finally, we note that the bonsai and regularly $U$-adic conditions are closed under box tensor products, in the following sense:
\begin{lem}
\label{lem:box-tensor-product-bonsai}
 If ${}_{\cK} M^{\cK}$, ${}_{\cK} N^{\cK}$ are both regularly $U$-adic and ${}_{\cK}N^{\cK}$ is bounded, then  ${}_{\cK} M^{\cK}\boxtimes {}_{\cK} N^{\cK}$ is as regularly $U$-adic. If both modules are bonsai, then the box tensor product is as well. A similar statement holds for a pair of bimodules ${}^{\cK^!} M^{\cK}$ and ${}_{\cK} N^{\cK}$.
\end{lem}
\begin{proof} We focus on the statement about $DA$-bimodules, since the statement about $DD$-modules is similar.

 The statement about the tensor product being regularly $U$-adic is straightforward, so we consider the statement about being bonsai. Suppose that $\cT$ is an ordinary $A_\infty$-module structure tree. Since $\cK$ is an associative algebra, we may assume that all interior vertices occur along the path from the module input to the root.
 
 Let $n_0$ be maximum of the bonsai constants of ${}_{\cK} M^{\cK}$ and ${}_{\cK} N^{\cK}$, so that the structure maps vanish for ordinary $A_\infty$-module trees $\cT$ with $\deg(\cT)\ge n_0$. Suppose that $\cT$ is an $A_\infty$-module structure tree with $\deg(\cT)>n_0(n_0+1)$. The structure map $\delta_{\cT}^n$ on the tensor product is constructed by taking certain pairs of structure trees $(\cT_+^L, \cT^R)$, where $\cT_+^L$ is a generalized $A_\infty$-module structure tree, and $\cT^R$ is an ordinary $A_\infty$-module structure tree, and composing the maps $\delta_{\cT_+^L}^{k}$ of $M$ with the map $\delta_{\cT^R}^{n}$ of $N$. We observe that
 \[
 \deg(\cT^L_+)+\deg(\cT^R)=\deg(\cT). 
 \]
 Here $k$ is the number of module vertices of $\cT_+^L$, and $n$ is the number of module vertices of $\cT^R$.

  If $\deg(\cT^R)\ge n_0$, then the structure map for $\cT$ on the tensor product vanishes by bonsai-ness of ${}_{\cK}N^{\cK}$. Therefore, assume $\deg(\cT^R)<n_0$.   This implies that the total number of valence 2 vertices on $\cT^L_+$ is less than $n_0$, since each valence 2 vertex of $\cT^L_+$ has to connect to a vertex of $\cT^R_+$ which has valence greater than 3 (otherwise such a configuration would correspond to a valence 2 vertex on $\cT$, which is excluded because $\cT$ is an ordinary $A_\infty$-module structure tree). Let $V_2$ denote the number of valence 2 vertices along the module path of $\cT^L_+$. As mentioned above,
  \[
  V_2<n_0.
  \]

   We can cut $\cT^L_+$ along these valence 2 vertices to obtain $V_2+1$ ordinary $A_\infty$-module trees $\cT^L_{0},\dots, \cT^L_{V_2}$.  
  
  We observe that
  \[
  \deg(\cT^L_+)=-V_2+\sum_{i=0}^{V_2} \deg(\cT^L_i).
  \]
  We take the inequality
  \[
n_0(n_0+1)<  \deg(\cT)=\deg(\cT_+^L)+\deg(\cT^R)=-V_2+\deg(\cT^R)+\sum_{i=0}^{V_2} \deg(\cT^L_i),
  \]
  and rearrange to obtain
  \[
  \sum_{i=0}^{V_2} \deg(\cT_i^L)=n_0(n_0+1)+V_2-\deg(\cT^R).
  \]
  Since $\deg(\cT^R)<n_0$, we conclude
  \[
  \sum_{i=0}^{V_0} \deg(\cT_i^L)>n_0^2.
  \]
  Therefore
  \[
\max_{i\in \{0,\dots, V_2\}} \deg(\cT^L_i)\ge \frac{n_0^2}{V_2+1}\ge n_0
  \]
 Since $M$ is bonsai, we conclude that one of the $\delta_{\cT_i^L}^{k_i}$ is zero, and therefore so is $\delta_{\cT^L_+}^k$. Hence $\delta^n_{\cT}$ is as well. 
 \end{proof}

\section{Examples}
\label{sec:examples}

In this section, we study several examples. We firstly study several important $A_\infty$-algebra automorphisms of the surgery algebra, and also compute Koszul dual descriptions of these automorphisms. Later we study the bimodules of 2-component L-space links from the perspective of Koszul duality.

\subsection{The elliptic bimodule}

We recall that the \emph{elliptic symmetry} is the algebra morphism $E\colon \cK\to \cK$ given by
\[
E(W)=Z, \, E(Z)=W,\, E(\sigma)=\tau, \, E(\tau)=\sigma,\, E(U)=U,\, E(T^{\pm 1})=T^{\mp 1}.
\]
It is helpful to package the type-$DA$ bimodule ${}_{\cK} [E]^{\cK}$ as a type-$DD$ bimodule by tensoring with ${}^{\cK^!} [\cCo]^{\cK}$. 

The algebra $\cK^!$ also admits a symmetry, which we denote $\scE\colon \cK^!\to \cK^!$ given by
\[
\scE(w)=z,\quad \scE(w)=z, \quad \scE(\theta)=\theta, \quad  \scE(s)=t,\quad  \scE(t)=s, \quad  \scE(\varphi_{\pm})=\varphi_{\mp},
\]
extended multiplicatively.

One computes easily:
\begin{equation}
{}^{\cK^!} [\cCo]^{\cK}\boxtimes {}_{\cK}[E]^{\cK}= {}^{\cK^!} [\scE]_{\cK^!} \boxtimes {}^{\cK^!}[\cCo]^{\cK}=\begin{tikzcd}[row sep=2cm] i_0 \ar[loop above, "\theta|U"] \ar[loop right, "w|Z"] \ar[loop left, " z |W"] \ar[d, " s |\tau", bend left] \ar[d, bend right, " t |\sigma",swap]\\
i_1 \ar[loop left, "\varphi_+|T^{-1}"] \ar[loop right, "\varphi_-|T"] \ar[loop below, "\theta|U"]
\end{tikzcd}
\end{equation}

\subsection{The transformer  bimodule}
\label{sec:transformer}
It is also interesting to study the \emph{transformer bimodule} ${}_{\cK} \cT^{\cK}$, studied in \cite{ZemBordered}*{Section~14}.  We recall that this has underlying vector space $\ve{I}$, and has $\delta_1^1=0$, $\delta_2^1(a,1)=1\otimes a$ and 
\[
\delta_4^1(a,b,\tau,i_1)=i_0\otimes \d_U(a) T\d_T(b) \tau.
\]
All other $\delta_j^1$ vanish. We compute
\[
{}^{\cK^!} [\cCo]^{\cK}\boxtimes {}_{\cK}\cT ^{\cK}=
\begin{tikzcd}[row sep=2cm] i_0 \ar[loop right, "w|W"] \ar[loop left, " z |Z"] \ar[loop above, "\theta|U"] \ar[d, " s |\sigma", bend left] \ar[d, bend right, "\substack{ t |\tau+\\ ( t  \theta \varphi_-)|(T^{-1} \tau)}",swap]\\
i_1 \ar[loop left, "\varphi_+|T"] \ar[loop right, "\varphi_-|T^{-1}"] \ar[loop below, "\theta|U"]
\end{tikzcd}
\]

There is $DA$-bimodule ${}^{\cK^!} [\scT]_{\cK^!}$, as follows. The underlying vector space is $\ve{I}$. We set
\[
\delta_2^1(1,a)=a\otimes 1
\]
and
\[
\delta_3^1(i_0,  t , \varphi_-)= t  \theta \varphi_-\otimes i_1.
\]
We define $\delta_3^1$ to vanish on all other collections of monomials. We define $\delta_j^1=0$ for all other $j$. It is straightforward to verify that ${}^{\cK^!} [\scT]_{\cK^!}$ is a $DA$-bimodule. Furthermore:
\begin{equation}
{}^{\cK^!} [\scT]_{\cK^!}\boxtimes {}^{\cK^!} [\cCo]^{\cK}={}^{\cK^!} [\cCo]^{\cK}\boxtimes {}_{\cK} \cT^{\cK}.
\end{equation}

\subsection{2-component L-space links}

\label{sec:CZZ}

We now discuss the $DA$-bimodules for two component L-space links described in \cite{CZZSatellites} and how they fit into the constructions of the this paper. Recall that a rational homology 3-sphere $Y$ is called an \emph{L-space} if 
\[
\rank_{\Z/2} \left(\widehat{\HF}(Y)\right)= |H_1(Y)|.
\]
We recall the definition of an L-space link, due to Gorsky and N\'{e}methi \cite{GorskyNemethiAlgebraicLinks}.

\begin{define} 
  A link $L\subset S^3$ is called an \emph{L-space link} if $S^3_{\Lambda}(L)$ is an L-space for all integral framings $\Lambda\gg 0$.
\end{define}

 We will write $R_0=\bF[W,Z]$. Therein, we viewed $R_0$ as $\ve{I}_0\cdot \cK\cdot \ve{I}_0$ and studied the bimodules ${}_{\cK} \cX(L)^{R_0}$ for two component L-space links. 

Suppose $L\subset S^3$ is a 2-component L-space link. Write $L=K_1\cup K_2$. Write $\ell=\lk(K_1,K_2)$ for the linking number of the two components. Therein, we described a sequence of chain complexes $\cC_s$, indexed over $s\in \Z+\ell/2$. The complexes $\cC_s$ were free, finitely generated chain complexes over $R_0$, and therefore can be viewed as type-$D$ modules over $R_0$. Additionally, we wrote  $ \cS $ for the knot Floer complex of the component $K_2\subset S^3$, viewed as a type-$D$ module over $R_0$.

In \cite{CZZSatellites}, we constructed a type-$DA$ bimodule ${}_{\cK} \cX(L)^{R_0}$, as follows.  We define
\[
\ve{I}_0\cdot \cX(L):=\bigoplus_{s\in \Z+\ell/2} \cC_s\quad \text{and} \quad \ve{I}_1\cdot \cX(L):=\bF[\Z+\ell/2]\otimes  \cS.
\]
Here, $\bF[\Z+\ell/2]$ denotes the vector space generated by $T^s$ for $s\in \Z+\ell/2$. 

We additionally defined maps
\[
L_W\colon \cC_{s}\to \cC_{s-1},\,\, L_Z\colon \cC_s\to \cC_{s+1},\,\, L_\sigma\colon \cC_s\to  T^s\otimes \cS ,\,\, \text{and}\,\,L_\tau\colon \cC_s\to  T^s\otimes \cS .
\]
We view these as type-$D$ morphisms.

We write
\[
L_{T^{\pm 1}}\colon T^{s}\otimes \cS\to T^{s\pm 1} \cS
\]
for the canonical isomorphism (the identity map on $\cS$).

We defined several homotopies
\[
h_{W,Z}, h_{Z,W}\colon \cC_s\to \cC_s, \quad h_{\sigma,Z},h_{\tau,Z}\colon \cC_s\to T^{s+1}\otimes \cS\]
\[
h_{\sigma,W}, h_{\tau, W} \colon \cC_s\to  T^{s-1}\otimes\cS .
\]
The maps $h_{W,Z}$ and $h_{Z,W}$ satisfy
\[
\d (h_{W,Z})=L_W\circ L_Z+\id\otimes U, \quad \d (h_{Z,W})=L_Z\circ L_W+ \id\otimes U,
\]
where $\d$ denotes the morphism differential where we view each map as a morphism of type-$D$ modules, i.e., $\d(f)=\delta^1\circ f+f\circ \delta^1$. Similarly
\[
\d (h_{\sigma, Z})=L_{\sigma}\circ L_Z+L_T\circ L_\sigma,\quad \d(h_{\sigma,W})=L_{\tau}\circ L_W+(\id\otimes U)\circ L_{T^{-1}}\circ  L_\tau,
\]
and similarly for $h_{\tau,Z}$ and $h_{\tau,W}$.

\begin{rem} The complexes $\cC_s$ and $ \cS $, and the maps $L_W$, $L_Z$ (and so forth) were defined in \cite{CZZSatellites}*{Section~5} using the $H$-function for $L$. Their definition required $L$ to be a 2-component L-space link, but was purely algebraic.
\end{rem}

The $DA$-bimodule structure maps are defined as follows. We set $\delta_1^1$ to be the internal differentials of $\cC_s$ and $ \cS $. We define all structure maps to be $U$-equivariant. We define
\begin{equation}
\delta_2^1(Z^j,-)=\underbrace{(L_Z\circ \cdots \circ L_Z)}_j
\label{eq:delta-2-1-Z^j}
\end{equation}
and similarly for $\delta_2^1(W^j,-)$. We define 
\begin{equation}
\delta_3^1(W^i, Z^j,-)=\sum_{n=1}^{\min(i,j)} L_W^{i-n}\circ  h_{W,Z}\circ L_Z^{j-n}.
\label{eq:delta_3^1-W^i-Z^j}
\end{equation}
We define $\delta_2^1(Z^i, W^j,-)$ by swapping the roles of $W$ and $Z$ in the above formula.

In idempotent 1, we define $\delta_2^1(T^i,-)$ to send $T^s\otimes  \cS $ to $T^{s+i}\otimes  \cS $ via the canonical identification.

We define 
\begin{equation}
\delta_2^1(T^i \sigma,-)=L_T^i\circ L_\sigma, \label{eq:delta-2-1-sigma}
\end{equation}
and similarly for $\delta_2^1(T^i \tau,-)$. Finally, we set
\begin{equation}
\delta_3^1(T^i, T^j \sigma,-)=0,
\quad \delta_3^1(T^i \sigma, Z^j,-)=\sum_{n=1}^j L_T^{i+n-1}\circ h_{\sigma,Z}\circ L_Z^{j-n}
\label{eq:delta-3-1-sigma-1}
\end{equation}
and
\begin{equation}
\delta_3^1(T^i \sigma, W^j,-)=\sum_{n=1}^j U^{n-1}\cdot L_T^{i-n+1}  \circ h_{\sigma, W}\circ W^{j-n}.
\label{eq:delta-3-1-sigma-2}
\end{equation}
The $\delta_3^1$ actions involving $\tau$ are defined symmetrically.

There is are some choices made in the above construction (e.g. $L_W$, $L_Z$ and so forth are only defined up to chain homotopy), however in \cite{CZZSatellites}, the authors prove that the above definitions determine the $DA$-bimodule ${}_{\cK} \cX(L)^{R_0}$ which is well-defined up to homotopy equivalence.

As an example, we illustrate the type-$DA$ bimodule of the positively clasped Whitehead link below.  Here, $\cS$ is generated by a single generator, and we write $T^i$ for $T^i\otimes \cS$:
\[
\begin{tikzcd}[labels=description, column sep=1.4cm]
&&& \xs_0
	\ar[dl, bend left, "W|Z",pos=0.7]
	\ar[dr, "Z|Z", bend left]
	\ar[ddd,bend left=20, "\substack{\sigma|Z\\ +\tau|Z}",pos=.7]
	\ar[d,bend left =80, "{(W,Z)|Z}"]
&&
\\[.5cm]
\cdots
	\ar[r,bend left, "Z|U"]
& \xs_{-2}
	\ar[r, bend left, "Z|U"]
	\ar[l, bend left, "W|1"]
	\ar[dd,"\substack{\sigma|U^2\\ +\tau|1}"]
&
\xs_{-1}
	\ar[ur, bend left, "Z|W"]
	\ar[l, "W|1", bend left]
	\ar[dd, "\substack{\sigma|U \\ +\tau|1 }"]
&
\ys_0
	\ar[u,gray, "W"]
	\ar[d,gray, "Z"]
&
\xs_1
	\ar[r, bend left, "Z|1"]
	\ar[dl, "W|Z", bend left]
	\ar[dd, "\substack{\sigma|1 \\ +\tau|U}"]
&
\xs_2
	\ar[l, bend left, "W|U"]
	\ar[r,bend left, "Z|1"]
	\ar[dd, "\substack{\sigma|1\\+\tau|U^2}"]
&
\cdots
	\ar[l,bend left, "W|U"]
\\[.5cm]
&&&
\zs_0
	\ar[ul, "W|W", bend left]
	\ar[ur, "Z|W", bend left,crossing over,pos = 0.75]
	\ar[d,bend right, "\substack{\sigma|W\\ +\tau|W}"]
	\ar[u,bend left =70, "{(Z,W)|W}"]
&&
\\[1.5cm]
\cdots
	\ar[r, bend left, "T|1"]
&T^{-2}
	\ar[r, bend left, "T|1"]
	\ar[l, bend left, "T^{-1}|1"]
	\ar[loop below,looseness=20, "U|U"]
& T^{-1}
	\ar[r, bend left, "T|1"]
	\ar[l, bend left, "T^{-1}|1"]
	\ar[loop below,looseness=20, "U|U"]
&T^{0}
	\ar[r, "T|1", bend left]
	\ar[l, "T^{-1}|1", bend left]
	\ar[loop below,looseness=20, "U|U"]
&T^{1}
	\ar[r, bend left, "T|1"]
	\ar[l, bend left,"T^{-1}|1"]
	\ar[loop below,looseness=20, "U|U"]
&T^{2}
	\ar[r, bend left, "T|1"]
	\ar[l, bend left,"T^{-1}|1"]
	\ar[loop below,looseness=20, "U|U"]
&\cdots 
	\ar[l, bend left,"T^{-1}|1"]
\end{tikzcd}
\]
(See \cite{CZZSatellites}*{Figure~6.3} for its derivation from the $H$-function of $L$). In the above, each $\xs_i$ and $\ys_i$ denotes a generator of $\cC_s$. An arrow labeled by $a|b$ from $\xs$ to $\ys$ means that the $\ys$ component of $\delta_2^1(a,\xs)$ is $b$.

The  maps $L_W$, $L_Z,$ $h_{W,Z}$, $h_{Z,W}$, $h_{\sigma,Z}$, $h_{\tau,W}$, $L_\sigma$ and $L_\tau$ appear as components of the structure map $\delta^{1,1}$ on the box tensor product 
\[
{}^{\cK^{!}} [\cCo]^{\cK} \boxtimes {}_{\cK} \cX(L)^{R_0}.
\]

From another perspective, if we first pick maps $L_W$, $L_Z$, $h_{W,Z}$, $h_{Z,W}$, $L_\sigma$, $L_\tau$, $h_{\sigma,Z}$, and $h_{\tau,W}$, we may build a type-$DD$ structure ${}^{\cK^!} \cX(L)^{R_0}$ by weighting these maps by $w$, $z$, $zw$, $wz$, $s$, $t$, $zs$ and $tw$, respectively. We add self arrows weighted by $\theta|U$ to each generator. In idempotent 1, we also add right and left moving arrows weighted by $\varphi_+$ and $\varphi_-$ (moving a generator $T^i \otimes \xs$ to $T^{\pm i}\otimes \xs$).

\begin{rem}
\label{rem:h_sigma,W-actions-not-encoded} In our description from \cite{CZZSatellites}, we also considered the maps $h_{\sigma,W}$ and $h_{\tau,Z}$. However the corresponding elements of $\cK^!$ are $ws$ and $zt$, which are zero, so they do not feature in the description here. Note that these maps contain redundant information, since for a 2-component L-space link we can always take
\begin{equation}
h_{\sigma,W}=L_{T^{-1}} \circ h_{\sigma,Z}\circ L_W+L_{T^{-1}}\circ L_{\sigma}\circ  h_{Z,W}.
\label{eq:h-sigma-W-alternate} 
\end{equation}
Indeed, we see that the boundary of the above map is
\[
L_{T^{-1}}\circ L_{\sigma}\circ L_U+L_{T^{-1}} \circ L_T\circ L_\sigma\circ L_W
\]
which we can rewrite as
\[
L_{U T^{-1}}\circ L_\sigma+L_\sigma\circ L_W
\]
(using the fact that $L_U$ is central, $L_{T^{-1}}\circ L_U=L_{UT^{-1}}$ and $L_{T^{-1}}\circ L_T=\id$).
\end{rem}

For the Whitehead link, we obtain the following diagram (omitting self arrows which are weighted by $\theta|U$ which are present on each generator): 
\[
\begin{tikzcd}[labels=description, column sep=1.4cm]
&&& \xs_0
	\ar[dl, bend left, "w|Z",pos=0.7]
	\ar[dr, " z |Z", bend left]
	\ar[ddd,bend left=20, "\substack{s|Z\\ +t|Z}",pos=.7]
	\ar[d,bend left =80, "{zw|Z}"]
&&
\\[.5cm]
\cdots
	\ar[r,bend left, " z |U"]
& \xs_{-2}
	\ar[r, bend left, " z |U"]
	\ar[l, bend left, "w|1"]
	\ar[dd,"\substack{s |U^2\\ + t |1}"]
&
\xs_{-1}
	\ar[ur, bend left, " z |W"]
	\ar[l, "w|1", bend left]
	\ar[dd, "\substack{ s  |U \\ + t |1 }"]
&
\ys_0
	\ar[u,gray, "1|W"]
	\ar[d,gray, "1|Z"]
&
\xs_1
	\ar[r, bend left, " z |1"]
	\ar[dl, "w|Z", bend left]
	\ar[dd, "\substack{ s |1 \\ +  t | U}"]
&
\xs_2
	\ar[l, bend left, "w|U"]
	\ar[r,bend left, " z |1"]
	\ar[dd, "\substack{ s |1\\+ t |U^2}"]
&
\cdots
	\ar[l,bend left, "w|U"]
\\[.5cm]
&&&
\zs_0
	\ar[ul, "w|W", bend left]
	\ar[ur, " z |W", bend left,crossing over,pos = 0.75]
	\ar[d,bend right, "\substack{ s |W\\ + t |W}"]
	\ar[u,bend left =70, "{w z |W}"]
&&
\\[1.5cm]
\cdots
	\ar[r, bend left, "\varphi_+|1"]
&T^{-2}
	\ar[r, bend left, "\varphi_+|1"]
	\ar[l, bend left, "\varphi_-|1"]
& T^{-1}
	\ar[r, bend left, "\varphi_+|1"]
	\ar[l, bend left, "\varphi_-|1"]
&T^{0}
	\ar[r, "\varphi_+|1", bend left]
	\ar[l, "\varphi_-|1", bend left]
&T^{1}
	\ar[r, bend left, "\varphi_+|1"]
	\ar[l, bend left,"\varphi_-|1"]
&T^{2}
	\ar[r, bend left, "\varphi_+|1"]
	\ar[l, bend left,"\varphi_-|1"]
&\cdots  
	\ar[l, bend left,"\varphi_-|1"]
\end{tikzcd}
\]

\begin{rem}
In \cite{CZZSatellites}, we did not consider the notion of a regularly $U$-adic module. In the present article, we view ${}_\cK \cX(L)^{R_0}$ as a regularly $U$-adic $DA$-bimodule by putting each generator in filtration level 0.
\end{rem}

\begin{lem}
\label{lem:CZZ-U-adic-cobonsai}
 If $L$ is a 2-component L-space link, then ${}^{\cK^!} \cX(L)^{R_0}$ is regularly $U$-adic and cobonsai.
\end{lem}
\begin{proof}
The regularly $U$-adic condition follows from the fact that the only $\theta$-weighted differentials are weighted by $\theta|U$, which has total filtration level 0. Since we put all module generators in filtration level 0, we see that $\delta^{1,1}$ maps into nonnegative filtration level, and therefore so do its iterates.

For the cobonsai condition, we observe that there is a natural algebraic grading on the complex ${}^{\cK^!} \cX(L)^{R_0}$. Each of the $\cC_s$ and $\cS$ complexes is a staircase complex, and we put the generators in the lower part of the staircase in algebraic grading 0, and we put the generators in the upper part of the staircase in algebraic grading $1$. By construction, the differential $\delta^{1,1}$ preserves the algebraic grading. Furthermore, in any summand $b_1\otimes \cdots\otimes  b_n\otimes \ys\otimes a_n\otimes \cdots \otimes a_1$ of $\delta^{n,n}_+$, at most one of the $b_i$ can have algebraic grading less than $-1$, and such a $b_i$ must have algebraic grading $-2$. In particular, the cobonsai condition is satisfied.
\end{proof}

\begin{prop}
\label{prop:Koszul-duality-L-space-computation} If $L$ is a 2-component L-space link and ${}^{\cK^!} \cX(L)^{R_0}$ is the type-$DD$ bimodule described above, then the formulas for $\delta_1^1$, $\delta_2^1$ and $\delta_3^1$ from Equations~\eqref{eq:delta-2-1-Z^j}, \eqref{eq:delta_3^1-W^i-Z^j}, \eqref{eq:delta-2-1-sigma},  and~\eqref{eq:delta-3-1-sigma-1} coincide with the structure maps on ${}_{\cK} [\cTr]_{\cK^!} \boxtimes {}^{\cK^!} \cX(L)^{R_0}$. 
For $\delta_3^1(T^i \sigma, W^i U^j,-)$, the action is given as in Equation~\eqref{eq:delta-3-1-sigma-2}, but with $h_{\sigma,W}$ equal to $L_{T^{-1}}\circ h_{\sigma,Z}\circ L_W+L_{T^{-1}} \circ L_\sigma\circ h_{Z,W}$, as in Remark~\ref{rem:h_sigma,W-actions-not-encoded}.  
\end{prop}

We first prove a lemma:

\begin{lem}
\item
\begin{enumerate}
\item Suppose that $a_1,\dots, a_k\in \cK$ and $b_1,\dots, b_n\in \cK^!$ are monomials such that 
\[
m_{k|1|n}(a_k,\dots, a_1,1,b_1,\dots, b_n)\neq 0
\]
and such that $b_i$ is a multiple of $\theta$ if and only if $b_i=\theta$ (e.g. $b_i$ cannot be $\theta wzwz$). Then there is some $i$ so that $b_1,\dots, b_i\neq \theta$, and $b_{i+1},\dots, b_n= \theta$. (The case that $i=0$ or $n$ are allowed).
\item Suppose that $b_1,\dots, b_n\in \cK^!$ are monomials so that $b_i$ is a multiple of $\theta$ if and only if $b_i=\theta$. Then
\[
m_{k|1|n}(a_k,\dots,a_i,\dots, a_1,1,b_1,\dots, b_n)=
m_{k|1|n}(a_k,\dots, Ua_i,\dots, a_1,1,b_1,\dots, b_n,\theta).
\]
\end{enumerate}
\end{lem}
\begin{proof} We begin with the first claim. We recall the notation from Lemma~\ref{lem:inequality-U-weights} that we decompose the bimodule $\scL$ and the homotopy $H$ as in the following diagram
\[
H=\begin{tikzcd}
 \ar[loop left, "H_{\theta,\theta}"] \scL_\theta & \scL_{1} \ar[l, "H_{\theta,1}"] \ar[loop right, "H_{1,1}"] 
 \end{tikzcd}
\]
Here, $\scL_\theta$ denotes the span of elementary tensors $\xs\otimes \ys^*$ where $\ys$ is a multiple of $\theta$, and $\scL_{1}$ is the span of such elementary tensors where $\ys$ is not a multiple of $\theta$.

 Suppose that $\xs\otimes \ys^*\in \scL$. We observe that 
\begin{equation}
(H\circ m_{0|1|1}(-,\theta)\circ H)(\xs\otimes \ys^*)=0\label{eq:HthetaH=0}
\end{equation}
unless $\xs\otimes \ys^*$ is $U^i\otimes 1^*$ (in either idempotent). 

We constructed the structure maps on $[\cTr]$ using the homological perturbation lemma. Consider a sequence of applications of $m_{0|1|1}(-,b_i)$, $m_{1|1|0|}(a_i,-)$, $\a_{0|1|0}^{\mu_1}$ and $H$ contributing to the homological perturbation lemma. Equation~\eqref{eq:HthetaH=0} implies that if any application of $\theta$ is not at the end of the sequence, it must be applied to $U^i\otimes 1^*$. 

Note that $\theta$ cannot be the first action applied, because $m_{0|1|1}(-,\theta)\circ I=0$. 

We now claim that if a $\theta$ is an input, (and there are no inputs of the form $b\theta$ for a non-trivial $b$), then the last input must be $\theta$. To see this, consider a sequence of the form
\[
\Pi\circ \a_i \circ H\circ  \cdots \circ \a_1\circ H\circ m_{0|1|1}(-,\theta)\circ H
\]
where $\a_1,\dots, \a_j$ is a sequence of $m_{0|1|1}(-,b_i)$, $m_{1|1|0}(a_i,-)$ and $\a_{0|1|0}^{\mu_1}$ terms (where no $b_i$ is $\theta$).
Observe that $H\circ m_{0|1|1}(-,\theta)\circ H$ has image in $\scL_\theta$.
Furthermore, $H\circ m_{0|1|1}(-,\theta)\circ H$ has image in $\scL_\theta$, each $\a_i$  preserves $\scL_\theta$ and $H$ preserves the subspace $\scL_{\theta}$. However this implies that $\a_i\circ H\circ \cdots \circ \a_1 \circ H\circ m_{0|1|1}(-,\theta)\circ H$ has image in $\scL_\theta$, on which $\Pi$ vanishes. Therefore if a $\theta$ appears as an input, the last input must be $\theta$.

Next, we claim that if $\theta$ appears as an input, then all of the final inputs must be $\theta$. To see this, we consider the possibility of a gap in the $\theta$-inputs, i.e. a composition of the form
\[
\Pi\circ (m_{0|1|1}(-,\theta)\circ  H)^N\circ  \a_i \circ H\circ  \cdots  \circ \a_1\circ H\circ m_{0|1|1}(-,\theta)\circ H
\]
for some $N>0$. This will only be non-zero if the output of $\a_i \circ H\circ  \cdots  \circ \a_1\circ H\circ m_{0|1|1}(-,\theta)\circ H$ is $U^N\otimes 1^*$. Using similar logic to the above argument, this can only happen if $i$ is zero, since otherwise $\a_i \circ H\circ  \cdots \circ H \circ \a_1\circ H\circ m_{0|1|1}(-,\theta)\circ H$ will have image in $\scL_\theta$. 
This concludes the proof of the first claim.

We move on to the second claim. We observe that the actions $m_{0|1|1}$, $m_{1|1|0}$ and $\a_{0|1|0}^{\mu_1}$ on $\scL$ are $U$-equivariant. The homotopy $H$ is $U$-equivariant, except when applied to the elements labeled $1|1^*$ (in both idempotents). If the element $1|1^*$ appears while applying homological perturbation, then it must be the last element which appears (or the first), since $H$ vanishes on it and it is not in the image of $H$. Therefore, if we have a sequence of maps $\a_{n+k}\circ H\circ \cdots \circ H\circ \a_1$ which contributes to $m_{k|1|n}(a_k,\dots, a_1,1,b_1,\dots, b_n)$ (where each $\a_j$ is some $m_{1|1|0}(a_i,-)$, $m_{0|1|1}(-,b_i)$ or $\a_{0|1|0}^{\mu_1}$, compatible with their given ordering as inputs to the structure map of $[\cTr]$), then we obtain a corresponding sequence of maps for 
\[
m_{k|1|n+1}(a_k,\dots, Ua_i,\dots, a_1,1,b_1,\dots, b_n,\theta)
\]
by multiplying the $\a_j$ corresponding to $a_i$  by $U$, and then appending $m_{0|1|1}(-,\theta)\circ H$ to the end. Similarly, removing $m_{0|1|1}(-,\theta)\circ H$ from the end and lowering the power of $U$ sends a sequence for $(a_k,\dots, Ua_i,\dots, a_1,1,b_1,\dots, b_n,\theta)$ to a sequence of maps for $(a_k,\dots, a_i,\dots,a_1,1,b_1,\dots, b_n)$.
\end{proof}

As an immediate consequence, we obtain the following:

\begin{cor}
\label{cor:U-equivariant} Suppose that ${}^{\cK^!} X^{R_0}$ is a type-$DD$ module such that each generator has a $\theta|U$ weighted self-arrow, and there are no other differentials which are weighted by a multiple of $\theta$. Then the structure maps on ${}_{\cK}[\cTr]_{\cK^!}\boxtimes {}^{\cK^!} X^{R_0}$ are $U$-equivariant, in the sense that
\[
\delta_{n+1}^1(a_1,\dots,Ua_i,\dots,  a_n,\xs)=(\id\otimes U)\circ \delta_{n+1}^1(a_1,\dots,a_i,\dots a_n, \xs).
\]
 The same holds for $DD$-bimodules of the form ${}^{\cK^!} X^{\cK}$. 
\end{cor}

\begin{proof}[Proof of Proposition~\ref{prop:Koszul-duality-L-space-computation}]
By Corollary~\ref{cor:U-equivariant}, the structure maps on ${}_{\cK} [\cTr]^{\cK^!}\boxtimes {}^{\cK^!} \cX(L)^{R_0}$ are $U$-equivariant.

We now consider the stated formulas. Consider first $\delta_2^1(Z^i,-)$. For this claim, it suffices to observe that the only sequence $b_1,\dots, b_n\in \cK^!$ for which the structure map on ${}_{\cK} [\cTr]_{\cK^!}$ satisfies $m_{2|1|n}(Z^i, 1, b_1,\dots, b_n)\neq 0$ is the sequence 
\[
(b_1,\dots, b_n)=(\underbrace{z,\dots, z}_i).
\]
This follows from the description of $m_{i|1|j}$ given in Lemma~\ref{lem:Tr-construction} using the homological perturbation lemma. The above implies that
\[
\delta_2^1(Z^i,-)=\underbrace{L_Z\circ \cdots \circ L_Z}_i.
\]
The same idea holds for $\delta_2^1(W^i,-)$.

 It is straightforward to see that there are no sequences with $m_{2|1|N}(Z^i , Z^j ,1,b_1,\dots, b_N)\neq 0$, so $\delta_3^1(Z^i, Z^j,-)=0$.

 Next, we consider $\delta_3^1(W^i , Z^j, -)$. We observe by straightforward analysis that the only sequence $b_1,\dots, b_N$ for which $m_{2|1|N}(W^i , Z^j ,1,b_1,\dots, b_N)\neq 0$ and all $b_k\in \{w,z,wz,zw,\theta\}$ are the sequences
\[
(b_1,\dots, b_N)=(\underbrace{z,\dots, z}_{j-k}, zw,\underbrace{w,\dots, w}_{i-k}, \underbrace{\theta,\dots, \theta}_{k-1})
\]
for $0< k\le  \min(i,j)$. This corresponds exactly to the action of $\delta_3^1$ given in Equation~\eqref{eq:delta_3^1-W^i-Z^j}. 

The correspondence between the remaining actions in Equations~\eqref{eq:delta_3^1-W^i-Z^j} and ~\eqref{eq:delta-3-1-sigma-1} follows from a straightforward extension of the above ideas, and we leave the details to the reader.

We now consider the actions $\delta_3^1(T^i \sigma, W^j,-)$. We observe firstly that $\delta_3^1(\sigma,W,-)$ is $L_{T^{-1}} \circ h_{\sigma,Z}\circ L_W+L_{T^{-1}}\circ L_{\sigma}\circ  h_{Z,W}$, since the only $b_1,\dots, b_n\in \cK^!$ for which $m_{2|1|n}(\sigma,W,1,b_1,\dots, b_n)\neq 0$ are $(wz,s,\varphi_-)$ and $(w,wz, \varphi_-)$. 

More generally, one computes easily that the only sequences $b_1,\dots, b_N\in \cK^!$ for which $m_{2|1|N}(T^i\sigma, W^j,1,b_1,\dots, b_N)\neq 1$ are the sequences
\[
\begin{split}
(b_1,\dots, b_N)=&(\underbrace{w,\dots, w}_{j-k-1},w,zs,\underbrace{\varphi_{\sgn(i-k-1)},\dots, \varphi_{\sgn(i-k-1)}}_{|i-k-1|}, \underbrace{\theta,\dots, \theta}_k )\\
(b_1,\dots, b_N)=&(\underbrace{w,\dots, w}_{j-k-1},wz,s,\underbrace{\varphi_{\sgn(i-k-1)},\dots, \varphi_{\sgn(i-k-1)}}_{|i-k-1|}, \underbrace{\theta,\dots, \theta}_k).
\end{split}
\]
for $k\in \{0,\dots, j-1\}$. Unpacking, this gives the stated formula for $\delta_3^1(T^i\sigma, W^j,-)$ in terms of $\delta_3^1(\sigma,W,-)$. 
\end{proof}


\bibliographystyle{custom}
\def\MR#1{}
\bibliography{biblio}

\begin{bibdiv}
\begin{biblist}

\bib{CZZSatellites}{unpublished}{
      author={Chen, Daren},
      author={Zemke, Ian},
      author={Zhou, Hugo},
       title={L-space satellite operations and knot {F}loer homology},
        date={2024},
        note={e-print, arXiv 2412.05755 [math.GT]},
}

\bib{CZZCode}{misc}{
      author={Chen, Daren},
      author={Zemke, Ian},
      author={Zhou, Hugo},
       title={satellites},
   publisher={GitHub},
        date={2024},
        note={\url{https://github.com/ian-zemke/satellites}, commit 0aac7b6},
}

\bib{GorskyNemethiAlgebraicLinks}{article}{
      author={Gorsky, Eugene},
      author={N\'{e}methi, Andr\'{a}s},
       title={Links of plane curve singularities are {$L$}-space links},
        date={2016},
        ISSN={1472-2747},
     journal={Algebr. Geom. Topol.},
      volume={16},
      number={4},
       pages={1905\ndash 1912},
         url={https://doi.org/10.2140/agt.2016.16.1905},
      review={\MR{3546454}},
}

\bib{LOTMorphisms}{article}{
      author={Lipshitz, Robert},
      author={Ozsv\'ath, Peter~S.},
      author={Thurston, Dylan~P.},
       title={Heegaard {F}loer homology as morphism spaces},
        date={2011},
        ISSN={1663-487X,1664-073X},
     journal={Quantum Topol.},
      volume={2},
      number={4},
       pages={381\ndash 449},
         url={https://doi.org/10.4171/qt/25},
      review={\MR{2844535}},
}

\bib{LOTBimodules}{article}{
      author={Lipshitz, Robert},
      author={Ozsv\'{a}th, Peter~S.},
      author={Thurston, Dylan~P.},
       title={Bimodules in bordered {H}eegaard {F}loer homology},
        date={2015},
        ISSN={1465-3060},
     journal={Geom. Topol.},
      volume={19},
      number={2},
       pages={525\ndash 724},
         url={https://doi.org/10.2140/gt.2015.19.525},
      review={\MR{3336273}},
}

\bib{LOTBordered}{article}{
      author={Lipshitz, Robert},
      author={Ozsvath, Peter~S.},
      author={Thurston, Dylan~P.},
       title={Bordered {H}eegaard {F}loer homology},
        date={2018},
        ISSN={0065-9266},
     journal={Mem. Amer. Math. Soc.},
      volume={254},
      number={1216},
       pages={viii+279},
         url={https://doi.org/10.1090/memo/1216},
      review={\MR{3827056}},
}

\bib{LOTDiagonals}{unpublished}{
      author={Lipshitz, Robert},
      author={Ozsv\'{a}th, Peter},
      author={Thurston, Dylan},
       title={Diagonals and {A}-infinity tensor products},
        date={2020},
        note={e-print, {arXiv:2009.05222} [math.GT]},
}

\bib{MOIntegerSurgery}{unpublished}{
      author={Manolescu, Ciprian},
      author={Ozsv{\'a}th, Peter~S.},
       title={Heegaard {F}loer homology and integer surgeries on links},
        date={2010},
        note={e-print, {arXiv:1011.1317} [math.GT]},
}

\bib{MOT_Grid_Diagrams}{article}{
      author={Manolescu, Ciprian},
      author={Ozsv\'ath, Peter~S.},
      author={Thurston, Dylan~P.},
       title={Grid diagrams and {H}eegaard {F}loer invariants},
        date={2025},
        ISSN={0003-486X,1939-8980},
     journal={Ann. of Math. (2)},
      volume={201},
      number={1},
       pages={1\ndash 78},
         url={https://doi.org/10.4007/annals.2025.201.1.1},
      review={\MR{4848040}},
}

\bib{OSDisks}{article}{
      author={Ozsv\'ath, Peter},
      author={Szab\'o, Zolt\'an},
       title={Holomorphic disks and topological invariants for closed
  three-manifolds},
        date={2004},
     journal={Ann. of Math. (2)},
      volume={159},
      number={3},
       pages={1027\ndash 1158},
}

\bib{OSIntegerSurgeries}{article}{
      author={Ozsv\'{a}th, Peter~S.},
      author={Szab\'{o}, Zolt\'{a}n},
       title={Knot {F}loer homology and integer surgeries},
        date={2008},
        ISSN={1472-2747},
     journal={Algebr. Geom. Topol.},
      volume={8},
      number={1},
       pages={101\ndash 153},
         url={https://doi.org/10.2140/agt.2008.8.101},
      review={\MR{2377279}},
}

\bib{OSRationalSurgeries}{article}{
      author={Ozsv\'ath, Peter},
      author={Szab\'o, Zolt\'an},
       title={Knot {F}loer homology and rational surgeries},
        date={2011},
     journal={Algebr. Geom. Topol.},
      volume={11},
      number={1},
       pages={1\ndash 68},
}

\bib{Positelski_Memoirs}{article}{
      author={Positselski, Leonid},
       title={Two kinds of derived categories, {K}oszul duality, and
  comodule-contramodule correspondence},
        date={2011},
        ISSN={0065-9266,1947-6221},
     journal={Mem. Amer. Math. Soc.},
      volume={212},
      number={996},
       pages={vi+133},
         url={https://doi.org/10.1090/S0065-9266-2010-00631-8},
      review={\MR{2830562}},
}

\bib{PP_Quadratic}{book}{
      author={Polishchuk, Alexander},
      author={Positselski, Leonid},
       title={Quadratic algebras},
      series={University Lecture Series},
   publisher={American Mathematical Society, Providence, RI},
        date={2005},
      volume={37},
        ISBN={0-8218-3834-2},
         url={https://doi.org/10.1090/ulect/037},
      review={\MR{2177131}},
}

\bib{SeidelFukaya}{book}{
      author={Seidel, Paul},
       title={Fukaya categories and {P}icard-{L}efschetz theory},
      series={Zurich Lectures in Advanced Mathematics},
   publisher={European Mathematical Society (EMS), Z\"{u}rich},
        date={2008},
        ISBN={978-3-03719-063-0},
         url={https://doi.org/10.4171/063},
      review={\MR{2441780}},
}

\bib{ZemBordered}{unpublished}{
      author={Zemke, Ian},
       title={Bordered manifolds with torus boundary the link surgery formula},
        date={2021},
        note={e-print, arXiv 2109.11520 [math.GT]},
}

\bib{ZemExact}{unpublished}{
      author={Zemke, Ian},
       title={A general {H}eegaard {F}loer surgery formula},
        date={2023},
        note={e-print, arXiv 2308.15658 [math.GT]},
}

\end{biblist}
\end{bibdiv}

\end{document}